\documentclass[10pt,reqno]{amsart}
\usepackage{float}
\usepackage[numbers, sort&compress]{natbib}
\usepackage{graphicx}
\usepackage{graphics}
\usepackage{subfigure}
\usepackage{mathrsfs}
\usepackage{xypic,xcolor}
\usepackage{tikz}
\usepackage{bm}
\usepackage{mathtools}
\newtheorem{theorem}{Theorem}[section]
\newtheorem{lemma}[theorem]{Lemma}

\theoremstyle{definition}

\newtheorem{remark}[theorem]{Remark}
\newtheorem{corollary}[theorem]{Corollary}

\numberwithin{equation}{section}

\newtheorem{thmy}{Theorem}
\newenvironment{thmx}{\stepcounter{thm}\begin{thmy}}{\end{thmy}}

\begin{document}
\title[]{Fractal geometry of continued fractions with large coefficients and dimension drop problems}
\author{Lulu Fang, Carlos Gustavo Moreira, Yiwei Zhang}
%\address{Lulu~Fang, School of Science, Nanjing University of Science and Technology, Nanjing, 210094, China}
%\email{fanglulu1230@163.com}
%\address{Carlos Gustavo~Moreira, Institute of Pure and Applied Mathematics (IMPA), Rio de Janeiro, RJ 22.460, Brazil and Department of Mathematics \& SUSTech International Center for Mathematics, Southern University of Science and Technology, Shenzhen, Guangdong 518055, China}
%\email{gugu@impa.br}
%\address{Yiwei~Zhang, Department of Mathematics \& SUSTech International Center for Mathematics,
%Southern University of Science and Technology, Shenzhen, Guangdong 518055, China}
%\email{zhangyw@sustech.edu.cn}

%
%
\thanks{L. Fang is supported by NSFC No. 11801591, and Natural Science Foundation of Jiangsu Province No. BK20231452. C. G. Moreira is partially supported by CNPq and FRPERJ. Y. Zhang is partially supported by NSFC Nos. 12161141002, 12271432}. %and Guangdong Basic and Applied Basic Research Foundation No. 2024A1515010974.

\begin{abstract}
In 1928, Jarn\'{\i}k \cite{Jar} obtained that the set of continued fractions with bounded coefficients has Hausdorff dimension one.
Good \cite{Goo} observed a dimension drop phenomenon by proving that the Hausdorff dimension of the set of continued fractions whose coefficients tend to infinity is one-half. For the set of continued fractions whose coefficients tend to infinity rapidly, Luczak \cite{Luc} and Feng et al. \cite{FWLT} showed that its Hausdorff dimension decreases even further. Recently, Liao and Rams \cite{LR16} also observed an analogous dimension drop phenomenon when they studied the subexponential growth rate of the sum of coefficients.

In this paper, we consolidate and considerably extend the studies of the abovementioned problem into a general dimension drop problem on the distribution of continued fractions with large coefficients. For each $x\in[0,1)\backslash\mathbb{Q}$, let $\{a_{n}(x)\}$ be the sequence of the coefficients of the continued fraction expansion of $x$. Given a strictly increasing sequence of positive integers $\{n_{k}\}$, and two other sequences of positive real numbers $\{s_{k}\}$ and $\{t_{k}\}$ with $s_{k},t_{k}\to\infty$ as $k\to\infty$, we study the set
\begin{equation*}
E(\{n_k\},\{s_k\}, \{t_k\}):= E_{S}(\{n_k\}) \cap E_{L}(\{n_k\},\{s_k\}, \{t_k\}),
\end{equation*}
where
\begin{align*}
E_{S}(\{n_k\}):=\big\{x\in [0,1)\backslash\mathbb{Q}: \{a_j(x)\}_{j\neq n_k (k \in\mathbb N)}\ \text{is bounded}\big\},
\end{align*}
and
\begin{align*}
E_{L}(\{n_k\},\{s_k\}, \{t_k\}):=\big\{x\in [0,1)\backslash\mathbb{Q}: s_{k} < a_{n_k}(x) \leq s_{k}+ t_{k} \ \text{for all large $k$}\big\}.
\end{align*}
We introduce quantities $\alpha:=\lim_{k\to\infty}\frac{1}{n_{k}}\sum\limits_{j=1}^{k}\log s_{j}$ and $\beta:=\lim_{k\to\infty}\frac{1}{n_{k}}\log s_{k}$ to describe the distribution of large coefficients. By using thermodynamic formalism tools, we respectively formulate the Hausdorff dimensions of the sets $E(\{n_k\},\{s_k\}, \{t_k\})$ in terms of the solutions $\theta(\alpha,\beta)$ of the Diophantine pressure equations, considering the cases where $\alpha\in (0,\infty),\alpha=0$ and $\alpha=\infty$. Our results enable us to compare the Hausdorff dimensions of the sets $E_{S}, E$ and $E_{L}$ based on different values of $\alpha$. Thus, it explains the mechanism of the dimension drop phenomena.

As applications, we use a different approach to reprove a result of Wang and Wu on the dimensions of the Borel-Bernstein sets \cite{WW}, fulfil the dimension gap proposed by Liao and Rams \cite{LR16}, and establish several new results concerning the dimension theory of liminf and limsup sets related to the maximum of coefficients.
\end{abstract}

\maketitle

\tableofcontents

\section{Introduction}
\subsection{Backgrounds}
Continued fractions are a central mathematical tool used in many areas of mathematics. As a well-known representation of numbers, continued fractions not only play an important role in the arithmetic and geometric nature of real numbers, but are also closely connected with dynamical systems, probability theory and fractal geometry, see, for examples, \cite{EW,Hen06,IK,Khi,LMMR,Sch,Yoc} and the references therein.

One of the most successful applications of continued fractions is in the theory of Diophantine approximation. Dirichlet's theorem (see \cite[page 1]{Sch}) implies that for any irrational number $x\in [0,1]$, there exist infinitely many rational numbers $p/q$ such that
\begin{equation}\label{DAineq}
\left|x-\frac{p}{q} \right| < \frac{1}{q^2}.
\end{equation}
Continued fractions provide a simple mechanism for generating these rational approximations. Moreover, irrational numbers whose coefficients include large numbers have good rational approximations, see \cite[page 34]{Khi}.

On the other hand, the right-hand side of \eqref{DAineq} cannot be replaced by $\varepsilon/q^2$ for any $0<\varepsilon<1$. This is due to the existence of the so-called \emph{badly approximable numbers}, namely an irrational number $x$ satisfying there exists a positive constant $c_x$ such that $\left|x-p/q\right| \geq c_x/q^{2}$ for all rational numbers $p/q$.  Given $x\in \mathbb{I}:=[0,1)\backslash \mathbb{Q}$, let $\{a_n(x)\}$ be the sequence of the coefficients of the continued fraction expansion of $x$. Badly approximable numbers have a beautiful characterisation using continued fractions: $x\in \mathbb I$ is badly approximable if and only if the sequence $\{a_n(x)\}$ of its coefficients is bounded, see \cite[page 22]{Sch}. Therefore, the small or large value of the coefficients of an irrational number reveals how well rational numbers approximate it.

A natural question is: what is the size (e.g., Lebesgue measure, Hausdorff dimension) of sets of continued fractions with small or large coefficients? The Borel-Bernstein theorem (see \cite[page 63]{Khi}) states that, for any $\psi: \mathbb{R}^+ \to \mathbb{R}^+$, the set
\begin{equation}\label{Apsi}
A(\psi):=\left\{x\in \mathbb{I}:a_n(x) \geq \psi(n)\;\text{for infinitely many\;$n \in \mathbb N$}\right\}
\end{equation}
has full or null Lebesgue measure according to $\sum_{n \geq 1} 1/\psi(n)$ diverges or converges. The original proof of Borel \cite{Bor09} in 1909 was incomplete, as discussed by Bernstein \cite{Ber}. Further details were provided in a later paper by Borel \cite{Bor} in 1912. As a result of the Borel-Bernstein theorem, the set of irrational numbers with bounded coefficients has Lebesgue measure zero. Equivalently, $\limsup_{n\to \infty}a_n(x)=\infty$ for Lebesgue almost every $x\in \mathbb{I}$.

Much attention has been paid to continued fractions with small or large coefficients in the literature, and there are two main research directions.

On the one hand, for continued fractions with small coefficients, Jarn\'{i}k \cite{Jar} in 1928 initiated the study of the set $F_N$ of continued fractions whose coefficients do not exceed $N$, and the set $F$ of continued fractions whose coefficients are bounded. He showed that $\dim_{\rm H}F_2>1/4$, and for every $N>8$,
\begin{equation}\label{equ:dimFN}
1-\frac{4}{N\log 2}\leq \dim_{\rm H}F_N \leq 1-\frac{1}{8N\log N}.
\end{equation}
Taking $N\to\infty$ on both sides of \eqref{equ:dimFN}, it directly implies that
\begin{equation}\label{equ:F}
\dim_{H}F=1.
\end{equation}
Jarn\'{i}k's work motivated a number of subsequent improvements. In 1941, Good \cite{Goo} gave a formula for the Hausdorff dimension of $F_N$ in terms of the Euler polynomials, and in particular proved that the dimension of $F_2$ is between $0.5306$ and $0.5320$. In 1977, Cusick \cite{Cus77} used some of Good's results to relate the dimension of $F_N$ with the exponent of convergence of some Dirichlet series. In 1992, Hensley \cite{Hen92} leveraged functional analytic techniques (e.g., linear operator theory) to obtain the asymptotic estimate of the Hausdorff dimension of $F_N$.

Over the last thirty years, with the development of smooth ergodic theory, $F_N$ can be viewed as the attractor of some iterated function systems consisting of conformal maps. This allows us to take advantage of the method of thermodynamic formalism developed in \cite{DFSU,JP,MU96}. For example, the Hausdorff dimension of $F_N$ is the unique zero point of the corresponding Diophantine pressure function, and the value of the Hausdorff dimension of $F_2$ is rigorously approximated with an accuracy of over $100$ decimal places, see the works of Jenkinson-Pollicott \cite{JP18}, and Pollicott-Vytnova \cite{PV}.

The dimension estimates for $F_N$ frequently appear in Diophantine approximation. In particular, such estimations have been employed in the study of Markoff and Lagrange spectra \cite{Bum,CF,LMMR,MMPV}, and in a contribution to the Zaremba conjecture \cite{BK,Huang,JP20,PV}. For examples, Matheus and the second named author\cite[Theorem 5.3]{MM} showed that the dimension of $F_2$ is a lower bound of the dimension of the difference of Markov and Lagrange spectra; Bourgain and Kontorovich \cite[Theorem 1.2]{BK} proved a density one version of the Zaremba conjecture, and their arguments relied on the fact that the dimension of $F_{50}$ is sufficiently close to $1$ (more precisely, $\dim_{\rm H}F_{50} >\frac{307}{312}$). We refer to a survey of Shallit \cite{Sha} for more information on continued fractions with small coefficients.

%More recently, Das et al. \cite{DFSU} improved Hensley's asymptotic formula of the Hausdorff dimension of $F_N$.
%We also remark that dimension estimates for the sets of continued fractions with small coefficients frequently appear in the Diophantine approximation problems.

On the other hand, for continued fractions with large coefficients, the study of the Hausdorff dimension of such continued fractions was initiated by Good \cite{Goo} and he obtained a number of results. The principal one (i.e., Theorem 1 of \cite{Goo}) states that
\begin{equation}\label{equ:1/2}
\dim_{\rm H}\left\{x\in \mathbb I: \lim_{n\to \infty}a_n(x)=\infty \right\} =\frac{1}{2}.
\end{equation}
Comparing \eqref{equ:F} and \eqref{equ:1/2}, it follows that Hausdorff dimension drops directly from $1$ to $1/2$ as the coefficients grow form small to large. There are also a number of refinements and extensions of Good's result.
%We will discuss the dimension drop phenomenon in detail in the next paragraph.
Consider the dual of the set $A(\psi)$ in \eqref{Apsi}:
\begin{equation}\label{Ahatpsi}
\widehat{A}(\psi):=\left\{x\in \mathbb{I}:a_n(x) \geq \psi(n)\;\text{for sufficiently large\;$n \in \mathbb N$}\right\},
\end{equation}
where $\psi: \mathbb{R}^+\to \mathbb{R}^+$ is a function satisfying $\psi(n) \to \infty$ as $n \to \infty$. It is known that the Hausdorff dimension of $\widehat{A}(\psi)$ decreases as $\psi$ grows more rapidly. For example, in 1970, Hirst \cite{Hir} proved that the Hausdorff dimension of $\widehat{A}(\psi)$ is $1/2$, when $\psi$ satisfies the property that: given every $c > 1$, there exists $b > 1$, such that $\psi(n) \leq b^{c^n}$ for sufficiently large $n \in \mathbb N$. In other words, the Hausdorff dimension always stays at $1/2$, provided that the coefficients grow slower than any doubly exponential function.

For any $b,c>1$, let $A(b,c)$ and $\widehat{A}(b,c)$ be the sets $A(\psi)$ and $\widehat{A}(\psi)$ with $\psi(n)=b^{c^n}$ respectively.
Then $\widehat{A}(b,c)$ is a subset of $A(b,c)$. Hirst \cite{Hir} in 1970 further studied the Hausdorff dimension of $\widehat{A}(b,c)$, but only obtained a lower bound: $\dim_{\rm H}\widehat{A}(b,c) \geq 1/(2c)$. For the upper bound of $\dim_{\rm H}A(b,c)$, in 1992, Moorthy \cite{Moo} showed that $\dim_{\rm H}A(b,c) \leq 2/(c+1)$. However, Moorthy's upper bound does not coincide with Hirst's lower bound. The exact Hausdorff dimensions of $A(b,c)$ and $\widehat{A}(b,c)$ were later obtained by {\L}uczak \cite{Luc}, and Feng et al. \cite{FWLT}.
They proved that
\begin{equation}\label{LFWLTbc}
\dim_{\rm H}\widehat{A}(b,c)= \dim_{\rm H}A(b,c)=\frac{1}{c+1}.
\end{equation}
This implies that when the coefficients increase sufficiently rapidly, the dimension further decreases, and possibly decreases to zero. For example, Cusick \cite{Cus} proved that $A(\psi)$ with $\psi(n)=2^{2^{2^n}}$ has Hausdorff dimension zero. In 2009, Fan et al. \cite{FLWW09} generalised the results in \eqref{LFWLTbc} by computing the Hausdorff dimension of the set of $x\in \mathbb I$ for which $a_n(x)$ belongs to $[s_n, Ns_n)$ for all $n \in \mathbb N$, where $N \geq 2$ is an integer and $\{s_n\}$ is a sequence of positive numbers tending to infinity. For more results of
continued fractions with large coefficients, see \cite{Hir73,IJ,JR,LR16,LR22,Ram,Tak,WW} and the references therein.

As we mentioned in \eqref{equ:F},\eqref{equ:1/2} and \eqref{LFWLTbc}, from the results of Jarn\'{i}k-Good-Cusick, there are two kinds of dimension drops of continued fractions: from $1$ to $1/2$ and from $1/2$ to $0$. In the case of the dimension drops from $1$ to $1/2$, it is possible to fill the dimension gaps by considering the set of continued fractions whose coefficients are large along a sparse sequence of positive integers but small at other positions. This idea was used by Good \cite{Goo} to study $A(\psi)$ with $\psi(n)=B^n$ and $1<B<\infty$, denoting it by $A(B)$ if no confusion arises. He obtained the lower and the upper bounds for the Hausdorff dimension of $A(B)$:
\begin{itemize}
  \item The lower bound: if $1/2<s\leq 1$ and $B^{4s} < \zeta(2s)-1-4^s$, then $$\dim_{\rm H}A(B) \geq s;$$
  \item The upper bound: if $s>1/2$ and $B^s > \zeta(2s)$, then $$\dim_{\rm H}A(B) \leq s.$$
\end{itemize}
Here $\zeta(\cdot)$ denotes the Riemann zeta function. The lower bound follows from the Hausdorff dimension of the subset of $A(B)$: the set of $x\in \mathbb I$ for which $1\leq a_j(x) \leq N$ for all $j\neq n_k$ and $a_{n_k}(x)= \lfloor B^{n_k}\rfloor +1$ for all $k\in \mathbb N$, where $N$ is an integer and $\{n_k\}$ is some increasing sequence of positive integers, see \cite[pages 214-215]{Goo} for more details. Note that the coefficient at the position $n_k$ has only one value, and this subset of $A(B)$ is too small such that its dimension cannot reach the optimal lower bound of the Hausdorff dimension of $A(B)$. In 2008, Wang and Wu \cite{WW} modified Good's subset of $A(B)$ by replacing $a_{n_k}(x)= \lfloor B^{n_k}\rfloor +1$ with $B^{n_k} <a_{n_k}(x) \leq 2 B^{n_k}$. They obtained the exact Hausdorff dimension of $A(B)$ provided that $\{n_k\}$ is a sparse sequence of positive integers such that $\frac{n_{k+1}}{k+1}\geq n_1+\cdots+n_k$ for all $k\in \mathbb N$. See \cite[page 1326]{WW} for more details. They also showed that the function $B\mapsto\dim_{\rm H}A(B)$ is continuous on $(1,\infty)$. Moreover,
\begin{equation*}
\dim_{\rm H}A(B) \to 1\ \ \text{as}\ \ B \to 1\ \ \ \ \text{and}\ \ \ \ \dim_{\rm H}A(B) \to \frac{1}{2}\ \ \text{as}\ \ B \to \infty.
\end{equation*}
In this sense, Wang and Wu's result is one way to fill the dimension gap from $1$ to $1/2$. For the dimension drops from $1/2$ to $0$, we remark that the results of $A(b,c)$ and $\widehat{A}(b,c)$ in \eqref{LFWLTbc} fill such dimension gaps.

Other than Jarn\'{i}k-Good-Cusick dimension drop, there are other dimension drop phenomena appearing in the study of continued fractions. For any $x\in \mathbb I$ and $n \in \mathbb N$, let $S_n(x):=\sum^n_{k=1}a_k(x)$ be the partial sum of the coefficients of $x$. In 1935, Khintchine \cite{Khi35} showed that $\lim_{n \to \infty} S_n(x)/n =\infty$ for Lebesgue almost every $x\in\mathbb I$. Subsequently, many authors investigated the Hausdorff dimension of the set
\begin{equation*}
S(\varphi):=\left\{x\in \mathbb I: \lim_{n \to \infty}\frac{S_n(x)}{\varphi(n)}=1 \right\},
\end{equation*}
where $\varphi: \mathbb{R}^+ \to \mathbb R^+$ is an increasing function with $\varphi(n) \to \infty$ as $n \to \infty$. The linear case $\varphi(n)=\gamma n$ with $\gamma \in [1,\infty)$ was studied by Iommi and Jordan \cite{IJ} who proved the function $\gamma \mapsto \dim_{\rm H}S(\varphi)$ is real analytic, increasing from $0$ to $1$, and tends to $1$ as $\gamma$ tends to infinity. For the polynomial case $\varphi(n)=n^p$ with $p \in (1,\infty)$, Wu and Xu \cite{WX} showed that the set $S(\varphi)$ has full Hausdorff dimension. In 2016, Liao and Rams \cite{LR16} found a dimension gap of $S_\varphi$ for the case $\varphi(n)=\exp(n^r)$ with $r\in (0,\infty)$:
\begin{equation}\label{Svarphijump}
\dim_{\rm H}S(\varphi)=
\left\{
  \begin{array}{ll}
    1, & \hbox{$0<r<1/2$;}\vspace{0.2cm} \\
   1/2, & \hbox{$1/2\leq r<\infty$.}
  \end{array}
\right. \vspace{0.15cm}
\end{equation}
Moreover, they also remarked in \cite[page 402]{LR16} that there is a jump of the Hausdorff dimension from $1$ to $1/2$ in the class $\varphi(n)=\exp(n^r)$ at $r=1/2$ and this jump \textbf{cannot be easily removed} by considering another class of functions. When $\varphi(n)$ tends to infinity sufficiently rapidly, such as $\varphi(n)=b^{c^n}$ with $b,c>1$, Xu \cite{Xu} obtained that the Hausdorff dimension of $S(\varphi)$ is equal to $1/(c+1)$. Refer to Figure $1$ for an illustration of the Hausdorff dimension of $S(\varphi)$.
\begin{figure}[H]
  \centering
  \includegraphics[width=12cm, height=6.7cm]{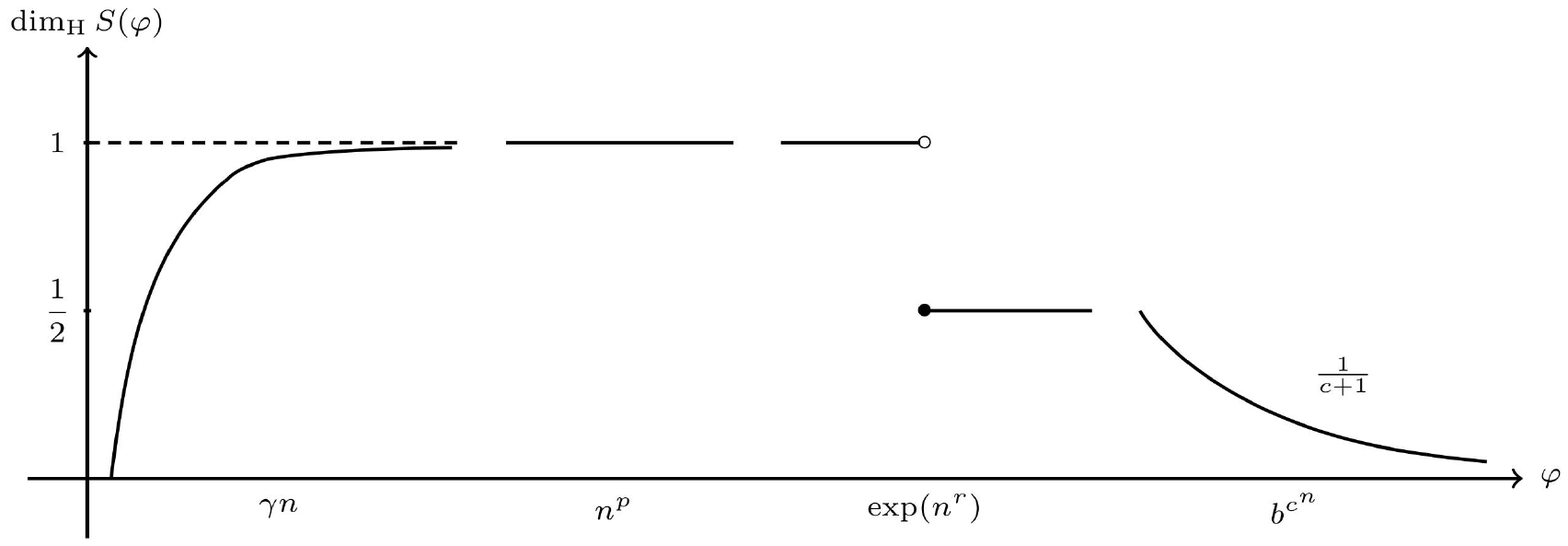}\\
 \caption{The illustration of the Hausdorff dimension of $S(\varphi)$}
\end{figure}

The dimension drop problem of $S(\varphi)$ in \eqref{Svarphijump} from $1$ to $1/2$ remains open, and it is different from the one discussed the results of Jarn\'{i}k-Good. Actually, fix a positive integer $N$ and a sparse sequence $\{n_k\}$ of positive integers with $\frac{n_{k+1}}{k+1}\geq n_1+\cdots+n_k$ for all $k\in \mathbb N$. For any $B>1$, if $x\in \mathbb I$ satisfies that $1\leq a_j(x) \leq N$ for all $j\neq n_k$ and $B^{n_k} <a_{n_k}(x) \leq 2 B^{n_k}$ for all $k\in \mathbb N$, then $\lim_{k\to \infty}S_{n_k-1}(x)/\varphi(n_k-1)=0$, and so $x\notin S(\varphi)$.
This means that the subset used by Wang and Wu \cite{WW} in dealing with the dimension drop problem of Jarn\'{i}k-Good is not a subset of $S(\varphi)$ in \eqref{Svarphijump}, and therefore, the result of Wang and Wu cannot fill the dimension gap of $S(\varphi)$. The reason is the lack of
understanding of the distribution of continued fractions with large coefficients. This is the primary motivation of the current paper.

\subsection{Statements of the main results}
We consolidate the dimension drop problems of both Jarn\'{i}k-Good and Liao-Rams into a problem on the distribution of continued fractions with large coefficients.

Let $\{n_k\}$ be a strictly increasing sequence of positive integers, and let $\{s_k\}$ and $\{t_k\}$ be two sequences of positive numbers with $s_k,t_k \to \infty$ as $k \to \infty$. We model the dimension drop problem of continued fractions by studying the sets
\begin{equation}\label{Ec}
E(\{n_k\},\{s_k\}, \{t_k\}):= E_{S}(\{n_k\}) \cap E_{L}(\{n_k\},\{s_k\}, \{t_k\}),
\end{equation}
where
\begin{align*}
E_{S}(\{n_k\}):=\big\{x\in \mathbb{I}: \{a_j(x)\}_{j\neq n_k (k \in\mathbb N)}\ \text{is bounded}\big\},
\end{align*}
and
\begin{align*}
E_{L}(\{n_k\},\{s_k\}, \{t_k\}):=\big\{x\in \mathbb{I}: s_k < a_{n_k}(x) \leq s_k+ t_k \ \text{for all large $k\in \mathbb N$}\big\}.
\end{align*}
We remark that $\{n_k\}$ is the sequence of positions where the coefficients are large, $\{s_k\}$ is the sequence of positive numbers controlling the growth of large coefficients, and $\{t_k\}$ is the sequence of positive numbers determining the number of possible values of large coefficients.
Note that $t_k \geq 1$ for every large $k$, so there exist integers in the interval $(s_k, s_k+t_k]$. Thus, $E(\{n_k\},\{s_k\}, \{t_k\})$ is nonempty.

We propose to compute the Hausdorff dimension of $E(\{n_k\},\{s_k\}, \{t_k\})$. To this end, we make the following assumptions:
\begin{enumerate}
  \item[(H1)] $\{n_k\}$ satisfies that $n_k/k \to \infty$ as $k \to \infty$;
   \item[(H2)]$\{s_k\}$ and $\{t_k\}$ are logarithmically equivalent in the sense that
  \begin{equation}\label{stequ}
\lim_{k \to \infty} \frac{\log s_k}{\log t_k} =1;
\end{equation}
  \item[(H3)] the limits
\begin{equation}\label{important}
\alpha:=\lim_{k \to \infty} \frac{1}{n_k} \sum^k_{j=1} \log s_j\ \ \ \text{and} \ \ \ \beta:= \lim_{k \to \infty} \frac{1}{n_k}\log s_k
\end{equation}
exist.
\end{enumerate}
(H1) shows that $\{n_k\}$ is a sparse sequence and increases to infinity with a super-linear growth speed. For (H2), we often take $t_k=s_k$ or $t_k =s_k/k$ when $s_k$ tends to infinity rapidly (see applications in subsection \ref{appli}), so it holds in many situations. For (H3), the notation $\alpha$ denotes the relative growth rate of sums $\sum^k_{j=1}\log s_j$ compared with the positions $n_k$, and $\beta$ denotes the limit of the ratios of $\log s_k$ and $n_k$. These two notations play a critical role in the study of the Hausdorff dimension of $E(\{n_k\},\{s_k\}, \{t_k\})$. We also remark that $\alpha$ and $\beta$ can be defined by replacing $\log s_k$ with $\log t_k$ or $\log(s_k+t_k)$ in \eqref{important} because of the equivalence of $\{s_k\}$ and $\{t_k\}$.

Our main method to characterise the Hausdorff dimension of $E(\{n_k\},\{s_k\}, \{t_k\})$ is from thermodynamic formalism. To see this, we introduce the \emph{Diophantine pressure function} $\mathrm{P}(\theta)$ of continued fractions:
\begin{equation}\label{pressuref}
\mathrm{P}(\theta):= \lim_{n \to \infty} \frac{1}{n}\log \sum_{a_1,\dots,a_n \in \mathbb{N}} q^{-2\theta}_n(a_1,\dots,a_n),\ \forall \theta >\frac{1}{2},
\end{equation}
where $q_n(a_1,\dots,a_n)$ is defined in Section \ref{Pre}. It was shown in \cite{KS07} that $\mathrm{P}(\theta)$ has a singularity at $1/2$, and is strictly decreasing, convex and real-analytic on $(1/2,\infty)$. The illustration of the Diophantine pressure function is described as follows. We remark that the Diophantine pressure function and its variants have been employed in the study of the subsystems of continued fractions \cite{JP,JP18,MU96,PV} and the multifractal analysis of continued fractions \cite{FLWW09,IJ,KS07,PW}. We refer the reader to \cite{MU96, Mayer} for detailed analyses of the Diophantine pressure function.
\begin{figure}[H]
  \centering
  \includegraphics[width=6cm, height=5cm]{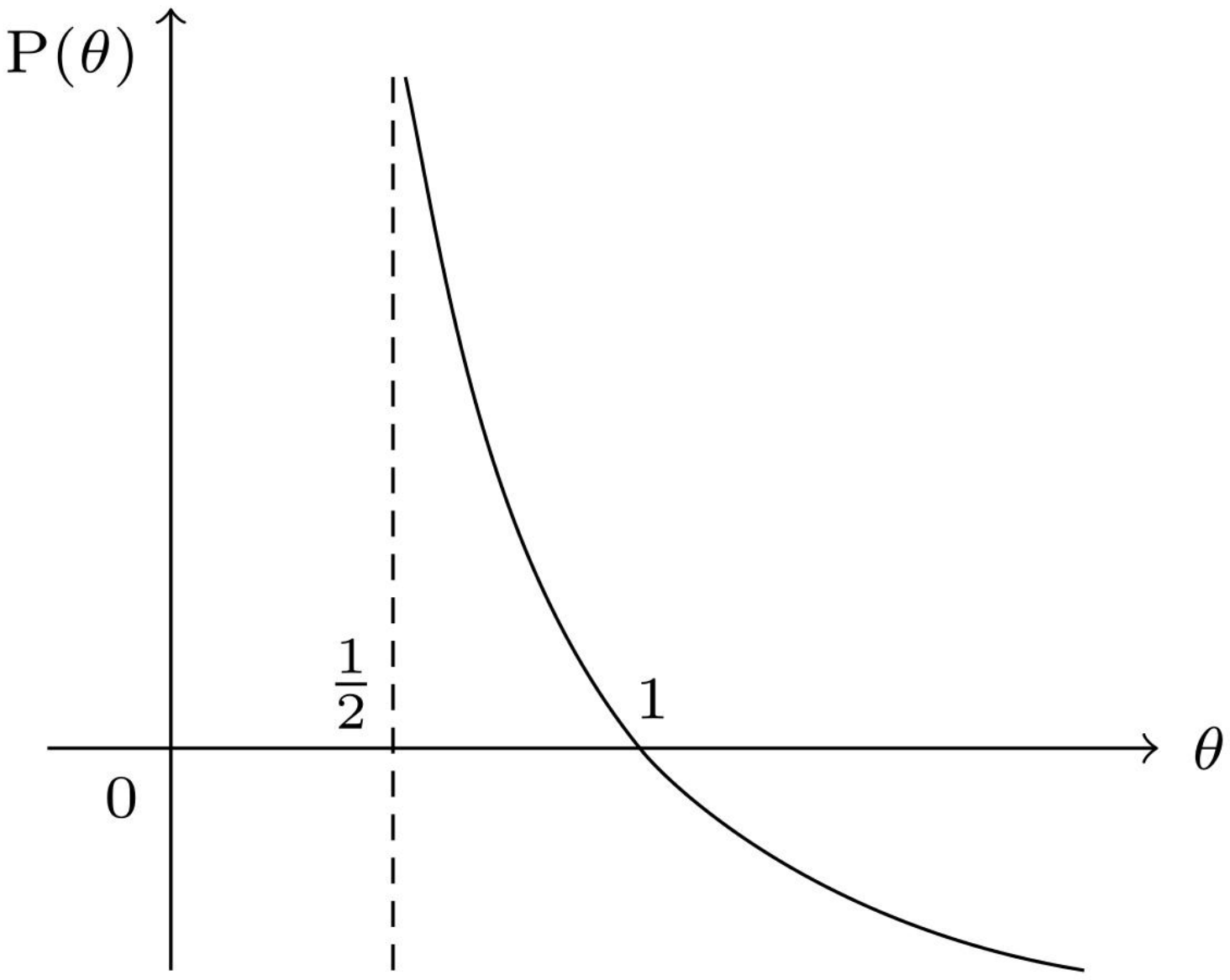}\\
 \caption{The illustration of the Diophantine pressure function}
\end{figure}

Note that $E_{S}(\{n_k\})$ has full Hausdorff dimension. We are now in a position to give a full description of the Hausdorff dimensions of $E(\{n_k\},\{s_k\}, \{t_k\})$ and $E_{L}(\{n_k\},\{s_k\}, \{t_k\})$ according to the values of $\alpha$.

\begin{thmx}\label{Alpha}
Under the hypotheses (H1), (H2) and (H3), we have
\begin{enumerate}
  \item[(i)] when $\alpha \in (0,\infty)$,
\[
\dim_{\rm H} E(\{n_k\},\{s_k\}, \{t_k\}) =\dim_{\rm H} E_{L}(\{n_k\},\{s_k\}, \{t_k\})=\theta(\alpha,\beta),
\]
where $\theta(\alpha,\beta)$ is the unique solution of the Diophantine pressure equation
\begin{equation}\label{PF}
\mathrm{P}(\theta)=(2\alpha -\beta)\theta - (\alpha -\beta);
\end{equation}

  \item[(ii)] when $\alpha =0$,
  \[
\dim_{\rm H} E(\{n_k\},\{s_k\}, \{t_k\})=\dim_{\rm H} E_{L}(\{n_k\},\{s_k\}, \{t_k\})=1;
  \]

  \item[(iii)] when $\alpha =\infty$,
  \[
  \dim_{\rm H} E(\{n_k\},\{s_k\}, \{t_k\}) = \frac{1}{2+\xi},
  \]
  where $\xi \in [0,\infty]$ is defined by
  \[
  \xi:= \limsup_{k \to \infty} \frac{\log s_{k+1}}{\log s_1+\cdots+\log s_k},
  \]
  and
  \[
  \dim_{\rm H} E_{L}(\{n_k\},\{s_k\}, \{t_k\})=\frac{1}{\gamma+1},
  \]
  where $\gamma \in [1,\infty]$ is defined by
  \[
\gamma:=\limsup_{k\to \infty}\exp\left(\frac{\log\log s_k}{n_k}\right).
\]
\end{enumerate}
\end{thmx}
For the convenience of reading, we write $E(\{n_k\},\{s_k\}, \{t_k\})$, $E_{S}(\{n_k\})$ and $E_{L}(\{n_k\},\{s_k\}, \{t_k\})$ as $E$, $E_{S}$ and $E_{L}$ respectively. Let us make some comments on the results of Theorem \ref{Alpha}.

%\begin{remark}\label{rmk_1}
%%The coefficients of the points in $E$ are divided in two parts: the small part and the large part. However, it is not clear which part affects the Hausdorff dimension of $E$ more.
%
%The assertions in Theorem \ref{Alpha} give a \textbf{criterion} on the impact of $E_{L}$ and $E_{S}$ for the Hausdorff dimension of $E$, according to the values of $\alpha$.
%\begin{itemize}
%\item when $\alpha$ is zero, the set $E_{S}$ dominates the Hausdorff dimension of $E$;
%\item  when $\alpha$ is a positive and finite number, by Theorem \ref{Alpha} and Lemma \ref{Eularge} below, we see that $E$ and $E_{L}$ have the same Hausdorff dimension, which means that the large part plays a leading role in the Hausdorff dimension of $E$;
%\item  when $\alpha$ is infinity, the Hausdorff dimension of $E$ is complicated and it cannot be determined by any party.
%\end{itemize}
%\end{remark}

\begin{remark}\label{rmk_2}
We compare the Hausdorff dimension of the sets $E$, $E_{S}$ and $E_{L}$. In view of Theorem \ref{Alpha}, we have
\begin{itemize}
  \item when $\alpha =0$,
  \[
  \dim_{\rm H}E_{S} = \dim_{\rm H} E_{L} =\dim_{\rm H}E =1;
  \]

  \item when $\alpha\in (0,\infty)$,
  \[
  \dim_{\rm H}E_{S}= 1> \dim_{\rm H} E_{L} =\dim_{\rm H}E>1/2;
  \]
  \item when $\alpha =\infty$,
  \begin{equation}\label{ELE}
\dim_{\rm H}E_{S}=1>1/2\geq \dim_{\rm H} E_{L} \geq \dim_{\rm H}E \geq 0.
  \end{equation}
\end{itemize}
There are many possible values for the Hausdorff dimension of $E_{L}$ and $E$ in \eqref{ELE}. For example, applying $n_k=k^2$ and $s_k=t_k=\exp(e^{k^2})$, we derive that $E_{L}$ and $E$ have Hausdorff dimension zero; for $n_k=k^4$ and $s_k=t_k=\exp(e^{k^2})$, we see that $E_{L}$ has Hausdorff dimension one-half but $E$ has Hausdorff dimension zero.
\end{remark}

\begin{remark}\label{rmk_3}
We are interested in the solutions of the Diophantine pressure equation in \eqref{PF}. In fact, the solution is the unique intersection of the Diophantine pressure function $\mathrm{P}(\theta)$ and the linear function $(2\alpha -\beta)\theta - (\alpha -\beta)$. See Figure 3 for the illustration of the solution of the Diophantine pressure equation.
\begin{figure}[H]
  \centering
  \includegraphics[width=8cm, height=6cm]{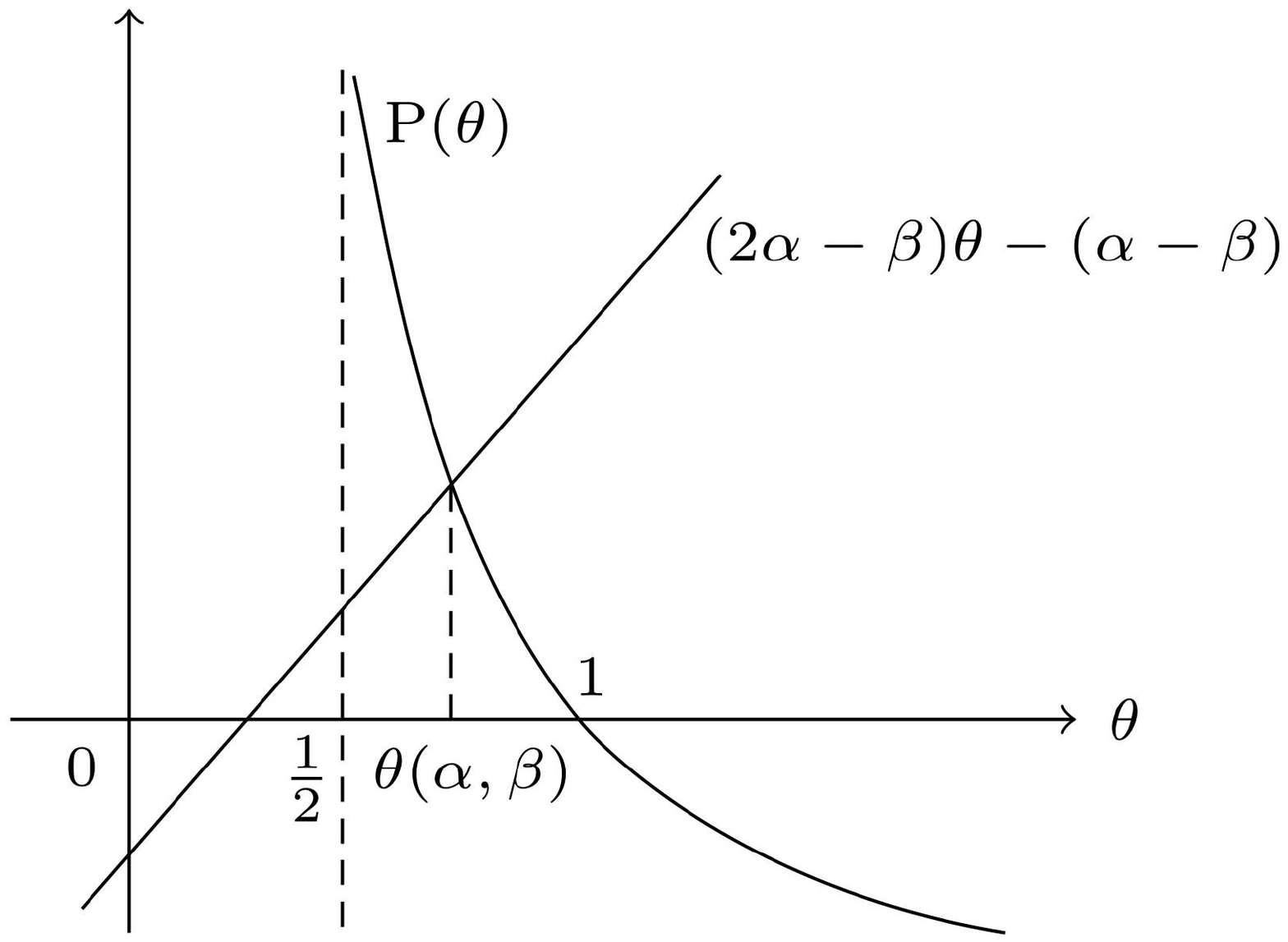}\\
 \caption{The illustration for the solution of the Diophantine pressure equation in \eqref{PF}.}
\end{figure}
According to the relationship between $\alpha$ and $\beta$, that is, $\alpha \geq \beta\geq 0$, we have three types of solutions of the Diophantine pressure equation:
\begin{itemize}
  \item Type I $\theta(\alpha, \alpha)$: the solution of $\mathrm{P}(\theta)=\alpha\theta$;
  \item Type II $\theta(\alpha, 0)$: the solution of $\mathrm{P}(\theta)=(2\theta -1)\alpha$;
  \item Type III $\theta(\alpha, \beta)$ with $0<\beta<\alpha$: the solution of $\mathrm{P}(\theta)=(2\alpha -\beta)\theta - (\alpha -\beta)$.
\end{itemize}
We also point out that if there exist $\alpha^\prime>0$ and $\beta^\prime \geq 0$ with $\alpha^\prime \geq \beta^\prime$ such that for all $\theta \in (1/2,1)$,
\[
(2\alpha^\prime -\beta^\prime)\theta - (\alpha^\prime -\beta^\prime)> (2\alpha -\beta)\theta - (\alpha -\beta),
\]
then we see from Figure 3 that $\theta(\alpha^\prime,\beta^\prime) < \theta(\alpha,\beta)$. In particular, the above three types of solutions have the relation: $\theta(\alpha, 0) <\theta(\alpha, \beta)<\theta(\alpha, \alpha)$.
\end{remark}

\subsection{Applications to coefficients, the sums and maximums of coefficients}\label{appli}
In this section, we will unify the studies of three applications of continued fractions, which was separately investigated in the literature. These three applications respectively correspond to three types of solutions of the Diophantine pressure equation mentioned in Remark \ref{rmk_3}.

The first application is to reprove the main result of Wang and Wu \cite{WW} on the Hausdorff dimension of the Borel-Bernstein set $A(\psi)$. The second application is to fulfill the dimension drop problem of $S(\varphi)$ in \eqref{Svarphijump} proposed by Liao and Rams. In the third application, we find another dimension drop phenomenon for the maximum of coefficients, and fulfill the dimension gap by the Type III solution of the Diophantine pressure equation. For the rest of this subsection, let us state these applications individually.

\subsubsection{Limsup set of coefficients}
For any $B>1$, recall that
\[
A(B):= \left\{x\in \mathbb I: a_n(x) \geq B^n\ \text{for infinitely many $n \in \mathbb N$}\right\}.
\]
The Hausdorff dimension of $A(B)$ has been extensively studied by Good \cite{Goo}, and Wang and Wu \cite{WW}. We will cover their results by means of Theorem \ref{Alpha}. To this end, let us briefly introduce the results of Wang and Wu on the Hausdorff dimension of $A(B)$. For any $n \in \mathbb N$ and $\rho \geq 0$, let
\[
f_{n}(\rho, B):= \sum_{a_1,\dots,a_n \in \mathbb N} \frac{1}{(B^n q^2_n)^\rho},
\]
where $q_n:=q_n(a_1,\dots,a_n)$ is defined in Section \ref{Pre}. Since $f_{n}(\cdot, B)$ is decreasing, we see that $f_{n}(\rho, B)<1$ as $\rho$ is large enough. Define
\[
s_{n}(B):=\inf\left\{\rho \geq 0: f_{n}(\rho, B)<1\right\}.
\]
It was proved in \cite[Lemma 2.4]{WW} that the limit $\lim_{n\to \infty}s_{n}(B)$ exists. Denote the limit by $s(B)$. The Lemma 2.6 of \cite{WW} shows that $s(B)$ is continuous with respect to $B\in (1,\infty)$, and
\[
s(B) \to 1\ \text{as}\ B \to 1\ \ \ \ \text{and}\ \ \ \ s(B) \to \frac{1}{2}\ \text{as}\ B \to \infty.
\]
Furthermore, Wang and Wu \cite[Theorem 3.1]{WW} proved that the Hausdorff dimension of $A(B)$ is $s(B)$. By Theorem \ref{Alpha}, we can obtain another formula for the Hausdorff dimension of $A(B)$ in terms of the Type I solution of the Diophantine pressure equation, which gives more regularities of $s(B)$.

Denote by $\theta(\log B):=\theta(\log B, \log B)$ the unique solution of $\mathrm{P}(\theta)=\theta\log B$.

\begin{thmx} \label{ABN}
For any $B>1$,
\[
\dim_{\rm H}A(B)= \theta(\log B).
\]
\end{thmx}

Theorem \ref{ABN} indicates that the Hausdorff dimension of $A(B)$ belongs to the Type I solutions of the Diophantine pressure equation. Moreover, $B\mapsto\theta(\log B)$ is real analytic, strictly decreasing and convex on $(1,\infty)$.

We give some comments for the lower bound of Theorem \ref{ABN}. Let $\{n_k\}$ be a strictly increasing sequence of positive integers (to be determined). Then $E(\{n_k\},\{s_k\}, \{t_k\})$ with $s_k=t_k=B^{n_k}$ is a subset of $A(B)$. Note that $\beta = \lim_{k \to \infty}\log s_k/n_k =\log B$, by Theorem \ref{Alpha}, we see that the Hausdorff dimension of $A(B)$ is not less than the unique solution of $\mathrm{P}(\theta)= (2\alpha -\log B)\theta - (\alpha -\log B)$, where $\alpha \geq \log B$ is defined as in \eqref{important}. To obtain the optimal lower bound of $\dim_{\rm H}A(B)$, it is sufficient to find the minimum of
\[
\alpha \mapsto (2\alpha -\log B)\theta - (\alpha -\log B)= (2\theta-1)\alpha+ (1-\theta)\log B
\]
for fixed $\theta>1/2$. Since the function is increasing on $[\log B, \infty)$, its minimum is achieved at $\alpha =\log B$. Hence $\theta(\log B)$ is the optimal lower bound of the Hausdorff dimension of $A(B)$. To see this, taking $n_k=2^{k^2}$, we have $\alpha=\beta=\log B$.

Based on the Hausdorff  dimension of $A(B)$, Wang and Wu \cite[Theorem 4.2]{WW} further obtained the Hausdorff dimension of $A(\psi)$.
In the light of Theorem \ref{ABN}, with the convention $\frac{1}{\infty}=0$, we can also give a full description for the Hausdorff dimension of $A(\psi)$.

\begin{thmx} \label{HDApsi}
Let $\psi: \mathbb{R}^+ \to \mathbb{R}^+$ be a function. Denote
\begin{equation}\label{DefBandb}
B_\psi:=\liminf_{n \to \infty}\exp\left(\frac{\log\psi(n)}{n}\right) \ \text{and}\ \ b_\psi:=\liminf_{n\to\infty}\exp\left(\frac{\log\log\psi(n)}{n}\right).
\end{equation}
Then
\[
\dim_{\rm H}A(\psi)=
\left\{
  \begin{array}{ll}
    1, & \hbox{$B_\psi=1$;} \vspace{0.15cm}\\
  \theta(\log B_\psi), & \hbox{$1<B_\psi<\infty$;}\vspace{0.12cm} \\
    \dfrac{1}{b_\psi+1}, & \hbox{$B_\psi=\infty$.}
  \end{array}
\right.
\]
\end{thmx}

\subsubsection{Asymptotic behaviours of the sum of coefficients}

We will apply Theorem \ref{Alpha} to the Hausdorff dimension of the set
\begin{equation*}
S(\varphi):=\left\{x\in \mathbb I: \lim_{n \to \infty}\frac{S_n(x)}{\varphi(n)}=1 \right\},
\end{equation*}
where $\varphi: \mathbb{R}^+ \to \mathbb R^+$ is an increasing function with $\varphi(n) \to \infty$ as $n \to \infty$. Particularly, we propose to fill the dimension gap of $S(\varphi)$ in \eqref{Svarphijump}.

The Hausdorff dimension of $S(\varphi)$ for the linear function $\varphi$ had been determined by Iommi and Jordan \cite{IJ}, so we are interested in the dimensions of $S(\varphi)$ for super-linear functions. There are other motivations to study this kind of problem. Philipp \cite[Theorem 1]{Phi} proved that if $\varphi$ satisfies that $\varphi(n)/n$ is non-decreasing, then for Lebesgue almost every $x\in\mathbb I$,
\[
\limsup_{n\to \infty}\frac{S_n(x)}{\varphi(n)}=0\ \ \ \text{or}\ \ \ \limsup_{n\to \infty}\frac{S_n(x)}{\varphi(n)}=\infty
\]
according to the series $\sum_{n \geq 1}1/\varphi(n)$ converges or diverges. This indicates the limit behaviours of $S_{n}(x)$ is complicated. To understand it better, much attention has been paid to the Hausdorff dimension of $S(\varphi)$ for super-linear functions $\varphi$. It was proved by Wu and Xu \cite[Section 4]{WX} that if
\begin{equation*}
\lim_{n \to \infty} \frac{\varphi(n)}{n}= \infty\ \ \ \text{and}\ \ \ \limsup_{n \to \infty}\frac{\log\log\varphi(n)}{\log n}<\frac{1}{2},
\end{equation*}
then $\dim_{\rm H}S(\varphi) =1$. Applying Theorem \ref{Alpha} to $S(\varphi)$, we extend the results of Wu and Xu \cite{WX} to a very large class of functions $\varphi$.

\begin{thmx}\label{AM}
Let $\varphi:\mathbb{R}^+\to\mathbb{R}^+$ be an increasing function with $\varphi(n)/n \to \infty$ as $n \to \infty$. Assume that
\[
\limsup_{n \to \infty} \frac{\log \varphi(n)}{\sqrt{n}}=0.
\]
Then
\[
\dim_{\rm H} S(\varphi)=1.
\]
\end{thmx}

As a consequence of Theorem \ref{AM}, for $\varphi(n)=\exp(\sqrt{n}\cdot r(n))$ with $r(n) \to 0$ as $n \to \infty$, we have the Hausdorff dimension of $S(\varphi)$ is one. This generalises the result of Liao and Rams \cite[Theorem 1.2]{LR16}, which requires some extra regular conditions of the function $r$.

As showed in \eqref{Svarphijump}, there is a jump discontinuity of the Hausdorff dimension of $S(\varphi)$ from $1$ to $1/2$ in the class $\varphi(n)=\exp(n^r)$ at $r=1/2$. We give some sufficient conditions for $S(\varphi)$ to have a dimension drop.

\begin{thmx}\label{one-halfM}
Let $\varphi:\mathbb{R}^+\to\mathbb{R}^+$ be an increasing function with $\varphi(n)/n \to \infty$ as $n \to \infty$. We have
\begin{enumerate}
  \item[(i)] if for any $\varepsilon>0$, there exists $\delta>0$ such that for all large $n$,
\begin{equation}\label{ed}
\log\varphi(n+\varepsilon\sqrt{n}) -\log \varphi(n) \geq \delta,
\end{equation}
then $\dim_{\rm H} S(\varphi) \leq 1/2$;
\item[(ii)] if
\begin{equation}\label{limsupinfinity}
\limsup_{n \to \infty}\frac{\log \varphi(n)}{n}=\infty,
\end{equation}
then $\dim_{\rm H} S(\varphi) \leq 1/2$;
%  \item[(iii)] if
%\begin{equation}\label{limsupzero}
%\limsup_{n \to \infty} \frac{\log \varphi(n+1)}{\log \varphi(n)} = 1,????
%\end{equation}
%then $\dim_{\rm H} S(\varphi) \geq 1/2$.
\end{enumerate}
\end{thmx}

As a corollary, we obtain the Hausdorff dimension of $S(\varphi)$ for some critical cases where $\varphi(n)$ is ``around" the function $\exp(c\sqrt{n})$ with $c>0$. This refines the results of Liao and Rams \cite[Theorem 1.1]{LR16}.

\begin{corollary}\label{onehalpCor1}
Let $\varphi(n) = \exp(c\sqrt{n} +r_1(n))$ with $0<c<\infty$, where $r_1:\mathbb{R}^+\to\mathbb{R}^+$ is increasing and $r_1(n)/\sqrt{n}\to 0$ as $n \to \infty$.
Then
\[
\dim_{\rm H} S(\varphi) =\frac{1}{2}.
\]
\end{corollary}

\begin{corollary}\label{onehalpCor2}
Let $\varphi(n)=\exp(c\sqrt{n}+r_2(n))$ with $0<c<\infty$, where $r_2:\mathbb{R}^+ \to \mathbb{R}$ is a function such that $c\sqrt{n}+r_2(n)$ is increasing and satisfies
\begin{equation}\label{maxine}
\lim_{m \to \infty}\max\left\{\left|r_2(k)-r_2(m^2)\right|:m^2< k \leq(m+1)^2\right\}=0.
\end{equation}
Then
\[
\dim_{\rm H} S(\varphi)  =\frac{1}{2}.
\]
\end{corollary}

Applying the first statement of Theorem \ref{Alpha} to $S(\varphi)$, we solve the dimension drop problem of $S(\varphi)$ in \eqref{Svarphijump}.
Actually, by Theorems \ref{AM} and \ref{one-halfM}, in order to fill the dimension gap, the potential function $\varphi$ needs to satisfy the properties
\[
\limsup_{n\to \infty}\frac{\log\varphi(n)}{\sqrt{n}}>0\ \ \text{and}\ \ \limsup_{n\to \infty}\frac{\log\varphi(n)}{n}<\infty.
\]
The following result gives some examples of functions of all possible orders to fill the dimension gap of $S(\varphi)$.

\begin{thmx}\label{CM}
We have
\begin{enumerate}
  \item[(i)] for $\varphi(n)=\exp(c\cdot\frac{(\lfloor dn^{1-r}\rfloor)^{r/(1-r)}}{d^{r/(1-r)}})$ with $c,d\in (0,\infty)$ and $r\in [1/2,1)$,
\[
\dim_{\rm H} S(\varphi) = \eta_d(c),
\]
where $\eta_d(c)$ is the unique solution of $\mathrm{P}(\theta)= cd(1-r)\left(2\theta-1\right)$;

  \item[(ii)] for $\varphi(n)= \exp(c \cdot \exp(\gamma \lfloor \frac{\log n}{\gamma}\rfloor))$ with $c, \gamma\in (0,\infty)$,
  \[
\dim_{\rm H} S(\varphi)= \xi_\gamma(c),
\]
where $\xi_\gamma(c)$ is the unique solution of
\[
\mathrm{P}(\theta)= c\left(\frac{e^\gamma+1}{e^\gamma-1}\theta- \frac{1}{e^\gamma-1}\right).
\]
\end{enumerate}
\end{thmx}

In particular, taking $d=1$ and $r=1/2$ in (i) of Theorem \ref{CM}, we have $\varphi(n)=e^{c\lfloor n\rfloor}$ and the Hausdorff dimension of $S(\varphi)$ is the unique solution of $\mathrm{P}(\theta)= c\left(\theta-1/2\right)$, which fills the dimension gap in \eqref{Svarphijump}. Actually, there are uncountable many functions to fill such a dimension gap, such as $\varphi(n)=p(n)e^{c\lfloor n\rfloor}$ or $\varphi(n)=e^{c\lfloor n\rfloor}/p(n)$, where $p(n)$ is a polynomial function.

\subsubsection{Limsup and liminf sets of the maximum of coefficients}
For any $x\in \mathbb I$ and $n \in \mathbb N$, let
\[
M_n(x):= \max\left\{a_1(x), a_2(x),\dots,a_n(x)\right\}
\]
be the maximum of the first $n$ coefficients in the continued fraction expansion of $x$. For any $\psi:\mathbb{R}^+\to\mathbb{R}^+$, let
\[
M(\psi):=\left\{x\in \mathbb{I}:M_n(x) \geq \psi(n)\ \text{for infinitely many $n \in \mathbb N$}\right\}.
\]
Philipp \cite[Theorem 3]{Phi75} proved a zero-one law for $M_n(x)$: if $\psi$ is nondecreasing, then $M(\psi)$ has full or null Lebesgue measure according to $\sum_{n \geq 1} 1/\psi(n)$ diverges or converges.

We show that $A(\psi)$ and $M(\psi)$ have the same Hausdorff dimension.

\begin{thmx}\label{Mpsi}
Let $\psi: \mathbb{R}^+ \to \mathbb{R}^+$ be a function. Denote
\[
B_\psi:=\liminf_{n \to \infty}\exp\left(\frac{\log\psi(n)}{n}\right) \ \text{and}\ \ b_\psi:=\liminf_{n\to\infty}\exp\left(\frac{\log\log\psi(n)}{n}\right).
\]
Then
\[
\dim_{\rm H}M(\psi)=
\left\{
  \begin{array}{ll}
    1, & \hbox{$B_\psi=1$;} \vspace{0.15cm}\\
  \theta(\log B_\psi), & \hbox{$1<B_\psi<\infty$;}\vspace{0.12cm} \\
    \dfrac{1}{b_\psi+1}, & \hbox{$B_\psi=\infty$.}
  \end{array}
\right.
\]
\end{thmx}

For any $\psi:\mathbb{R}^+\to\mathbb{R}^+$, let
\[
\widehat{M}(\psi):=\left\{x\in \mathbb{I}:M_n(x) \geq \psi(n)\ \text{for sufficiently large $n \in \mathbb N$}\right\},
\]
which can be treated as the dual set of $M(\psi)$. For any $C>1$, denote by $\widehat{M}(C)$ the set $\widehat{M}(\psi)$ with $\psi(n)=C^n$ if no confusion arises.

\begin{thmx}\label{BhatMN}
For any $C>1$,
\[
\dim_{\rm H}\widehat{M}(C)= \widehat{\theta}(\log C),
\]
where $\widehat{\theta}(\log C)$ is the unique solution of $\mathrm{P}(\theta)=(\sqrt{\theta}+\sqrt{2\theta-1})^2\log C$.
\end{thmx}

Let us make some comments on the lower bound of Theorem \ref{BhatMN}. For $\gamma >0$ (to be determined), let $n_k:=\lfloor e^{\gamma k}\rfloor$ and $s_k=t_k:= C^{e^{\gamma (k+1)}}$. Then $E(\{n_k\},\{s_k\},\{t_k\})$ is a subset of $\widehat{M}(C)$, and so $\dim_{\rm H}\widehat{M}(C) \geq \dim_{\rm H} E(\{n_k\},\{s_k\},\{t_k\})$. For this special case, we have
\[
\alpha = \lim_{k \to \infty} \frac{1}{n_k}\sum^k_{j=1}\log s_j = \frac{e^{2\gamma}}{e^{\gamma}-1}\log C\ \text{and}\ \beta = \lim_{k \to \infty} \frac{\log s_k}{n_k}=e^{\gamma}\log C.
\]
Then, it follows from Theorem \ref{Alpha} that the Hausdorff dimension of $\widehat{M}(C)$ is not less than the unique solution of
\begin{equation}\label{TIII}
\mathrm{P}(\theta) = e^\gamma\left(\frac{e^\gamma+1}{e^\gamma-1} \theta -  \frac{1}{e^\gamma-1}\right)\log C.
\end{equation}
To obtain the optimal lower bound of $\dim_{\rm H}\widehat{M}(C)$, it is sufficient to compute the minimum of the function
\[
f_{\theta}(\gamma):=e^\gamma\left(\frac{e^\gamma+1}{e^\gamma-1} \theta -  \frac{1}{e^\gamma-1}\right)
\]
for fixed $1/2<\theta<1$. A simple calculation shows that $f_{\theta}(\gamma)$ has a unique minimum at
\[
\gamma=\log \left(1+\sqrt{\frac{2\theta-1}{\theta}} \right),\ \ \text{i.e.,}\  e^{\gamma}= 1+\sqrt{\frac{2\theta-1}{\theta}}.
\]
Moreover, the minimum of $f_{\theta}(\gamma)$ is
\[
\sqrt{\frac{\theta}{2\theta-1}}\left(1+\sqrt{\frac{2\theta-1}{\theta}}\right)\left( \left(2+\sqrt{\frac{2\theta-1}{\theta}}\right)\theta-1\right)=( \sqrt{\theta}+\sqrt{2\theta-1})^2.
\]
Hence the solution of $\mathrm{P}(\theta)=( \sqrt{\theta}+\sqrt{2\theta-1})^2\log C$ is the desired lower bound of $\dim_{\rm H}\widehat{M}(C)$.
Actually, $\widehat{\theta}(\log C)$ is obtained from the optimization of the Type III solutions of the Diophantine pressure equation in \eqref{TIII}.

Based on the dimension of $\widehat{M}(C)$, we obtain a full description for the Hausdorff dimension of $\widehat{M}(\psi)$.

\begin{thmx}\label{AMhat}
Let $\psi: \mathbb{R}^+ \to \mathbb{R}^+$ be a function.
Denote
\begin{equation}\label{DefC}
C_\psi:=\limsup_{n \to \infty}\exp\left(\frac{\log\psi(n)}{n}\right) \ \ \ \text{and}\ \ \ c_\psi:=\limsup_{n \to \infty}\exp\left(\frac{\log\log\psi(n)}{n}\right).
\end{equation}
Then
\begin{enumerate}
  \item[(i)] If $C_\psi=1$, then $\dim_{\rm H} \widehat{M}(\psi)=1$;
  \item[(ii)] Assume that the limsup in the definition of $C$ is a limit. If $1<C_\psi<\infty$, then $$\dim_{\rm H} \widehat{M}(\psi)=\widehat{\theta}(\log C_\psi);$$
  \item[(iii)] If $C_\psi=\infty$, then $$\dim_{\rm H} \widehat{M}(\psi)=\frac{1}{c_\psi+1}.$$
\end{enumerate}
\end{thmx}

\begin{remark}
We remark that the hypothesis ``the limsup in the definition of $C_\psi$ is a limit" in (ii) cannot be removed. For example, for $C>1$, consider the function
\[
\phi(n):=C^{2^{k^2}}\ \ \ \text{as}\ \ \ 2^{k^2} \leq n<2^{(k+1)^2}, \ \ \ \text{for every}\ \  k\in \mathbb{N}\cup\{0\},
\]
we have
\[
\limsup_{n \to \infty}\frac{\log\phi(n)}{n}= \log C,
\]
but the Hausdorff dimension of $\widehat{M}(\phi)$ is $\theta(\log C)$, namely the unique solution of the Diophantine pressure equation $\mathrm{P}(\theta)=\theta\log C$. To see this, for any $k \in \mathbb N$, let $n_k:=2^{k^2}$ and $s_k=t_k=C^{n_k}$. Then $E(\{n_k\},\{s_k\}, \{t_k\})$ is a subset of $\widehat{M}(\phi)$ and has Hausdorff dimension $\theta(\log C)$. Hence $\dim_{\rm H}\widehat{M}(\phi) \geq \theta(\log C)$. For the upper bound, note that $\widehat{M}(\phi)$ is contained in $M(\psi)$ with $\psi(n)=C^n$, by Theorem \ref{Mpsi}, we deduce that $\dim_{\rm H}\widehat{M}(\phi) \leq \theta(\log C)$.
\end{remark}

As a consequence of Theorems \ref{Mpsi} and \ref{AMhat}, we obtain the multifractal analyses of $\log M_n(x)$.

\begin{corollary}
For any $\tau\in (0,\infty)$,
 \[
\dim_{\rm H}\left\{x\in \mathbb{I}: \limsup_{n \to \infty} \frac{\log M_n(x)}{n} =\tau \right\}=\theta(\tau)
\]
and
\[
\dim_{\rm H}\left\{x\in \mathbb{I}: \liminf_{n \to \infty} \frac{\log M_n(x)}{n} =\tau \right\}=\widehat{\theta}(\tau).
\]
\end{corollary}

\subsection{Structure of the paper}

The rest of the paper is organised as follows. In Section 2, we introduce some definitions and basic properties of continued fractions, and Hausdorff measure and Hausdorff dimension. Section 3 is devoted to giving the proof of the first statement of Theorem \ref{Alpha}, which is divided into two parts: the upper bound and the lower bound. For the upper bound, we define a set $X(b,c)$ in \eqref{equ:XBC} and estimate its exact Hausdorff dimension in terms of the unique solutions of the Diophantine pressure function. For the lower bound, we introduce the Diophantine pressure function in our proof via Lemma \ref{NP}, and then construct the subset of $E(\{n_k\},\{s_k\}, \{t_k\})$ to adapt Falconer's formula. The proofs of the second and third statements of Theorem \ref{Alpha} are provided in Sections 4 and 5. In Section 6, we give the proofs of three applications derived from Theorem \ref{Alpha}.

\section{Preliminaries}\label{Pre}

Throughout the paper, we write $\mathbb R$, $\mathbb R^+$ and $\mathbb Q^+$ to denote the sets of real numbers, positive real numbers and positive rational numbers respectively. For any $y \in \mathbb R$, let $\lfloor y\rfloor$ be the greatest integer less than or equal to $y$, and let $\lceil y\rceil$ be the smallest integer greater than or equal to $y$. We follow the convention $\mathbb N:=\{1,2,3,\dots\}$, $\mathbb N_{\leq N}:=\{1,2,\dots,N,\}$ and $\mathbb N_{\geq N}:=\{N,N+1,\dots\}$ for any $N\in \mathbb N$. For any $n\in \mathbb N$, let $\mathbb N^n$ be the set of all $n$-tuples $(\sigma_1,\sigma_2,\dots,\sigma_n)$, where $\sigma_i \in \mathbb N$ for all $1\leq i\leq n$. We use $|I|$ to denote the diameter of a set $I \subseteq \mathbb R$, and use $\sharp \mathcal{N}$ to mean the cardinality of a finite set $\mathcal{N}$.

Let us first introduce the continued fraction expansions of real numbers, and collect several basic properties of continued fractions, and We give the definitions of Hausdorff measure and Hausdorff dimension, and some techniques of calculating Hausdorff dimensions in this section.

\subsection{Continued fractions}
Let $T: [0,1) \to [0,1)$ be the \emph{Gauss map} defined as $T(0)\vcentcolon=0$ and
\begin{equation*}
T(x)\vcentcolon= \frac{1}{x}-  \left\lfloor\frac{1}{x}\right\rfloor, \ \ \  \forall\ x\in (0,1).
\end{equation*}
Denote by $T^n$ the $n$th iteration of the map $T$. For any $x\in (0,1)$, let $a_1(x)\vcentcolon=\lfloor1/x\rfloor$ and $a_{n+1}(x)\vcentcolon=a_1(T^{n}(x))$ for $n\geq 1$.
Then $x$ is written as the continued fraction expansion of the form
\begin{equation}\label{CF}
x =  \dfrac{1}{a_1(x) +\dfrac{1}{a_2(x) + \dfrac{1}{a_3(x)+ \ddots}}}=\vcentcolon [a_1(x),a_2(x),a_3(x),\ldots],
\end{equation}
where $a_1(x),a_2(x),a_3(x),\dots$ are positive integers, called the \emph{coefficients} (or \emph{partial quotients}) of the continued fraction expansion of $x$.

It is known in \cite[Theorem 14]{Khi} that $x \in (0,1)$ is a rational number if and only if its continued fraction expansion is finite, that is, $\exists\, n \in \mathbb{N}$ such that $T^n(x)=0$. Hence an irrational number admits a unique infinite sequence of its coefficients. Therefore, it is convenient no longer to distinguish between an irrational number and an infinite sequence of positive integers.

For any $x=[a_1,a_2,\dots,a_n,\dots]\in \mathbb{I}$ and $n \in \mathbb N$, the \emph{continuant} $q_n(a_1, a_2, \dots, a_n)$ or $q_n(x)$ is defined to be the denominator of the rational number $[a_1,a_2,\ldots,a_n]$. The continuants are also sometimes called \emph{Euler polynomials}, see \cite[page 89]{CF} for more information of Euler polynomials. With the conventions $q_{-1}\equiv0$ and $q_1\equiv1$, $q_n$ has the recursive formula: for each $n \in \mathbb N$,
\begin{equation}\label{pq}
q_n(a_1,a_2,\dots,a_n)=a_nq_{n-1}(a_1,a_2,\dots,a_{n-1})+q_{n-2}(a_1,a_2,\dots,a_{n-2}).
\end{equation}
 As a consequence, we have
\begin{equation}\label{qnleq}
a_1a_2 \cdots a_n  \leq q_n(a_1,a_2,\dots,a_n)< 2^n a_1 a_2\cdots a_n.
\end{equation}
Actually, \eqref{pq} can be extended to the more general formula
\begin{align}\label{mgf}
q_{n+k}(a_1,\dots,a_n,b_1,\dots,b_k)=q_{n}&(a_1,\dots,a_n)q_k(b_1,\dots,b_k) \nonumber\\
&+q_{n-1}(a_1,\dots,a_{n-1})q_{k-1}(b_2,\dots,b_k)
\end{align}
for every $(a_1,\dots,a_n) \in \mathbb N^n$ and every $(b_1,\dots,b_k) \in \mathbb N^k$.

\begin{lemma}\label{qnsep}
For any $k,n \in\mathbb N$ with $1\leq k \leq n$ and $a_1,\dots,a_n \in \mathbb{N}$,
\[
q_n(a_1,\dots,a_k,\dots,a_n)> \frac{a_k}{2}q_{n-1}(a_1,\dots,a_{k-1},a_{k+1},\dots,a_n).
\]
\end{lemma}
\begin{proof}
The proof is a consequence of \eqref{mgf}.
\end{proof}

For any $n \in \mathbb{N}$ and $(a_1,\dots,a_n) \in \mathbb{N}^n$, we call
\begin{equation*}\label{equ:cylinder}
I_n(a_1,\ldots,a_n)\vcentcolon=\left\{x\in (0,1):a_1(x)=a_1,\ldots, a_n(x)=a_n\right\}
\end{equation*}
a \emph{cylinder of order $n$} associated with $(a_1, \ldots, a_n)$. Then $I_n(a_1,\ldots,a_n)$ is an interval of the length
\[
|I_n(a_1,\dots,a_n)| = \frac{1}{q_n(q_n+q_{n-1})},
\]
where $q_n$ satisfies the recursive formula in \eqref{pq}. Hence
\begin{equation}\label{cylinder}
\frac{1}{2^{2n+1}(a_1\cdots a_n)^2}< \frac{1}{2q^2_n} <|I_n(a_1, \ldots, a_n)| < \frac{1}{q^2_n} \leq \frac{1}{(a_1\cdots a_n)^2}.
\end{equation}

The following result can be viewed as the bounded distortion property of continued fractions.

\begin{lemma}[{\cite[Lemma A.2]{Moreira}}]\label{q}
For any $(a_1,\dots,a_n) \in \mathbb N^n$ and $(b_1,\dots,b_k) \in \mathbb N^k$,
\[
\frac{1}{2} \leq \frac{|I_{n+k}(a_1,\dots,a_n,b_{1},\dots,b_{k})|}{|I_{n}(a_1,\cdots,a_n)|\cdot|I_{k}(b_{1},\dots,b_{k})|} \leq 2.
\]
\end{lemma}
As a consequence, for any $(a_1,\dots,a_n,\dots,a_{n+k}) \in \mathbb N^{n+k}$,
\begin{equation}\label{cysep}
\frac{1}{8} \leq \frac{|I_{n+k}(a_1,\dots,a_n,\dots,a_{n+k})|}{|I_{1}(a_n)|\cdot|I_{n+k-1}(a_1,\dots,a_{n-1},a_{n+1},\dots,a_{n+k})|} \leq 8.
\end{equation}

\begin{proof}
The proof is a direct consequence of \eqref{mgf}.
\end{proof}

\subsection{Hausdorff measure and Hausdorff dimension}

Let $E$ be a subset of $\mathbb R$, and let $s\geq 0$ be a number. For any $\delta>0$, define
\[
\mathcal{H}^s_\delta(E):= \inf\left\{\sum^\infty_{n = 1}|U_n|^s: E \subset \bigcup^\infty_{n =1} U_n \ \text{and}\ |U_n| \leq \delta,\ \forall n \in \mathbb N\right\}.
\]
When $\delta$ decreases, the infimum $\mathcal{H}^s_\delta(E)$ increases, and so it approaches a limit as $\delta$ goes to zero. Write
\[
\mathcal{H}^s (E):=\lim_{\delta \to 0} \mathcal{H}^s_\delta(E).
\]
We call $\mathcal{H}^s (E)$ the \emph{$s$-dimensional Hausdorff measure} of $E$. The \emph{Hausdorff dimension} $\dim_{\rm H}E$ of $E$ is defined by
\[
\dim_{\rm H}E:=\inf\left\{s\geq 0: \mathcal{H}^s (E) =0\right\}=\sup\left\{s \geq 0: \mathcal{H}^s (E) =\infty\right\}.
\]
By definition, the Hausdorff dimension of $E$ is a critical value such that $\mathcal{H}^s (E)=0$ for $s>\dim_{\rm H}E$ and $\mathcal{H}^s (E)=\infty$ for $0\leq s<\dim_{\rm H}E$. However, for $s=\dim_{\rm H}E$, $\mathcal{H}^s (E)$ may be zero or infinite, or may satisfy $0<\mathcal{H}^s (E)<\infty$.

Hausdorff dimension has the following properties: (i) Monotonicity: if $E\subseteq F$, then $\dim_{\rm H}E\leq \dim_{\rm H}F$; (ii) Countable stability: if $\{E_n\}$ be a sequence of subsets of $\mathbb R$, then $\dim_{\rm H} (\cup_{n \geq 1}E_n) = \sup_{n \geq 1}\{\dim_{\rm H}E_n\}$; (iii) Each countable set is of Hausdorff dimension zero.

It is worth pointing out that calculations of Hausdorff dimension usually involve an upper estimate and a lower estimate, which are hopefully equal. As a general rule, we obtain upper bounds for Hausdorff dimension by finding natural coverings, and lower bounds by putting mass distributions on the fractal set. We refer the reader to \cite[Section 4]{Fal90} for techniques of calculating Hausdorff dimensions.

The following result is often used to estimate the upper bound of Hausdorff dimension.

\begin{lemma}[{\cite[Proposition 4.1]{Fal90}}]\label{upp}
For each $n \in \mathbb N$, let $\mathcal{E}_n$ be a finite collection of sets of diameter at most $\delta_n$ that cover $E$. If $\delta_n \to 0$ as $n \to \infty$, then
\[
\dim_{\rm H}E \leq \liminf_{n \to \infty} \frac{\log \sharp\mathcal{E}_n}{-\log \delta_n}.
\]
\end{lemma}

The following lemma is called the \emph{mass distribution principle}, which provides an effective way to get lower bounds for Hausdorff dimension.

\begin{lemma}[{\cite[Mass Distribution Principle]{Fal90}}]\label{MDP}
Let $\mu$ be a mass distribution on $E$ and suppose that for some $s>0$ there are numbers $c>0$ and $\varepsilon >0$ such that
\[
\mu(U) \leq c|U|^s
\]
for all sets $U$ with $|U| \leq \varepsilon$. Then $\dim_{\rm H}E \geq s$.
\end{lemma}

The mass distribution principle can give the lower bounds for the Hausdorff dimension of Cantor-type sets.

Let $\mathrm{E}_0 \supset \mathrm{E}_1  \supset \mathrm{E}_2 \supset\cdots$ be a decreasing sequence of sets, with each $\mathrm{E}_k$ a union of a finite number of disjoint closed intervals (called \emph{$k$th level basic intervals}), with each interval of $\mathrm{E}_k$ containing at least two intervals of $\mathrm{E}_{k+1}$, and the maximum length of $k$th level intervals tending to $0$ as $k$ goes to infinity. Then
\begin{equation}\label{Falexample}
\mathrm{E} := \bigcap^\infty_{k = 0} \mathrm{E}_k.
\end{equation}
is a totally disconnected subset of $[0, 1]$ which is generally a fractal (see Figure 4).
There are many formulas for the lower bound of the Hausdorff dimension of $\mathrm{E}$ in Section 4.1 of \cite[]{Fal90}. Here we list one of them.

\begin{figure}[H]
  \centering
  \includegraphics[width=11.6cm, height=4.9cm]{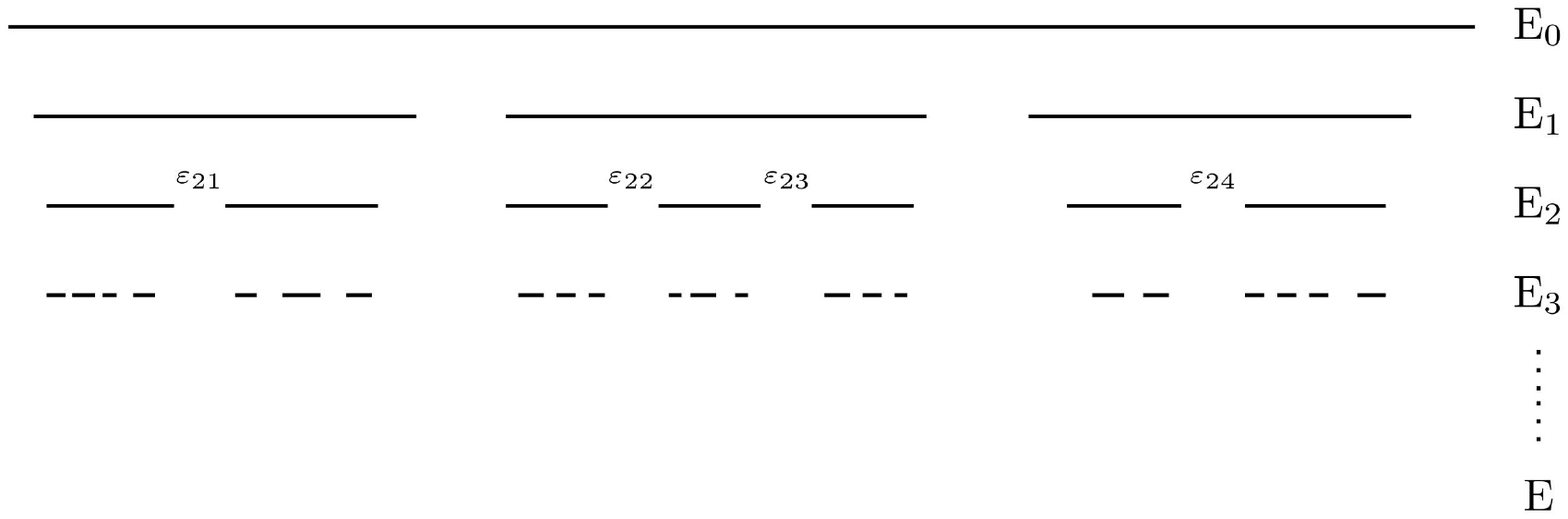}\\
 \caption{The structure of $\mathrm{E}_k$. By the definitions of $m_k$ and $\varepsilon_k$ in Lemma \ref{lower}, $m_1$ can take the value $3$, $m_2$ can take the value $2$, and $\varepsilon_2$ can take the value $\min\{\varepsilon_{21},\varepsilon_{22},\varepsilon_{23},\varepsilon_{24}\}$.}
\end{figure}

\begin{lemma}[{\cite[Example 4.6]{Fal90}}]\label{lower}
Suppose in the construction \eqref{Falexample}, each $(k-1)$th level basic interval contains at least $m_k$ $k$th level basic intervals $(k = 1, 2,\dots)$ which are separated by gaps of at least $\varepsilon_k$, where $m_k \geq 2$ and $0 < \varepsilon_{k+1} < \varepsilon_k$ for each $k$. Then
\[
\dim_\mathrm{H} \mathrm{E} \geq \liminf_{k \to \infty} \frac{\log(m_1m_2 \cdots m_{k-1})}{-\log (m_{k}\varepsilon_k)}.
\]
\end{lemma}

The Hausdorff dimensions of many Cantor-type sets of continued fractions are well understood. Let $\{u_n\}$ and $\{v_n\}$ be sequences of positive numbers such that $v_n \geq 1$ for all $n \in \mathbb N$. Write
\begin{equation}\label{Euv}  F(\{u_n\}, \{v_n\}):= \big\{x\in \mathbb{I}: u_n < a_n(x) \leq u_n+v_n\ \text{for all $n\in \mathbb N$}\big\}.
\end{equation}
Liao and Rams \cite{LR22} obtained the Hausdorff dimension of $F(\{u_n\}, \{v_n\})$ by using mass distribution principle. However, an alternative proof
can be derived by using Lemmas \ref{upp} and \ref{lower}.

\begin{lemma}[{\cite[Lemma 2.3]{LR22}}]\label{LRlemma}
Assume that $u_n,v_n\to \infty$ as $n \to \infty$, and
\[
\limsup_{n \to \infty} \frac{v_n}{u_n} <\infty.
\]
Then
\[
\dim_{\rm H} F(\{u_n\}, \{v_n\}) = \liminf_{n \to \infty}\frac{\sum^n_{k=1}\log v_k}{2\sum^{n+1}_{k=1}\log u_k - \log v_{n+1}}.
\]
\end{lemma}

\section{Proof of the first statement of Theorem \ref{Alpha}}
%Recall that
%\begin{equation*}
%\alpha:=\lim_{k \to \infty} \frac{1}{n_k} \sum^k_{j=1} \log s_j\ \ \ \text{and} \ \ \ \beta:= \lim_{k \to \infty} \frac{\log s_k}{n_k}.
%\end{equation*}
In this section, we assume that $0<\alpha<\infty$. Let $\theta(\alpha,\beta)$ be the unique solution of the Diophantine pressure equation $\mathrm{P}(\theta)=(2\alpha -\beta)\theta - (\alpha -\beta)$. The aim is to prove
\[
\dim_{\rm H} E(\{n_k\},\{s_k\}, \{t_k\})=\dim_{\rm H} E_{L}(\{n_k\},\{s_k\}, \{t_k\})= \theta(\alpha,\beta).
\]
Note that $E(\{n_k\},\{s_k\}, \{t_k\}) \subseteq E_{L}(\{n_k\},\{s_k\}, \{t_k\})$, so the proof is divided into two parts:
the upper bound of $\dim_{\rm H} E_{L}(\{n_k\},\{s_k\}, \{t_k\})$ and the lower bound of $\dim_{\rm H} E(\{n_k\},\{s_k\}, \{t_k\})$.

\subsection{Upper bound}
For any $b,c \in [0,\infty)$, define
\begin{equation}\label{equ:XBC}
X(b,c):=\bigcap^\infty_{m=0} X_m(b,c),
\end{equation}
where $X_m(b,c)$ is given by
\begin{align*}
X_m(b,c):=\Big\{x\in \mathbb{I}:\ &\Pi^{(m)}_n(x) \geq e^{bn}\ \text{and}\ a_{n+1}(x) \geq e^{cn}\ \text{simultaneously}\\
& \text{for infinitely many $n \in \mathbb N$}\Big\}
\end{align*}
and
\[
\Pi^{(m)}_n(x):=\prod_{1\leq j\leq n, a_j(x)> e^m} a_j(x)
\]
denotes the product of the first $n$ coefficients of $x$ whose value is greater than $e^m$. For any $b=\infty$ or $c=\infty$, $X(b,c)$ is defined as
\[
X(\infty, c):= \bigcap^\infty_{M=1}X(M,c), \ \ \ \ \ X(b, \infty):= \bigcap^\infty_{N=1}X(b,N),
\]
and
\[
X(\infty, \infty):= \bigcap^\infty_{M=1}\bigcap^\infty_{N=1}X(M,N).
\]

\begin{theorem}\label{Xbc}
For any $b,c \in [0,\infty)$,
\begin{equation}\label{Xbcpressure}
\dim_{\rm H} X(b,c) = \Theta(b,c),
\end{equation}
where $\Theta(b,c)$ is the unique solution of the Diophantine pressure equation $\mathrm{P}(\theta)=b(2\theta-1)+c\theta$. Moreover, for any $b=\infty$ or $c=\infty$,
\begin{equation}\label{Xbconehalf}
\dim_{\rm H} X(b,c) = \frac{1}{2}.
\end{equation}
\end{theorem}

We first prove the upper bound of Theorem \ref{Xbc}. The proof for the lower bound will be provided after proving the lower bound of Theorem \ref{Alpha}.

\begin{proof}[Proof of the upper bound of Theorem \ref{Xbc}]
We remark that the upper bound of \eqref{Xbconehalf} can be derived from the upper bound of \eqref{Xbcpressure}.
In fact, when $0\leq c<\infty$, for any $G>0$, we see that $X(\infty,c)$ is a subset of $X(G,c)$, and so $\dim_{\rm H} X(\infty,c) \leq \Theta(G,c)$.
Since $\Theta(G,c) \to 1/2$ as $G\to \infty$, we obtain $\dim_{\rm H} X(\infty,c) \leq 1/2$; when $0\leq b<\infty$, for any $G>0$, we get that $\dim_{\rm H} X(b,\infty) \leq \Theta(b,G)$. Letting $G \to \infty$, we have $\dim_{\rm H} X(b,\infty) \leq 1/2$; when $b=c=\infty$, since $X(\infty,\infty) \subseteq X(1,\infty)$, we deduce that $\dim_{\rm H}X(\infty,\infty) \leq 1/2$.

Now it remains to show the upper bound of \eqref{Xbcpressure}. When $b=c=0$, we have $\dim_{\rm H} X(b,c) =1$, which is the unique solution of $\mathrm{P}(\theta)=0$. When $b,c \in [0,\infty)$ with $b+c>0$, let $\Theta(b,c)$ be the solution of $\mathrm{P}(\theta)=b(2\theta-1)+c\theta$. Then $\Theta(b,c)\in (1/2,1)$. For any $s >\Theta(b,c)$, we deduce that $\mathrm{P}(s)<b(2s-1)+cs$. Let $\varepsilon >0$ be small such that
\[
\varepsilon< \min\left\{\frac{2s-1}{s+2}, \frac{b(2s-1)+cs - \mathrm{P}(s)}{(s+2)b+2}\right\}.
\]
Then
\begin{equation}\label{ineq:sepsilon1}
(1-2s)+\varepsilon(s+2)<0
\end{equation}
and
\begin{equation}\label{ineq:sepsilon2}
\left((1-2s) + \varepsilon(s+2)\right)b + \mathrm{P}(s) -cs +2\varepsilon <0
\end{equation}
Choose $m_* >2$ large enough such that
\begin{equation}\label{mle}
\max\left\{\left(e\sqrt{m_*}\right)^{1/\sqrt{m_*}},\ \left(em_*\right)^{1/m_*},\ \left(8e^2\right)^{1/m_*}\right\} \leq e^\varepsilon.
\end{equation}
By the definition of the Diophantine pressure function $\mathrm{P}(s)$ in \eqref{pressuref}, there exists $K_\varepsilon>0$ such that for all $n \geq 1$,
\begin{equation}\label{P+}
\sum_{a_1,\ldots,a_n \in \mathbb{N}} q^{-2s}_n(a_1,\ldots,a_n) \leq K_\varepsilon e^{n\left(\mathrm{P}(s)+\varepsilon\right)}.
\end{equation}
We are going to compute the upper bound of the Hausdorff dimension of $X(b,c)$. Note that $X(b,c) \subseteq X_{m_{*}}(b,c)$ and
\begin{equation}\label{Xm}
X_{m_*}(b,c)= \bigcap^\infty_{N=1} \bigcup^\infty_{n=N} X_{m_*,n}(b,c),
\end{equation}
where $X_{m_*,n}(b,c)$ is defined by
\[
X_{m_*,n}(b,c):= \left\{x\in \mathbb{I}: \Pi^{(m_*)}_n(x) \geq e^{bn},\ a_{n+1}(x) \geq e^{cn}\right\}.
\]
It is sufficient to estimate the upper bound of the Hausdorff dimension of $X_{m_*}(b,c)$.
Since $b+c>0$, there exists $N_0 \in \mathbb N$ such that $\max\{e^{bn}, e^{cn}\}>e^{m_*}$ for all $n \geq N_0$.
Let $n\in \mathbb N$ with $n> N_0$ be fixed. For any $x\in X_{m_*,n}(b,c)$, there exist $1 \leq \ell \leq n$ and $1\leq j_1< \cdots<j_\ell \leq n$ such that $1\leq a_i(x) \leq e^{m_*}$ for all $1\leq i\leq n$ with $i \neq j_1, \ldots,j_\ell$; $a_{j}(x) >e^{m_*}$ for all $j=j_1,\dots,j_{\ell}$; $a_{j_1}(x)\cdots a_{j_\ell}(x) \geq e^{bn}$ and $a_{n+1}(x) \geq e^{cn}$. For any $1\leq k \leq \ell$, let $\lambda_k(x) := \lceil\log a_{j_k}(x)\rceil$. Then $\lambda_k(x) >m_*$, $\lambda_1(x) + \cdots +\lambda_\ell(x) > bn$ and $e^{\lambda_k(x)-1} <a_{j_k}(x) \leq e^{\lambda_k(x)}$. Hence
\begin{equation}\label{Xmn}
X_{m_*,n}(b,c) \subseteq \bigcup^n_{\ell =1}\bigcup_{\lambda > \max\{bn,m_*\ell\}} \bigcup_{1\leq j_1 <\cdots<j_\ell \leq n} \bigcup_{\substack{\lambda_1,\ldots,\lambda_\ell >m_*,\\ \lambda_1 + \cdots +\lambda_\ell =\lambda}} X^{\lambda_1,\ldots,\lambda_\ell}_{j_1,\ldots,j_\ell}(b,c),
\end{equation}
where $X^{\lambda_1,\ldots,\lambda_\ell}_{j_1,\ldots,j_\ell}(b,c)$ is defined by
\begin{align*}
X^{\lambda_1,\ldots,\lambda_\ell}_{j_1,\ldots,j_\ell}(b,c):= \Big\{x\in \mathbb{I}:\ & 1\leq a_i(x) \leq e^{m_*} \,(1\leq i \leq n\ \text{and}\ i \neq j_1, \ldots,j_\ell),\\
&e^{\lambda_k-1} <a_{j_k}(x) \leq e^{\lambda_k}\,(1 \leq k \leq\ell),\ a_{n+1}(x) \geq e^{cn}\Big\}.
\end{align*}

Before computing the upper bound of $\dim_{\rm H}X_{m_*}(b,c)$, we need some notations. For any $n,\ell,\lambda \in \mathbb N$ with $1\leq \ell \leq n$ and $\lambda> \max\{bn,m_*\ell\}$, let
\[
\mathcal{A}_{n,\ell}:=\left\{(j_1,\ldots, j_\ell) \in \mathbb{N}^\ell: 1\leq j_1<\cdots<j_\ell \leq n\right\}
\]
and
\[
\mathcal{B}_{\ell,\lambda}:=\left\{(\lambda_1,\ldots, \lambda_\ell) \in \mathbb{N}^\ell: \lambda_1,\ldots, \lambda_\ell>m_*,\ \lambda_1 + \cdots +\lambda_\ell =\lambda \right\}.
\]
For any $n,\ell,\lambda \in \mathbb N$ with $1\leq \ell \leq n$ and $\lambda > \max\{bn,m_*\ell\}$, $\boldsymbol{j}_\ell:=(j_1,\ldots, j_\ell)\in \mathcal{A}_{n,\ell}$ and $\boldsymbol{\lambda}_\ell:=(\lambda_1,\ldots, \lambda_\ell) \in \mathcal{B}_{\ell,\lambda}$, let
\begin{align*}
\mathcal{C}^{\boldsymbol{\lambda}_\ell}_{\boldsymbol{j}_\ell}(n):= \Big\{(\sigma_1,\ldots,\sigma_n)\in \mathbb{N}^n:\ & e^{\lambda_k-1} <\sigma_{j_k} \leq e^{\lambda_k}\,(1\leq k\leq\ell),\\
& 1\leq \sigma_i\leq e^{m_{*}}\,(i \neq j_1, \ldots,j_\ell)\Big\}.
\end{align*}
For any $(\sigma_1,\ldots,\sigma_n)\in \mathcal{C}^{\boldsymbol{\lambda}_\ell}_{\boldsymbol{j}_\ell}(n)$, let
\[
J_n(\sigma_1,\ldots,\sigma_n) := \bigcup_{j \geq e^{cn}} I_{n+1}(\sigma_1,\ldots,\sigma_n,j).
\]
Then for any $N >N_0$, by \eqref{Xm} and \eqref{Xmn}, we see that $X_{m_*}(b,c)$ is covered by
\begin{align}\label{Cover}
\bigcup^\infty_{n=N} \bigcup^n_{\ell =1}
\bigcup_{\lambda > \max\{bn,m_*\ell\}}
\bigcup_{\boldsymbol{j}_\ell \in \mathcal{A}_{n,\ell}}
\bigcup_{\boldsymbol{\lambda}_\ell \in \mathcal{B}_{\ell,\lambda}}
\bigcup_{(\sigma_1,\dots,\sigma_n)\in \mathcal{C}^{\boldsymbol{\lambda}_\ell}_{\boldsymbol{j}_\ell}(n)} \!\!J_n(\sigma_1,\dots,\sigma_n).
\end{align}

Now let us do some estimates of the cardinalities of $\mathcal{A}_{n,\ell}$ and $\mathcal{B}_{\ell,\lambda}$, and the diameter of $J_n(\sigma_1,\ldots,\sigma_n)$. Fix $n, \ell, \lambda \in \mathbb N$ with $1\leq \ell \leq n$ and $\lambda > \max\{bn, m_*\ell\}$.
\begin{itemize}
  \item  For the cardinality of $\mathcal{B}_{\ell,\lambda}$, we have $\sharp \mathcal{B}_{\ell,\lambda}$ is less than the number of positive integer solutions of $\sum^\ell_{k=1}\lambda_k=\lambda$. The latter is exactly $\binom{\lambda-1}{\ell-1} \leq \binom{\lambda}{\ell}<\frac{\lambda^\ell}{\ell!}$. Using the fact that $k!> (k/e)^k$ for all $k \in \mathbb N$, and combining this with \eqref{mle}, we derive that
  \begin{equation}\label{Bnum}
\sharp \mathcal{B}_{\ell,\lambda} < \frac{\lambda^\ell}{\ell!}<\left(\frac{e \lambda}{\ell}\right)^\ell < (em_*)^{\lambda/m_*} \leq e^{\varepsilon\lambda},
  \end{equation}
where the second to last inequality follows from $\ell < \lambda/m_*$ and the map $\ell \mapsto \left(\frac{e \lambda}{\ell}\right)^\ell$ is increasing in $(0,\lambda)$.

 \item For the cardinality of $\mathcal{A}_{n,\ell}$, we obtain $\sharp \mathcal{A}_{n,\ell} = \binom{n}{\ell}$. Note that $\ell< \lambda/m_*$,
  if $\lambda \leq n\sqrt{m_*}$, then $\ell < n/\sqrt{m_*}$. Being similar to the proof of \eqref{Bnum}, we deduce from \eqref{mle} that
 \[
  \binom{n}{\ell} < \left(\frac{e n}{\ell}\right)^\ell <(e\sqrt{m_*})^{n/\sqrt{m_*}} \leq e^{\varepsilon n};
  \]
  if $\lambda > n\sqrt{m_*}$, then $n < \lambda/\sqrt{m_*}$. By the binomial theorem, we have $\binom{n}{\ell} <2^n$, and so
 \[
  \binom{n}{\ell}<  e^n < e^{\lambda/\sqrt{m_*}} \leq e^{\varepsilon \lambda}.
  \]
  Therefore,
 \begin{align}\label{An2}
  \sharp \mathcal{A}_{n,\ell}< \max\left\{e^{\varepsilon n}, e^{\varepsilon \lambda}\right\}< e^{\varepsilon n}e^{\varepsilon \lambda}.
  \end{align}

  \item For any $\boldsymbol{j}_\ell =(j_1,\ldots, j_\ell)\in \mathcal{A}_{n,\ell}$, $\boldsymbol{\lambda}_\ell =(\lambda_1,\ldots, \lambda_\ell) \in \mathcal{B}_{\ell,\lambda}$ and $(\sigma_1,\ldots,\sigma_n)\in \mathcal{C}^{\boldsymbol{\lambda}_\ell}_{\boldsymbol{j}_\ell}(n)$, by Lemma \ref{q}, we see that
  \begin{align}\label{J1}
  \left|J_n(\sigma_1,\ldots,\sigma_n)\right| &\leq 2\cdot \left|I_n(\sigma_1,\ldots,\sigma_n)\right|\cdot \sum_{j \geq e^{cn}} \frac{1}{j(j+1)}  \nonumber\\
  &<2e^{-cn} \left|I_n(\sigma_1,\ldots,\sigma_n)\right|.
  \end{align}
  It follows from \eqref{cylinder} and \eqref{cysep} that
\begin{align*}
\left|I_n(\sigma_1,\ldots,\sigma_n)\right| &\leq 8^\ell \left(\prod^\ell_{k=1} \left|I_1(\sigma_{j_k})\right|\right) \left|I_{n-\ell}(\tau_1,\ldots,\tau_{n-\ell})\right|\\
&\leq 8^\ell \left(\prod^\ell_{k=1} \sigma^{-2}_{j_k}\right) q^{-2}_{n-\ell} (\tau_1,\ldots,\tau_{n-\ell})\\
&< 8^\ell \left(\prod^\ell_{k=1} e^{-2(\lambda_k-1)}\right) q^{-2}_{n-\ell} (\tau_1,\ldots,\tau_{n-\ell})\\
&= \left(8e^2\right)^\ell e^{-2\lambda}q^{-2}_{n-\ell} (\tau_1,\ldots,\tau_{n-\ell}),
\end{align*}
where $(\tau_1,\ldots,\tau_{n-\ell})$ denotes the sequence obtained by deleting $\sigma_{j_1},\dots,\sigma_{j_\ell}$ from $(\sigma_1,\ldots,\sigma_n)$, namely $1\leq \tau_k \leq e^{m_*}$ for all $1\leq k\leq n-\ell$. Combining this with \eqref{J1}, we get that
  \begin{align}\label{J2}
  \left|J_n(\sigma_1,\ldots,\sigma_n)\right| \leq 2\left(8e^2\right)^\ell e^{-cn} e^{-2\lambda}q^{-2}_{n-\ell} (\tau_1,\ldots,\tau_{n-\ell}).
 \end{align}
\end{itemize}

We are ready to estimate the upper bound of the Hausdorff dimension of $X_m(b,c)$.
By \eqref{Cover}, we deduce that $\mathcal{H}^s(X_{m_{*}}(b,c))$ is bounded above by
\[
\liminf_{N\to \infty} \sum^\infty_{n=N} \sum^n_{\ell =1}
\sum_{\lambda> \max\{bn,m_*\ell\}}
\sum_{\boldsymbol{j}_\ell \in \mathcal{A}_{n,\ell}}
\sum_{\boldsymbol{\lambda}_\ell \in \mathcal{B}_{\ell,\lambda}}
\sum_{(\sigma_1,\ldots,\sigma_n)\in \mathcal{C}^{\boldsymbol{\lambda}_\ell}_{\boldsymbol{j}_\ell}(n)}\left|J_n(\sigma_1,\ldots,\sigma_n)\right|^s.
\]
Let
\[
\textcircled{1}:=\sum_{(\sigma_1,\ldots,\sigma_n)\in \mathcal{C}^{\boldsymbol{\lambda}_\ell}_{\boldsymbol{j}_\ell}(n)}\left|J_n(\sigma_1,\ldots,\sigma_n)\right|^s.
\]
It follows from \eqref{J2} that
\begin{align}\label{cir1sum}
\textcircled{1} \leq \sum_{(\sigma_1,\cdots,\sigma_n)\in \mathcal{C}^{\boldsymbol{\lambda}_\ell}_{\boldsymbol{j}_\ell}(n)} 2^s\left(8e^2\right)^{\ell s}e^{-csn} e^{-2\lambda s}q^{-2s}_{n-\ell} (\tau_1,\ldots,\tau_{n-\ell}).
\end{align}
Since $e^{\lambda_k-1} <\sigma_{j_k} \leq e^{\lambda_k}$ for all $1\leq k\leq \ell$ and $1\leq \sigma_i \leq e^{m_*}$ for all $i \neq j_1,\dots,j_k$,
 we see that the right-hand side of \eqref{cir1sum} is bounded above by
\begin{align*}
&2^s\left(8e^2\right)^{\ell s}e^{-csn} \prod^\ell_{k=1} \left(\sum_{e^{\lambda_k-1} <\sigma_{j_k} \leq e^{\lambda_k}}e^{-2\lambda s}\right) \sum_{1\leq \tau_1,\ldots,\tau_{n-\ell} \leq e^{m_*}}q^{-2s}_{n-\ell} (\tau_1,\ldots,\tau_{n-\ell})\\
&<2^s\left(8e^2\right)^{\ell s}e^{-csn} e^{(1-2s)\lambda} \sum_{\tau_1,\ldots,\tau_{n-\ell} \in \mathbb N}q^{-2s}_{n-\ell} (\tau_1,\ldots,\tau_{n-\ell}).
\end{align*}
Since $\lambda > m_*\ell$, i.e., $\ell < \lambda/m_*$, we deduce from \eqref{mle} that $(8e^2)^{\ell} <(8e^2)^{\lambda/m_*} \leq e^{\varepsilon\lambda}$.
Combining these with \eqref{P+} and \eqref{cir1sum}, we see that
\begin{equation*}
\textcircled{1} < 2^s K_\varepsilon e^{(1-2s+\varepsilon s)\lambda}e^{n(\mathrm{P}(s)-cs+\varepsilon)}.
\end{equation*}
In view of \eqref{Bnum} and \eqref{An2}, we get that
\begin{align*}
\sum_{\boldsymbol{j}_\ell \in \mathcal{A}_{n,\ell}}\sum_{\boldsymbol{\lambda}_\ell \in \mathcal{B}_{\ell,\lambda}} \textcircled{1}\
&< e^{\varepsilon\lambda}\cdot e^{\varepsilon n}e^{\varepsilon \lambda}\cdot 2^s K_\varepsilon e^{(1-2s+\varepsilon s)\lambda}e^{n(\mathrm{P}(s)-cs+\varepsilon)} \\
&=2^s K_\varepsilon e^{\left(1-2s+\varepsilon(s+2)\right)\lambda}e^{n(\mathrm{P}(s)-cs+2\varepsilon)}.
\end{align*}
This together with \eqref{ineq:sepsilon1} gives that
\begin{align*}
\sum_{\lambda > \max\{bn,m_*\ell\}}\sum_{\boldsymbol{j}_\ell \in \mathcal{A}_{n,\ell}}\sum_{\boldsymbol{\lambda}_\ell \in \mathcal{B}_{\ell,\lambda}} \textcircled{1}\ &\leq 2^s K_\varepsilon e^{n(\mathrm{P}(s)-cs+2\varepsilon)}\sum_{\lambda > \max\{bn,m_*\ell\}} e^{\left(1-2s+\varepsilon(s+2)\right)\lambda}\\
& \leq 2^s K_\varepsilon e^{n(\mathrm{P}(s)-cs+2\varepsilon)}\sum_{\lambda > bn} e^{\left(1-2s+\varepsilon(s+2)\right)\lambda} \\
&\leq K_\varepsilon^* e^{n\left(\mathrm{P}(s)-cs+2\varepsilon +\left(1-2s+\varepsilon(s+2)\right)b\right)},
\end{align*}
where $K_\varepsilon^*>0$ is a constant only depending on $s,\varepsilon$ and $b$ (and independent of $n$). By \eqref{ineq:sepsilon2}, we conclude that
\begin{align*}
\mathcal{H}^s(X_{m_{*}}(b,c)) &\leq \liminf_{N\to \infty} \sum^\infty_{n=N} \sum^n_{\ell =1} K_\varepsilon^* e^{n\left(\mathrm{P}(s)-cs+2\varepsilon +\left(1-2s+\varepsilon(s+2)\right)b\right)}\\
&\leq K_\varepsilon^* \cdot \liminf_{N\to \infty}\sum^\infty_{n=N} ne^{n\left(\mathrm{P}(s)-cs+2\varepsilon +\left(1-2s+\varepsilon(s+2)\right)b\right)}=0,
\end{align*}
which implies that $\dim_{\rm H}X_{m_*}(b,c) \leq s$. Then $\dim_{\rm H}X(b,c) \leq s$. Since $s> \Theta(b,c)$ is arbitrary, we obtain
\[
\dim_{\rm H}X(b,c) \leq \Theta(b,c).
\]
\end{proof}

Now we are able to prove the upper bound of Theorem \ref{Alpha}. Recall that
\begin{align*}
E_{L}(\{n_k\},\{s_k\}, \{t_k\}) =\big\{x\in \mathbb{I}: s_k < a_{n_k}(x) \leq s_k+ t_k \ \text{for all large $k\in \mathbb N$}\big\}.
\end{align*}

\begin{lemma}\label{Eularge}
Under the hypotheses (H1), (H2) and (H3), we have
\[
\dim_{\rm H} E_{L}(\{n_k\},\{s_k\}, \{t_k\})  \leq \theta(\alpha,\beta).
\]
\end{lemma}
\begin{proof}
For any $K\in \mathbb N$, let
\[
E_{L}(\{n_k\},\{s_k\}, \{t_k\},K):=\big\{x\in \mathbb{I}: s_k < a_{n_k}(x) \leq s_k+ t_k \  \forall k \geq K\big\}.
\]
Then $E_{L}(\{n_k\},\{s_k\}, \{t_k\}) = \bigcup^\infty_{K=1}E_{L}(\{n_k\},\{s_k\}, \{t_k\},K)$, and so
\begin{align}\label{ElargeK}
\dim_{\rm H} E_{L}(\{n_k\},\{s_k\}, \{t_k\}) = \sup_{K\geq 1}\big\{\dim_{\rm H} E_{L}(\{n_k\},\{s_k\}, \{t_k\},K)\big\}.
\end{align}
This leads to estimating the upper bound of $\dim_{\rm H} E_{L}(\{n_k\},\{s_k\}, \{t_k\},K)$ for all $K \geq 1$. In the following, we only consider the case $K=1$. Then the proofs for other cases $K\geq 2$ can be done by defining $\widetilde{n}_k:=n_{k+K-1}$, $\widetilde{s}_k:=t_{k+K-1}$ and $\widetilde{s}_k:=t_{k+K-1}$.

By \eqref{important}, we deduce that
\begin{equation}\label{alpha-beta}
\lim_{k \to \infty} \frac{1}{n_{k+1}-1}\sum^k_{j=1} \log s_j= \alpha -\beta.
\end{equation}

(i) When $\alpha -\beta>0$ with $\beta>0$, for any $0<\varepsilon<\min\{\alpha-\beta,\beta\}$, we claim that $E_{L}(\{n_k\},\{s_k\}, \{t_k\},1)$ is a subset of $X(\alpha-\beta-\varepsilon, \beta-\varepsilon)$. To see this, for any $m \geq 0$, note that $s_k \to \infty$ as $k \to \infty$, so there exists $N_0 \in \mathbb N$ such that $s_k > e^m$ for all $k \geq N_0$. In view of \eqref{important} and \eqref{alpha-beta}, there exists an integer $N_1>N_0$ such that for all $k \geq N_1$,
\[
s_{N_0}\cdots s_k \geq e^{(\alpha-\beta-\varepsilon)(n_{k+1}-1)}\ \ \text{and}\ \ s_{k+1} \geq e^{(\beta-\varepsilon)n_{k+1}}.
\]
For any $x\in E_{L}(\{n_k\},\{s_k\}, \{t_k\}, 1)$, we obtain $a_{n_k}(x)>s_k$ for all $k \in \mathbb N$, and so for any $k \geq N_1$,
\[
\Pi^{(m)}_{n_{k+1}-1}(x) \geq a_{n_{N_0}}(x)\cdots a_{n_k}(x) > s_{N_0}\cdots s_k \geq e^{(\alpha-\beta-\varepsilon)(n_{k+1}-1)}
\]
and
\[
a_{n_{k+1}}(x)>s_{k+1}>e^{(\beta-\varepsilon)n_{k+1}},
\]
which means that $x\in X_m(\alpha-\beta-\varepsilon, \beta-\varepsilon)$. Hence $x \in X(\alpha-\beta-\varepsilon, \beta-\varepsilon)$. By Theorem \ref{Xbc}, we see that
\[
\dim_{\rm H}E_{L}(\{n_k\},\{s_k\}, \{t_k\}, 1)\leq \Theta(\alpha-\beta-\varepsilon, \beta-\varepsilon).
\]
Since $\Theta(\alpha-\beta-\varepsilon, \beta-\varepsilon) \to \theta(\alpha,\beta)$ as $\varepsilon \to 0$, we have $\dim_{\rm H}E_{L}(\{n_k\},\{s_k\}, \{t_k\},1) \leq \theta(\alpha,\beta)$.

(ii) When $\alpha -\beta>0$ with $\beta=0$, for any $0<\varepsilon <\alpha$, we deduce from \eqref{alpha-beta} that $E_{L}(\{n_k\},\{s_k\}, \{t_k\}, 1)$ is a subset of $X(\alpha-\varepsilon, 0)$. It follows from Theorem \ref{Xbc} that
\[
\dim_{\rm H}E_{L}(\{n_k\},\{s_k\}, \{t_k\},1) \leq \Theta(\alpha-\varepsilon, 0).
\]
Sine $\Theta(\alpha-\varepsilon,0) \to \theta(\alpha,0)$ as $\varepsilon \to 0$, we obtain $\dim_{\rm H}E_{L}(\{n_k\},\{s_k\}, \{t_k\}, 1) \leq \theta(\alpha, 0)$.

(iii) When $\alpha -\beta=0$ with $\beta>0$, for any $0<\varepsilon<\alpha$, $E_{L}(\{n_k\},\{s_k\}, \{t_k\},1)$ is a subset of $X(0, \alpha-\varepsilon)$. By Theorem \ref{Xbc}, we get that
\[
\dim_{\rm H}E_{L}(\{n_k\},\{s_k\}, \{t_k\}, 1) \leq \Theta(0, \alpha-\varepsilon).
\]
Letting $\varepsilon \to 0$ yields that $\dim_{\rm H}E_{L}(\{n_k\},\{s_k\}, \{t_k\},1)\leq \theta(\alpha, \alpha)$.

In summary, $\dim_{\rm H}E_{L}(\{n_k\},\{s_k\}, \{t_k\},1)\leq \theta(\alpha, \beta)$. Therefore, for all $K \in \mathbb N$, $\dim_{\rm H}E_{L}(\{n_k\},\{s_k\}, \{t_k\},K)\leq \theta(\alpha, \beta)$. By \eqref{ElargeK}, we conclude that
\[
\dim_{\rm H}E_{L}(\{n_k\},\{s_k\}, \{t_k\})\leq \theta(\alpha, \beta)
\]
\end{proof}

\subsection{Lower bound}\label{PFA}
In this subsection, we prove the lower bound of Theorem \ref{Alpha}. To this end, we first give an auxiliary result. For any $k,m \in \mathbb{N}$, let
\[
\mathcal{N}_m(k):=\left\{(\sigma_1,\ldots,\sigma_k)\in \mathbb{N}^k: \frac{1}{2^{m}} \leq |I_k(\sigma_1,\ldots,\sigma_k)|<\frac{1}{2^{m-1}}\right\}.
\]
Denote by $N_m(k) := \sharp \mathcal{N}_m(k)$ the cardinality of $\mathcal{N}_m(k)$. The following lemma establishes a relation of $N_m(k)$ and the Diophantine pressure function $\mathrm{P}(\theta)$.

\begin{lemma}\label{NP}
For any $\theta \in (1/2, 1)$ and $0<\varepsilon <\mathrm{P}(\theta)$, and for sufficiently large $k$, there exists $m \in \mathbb{N}$ such that
\begin{equation*}
N_m(k) > 2^{(m+1)\theta}e^{\left(\mathrm{P}(\theta)-\varepsilon\right)k}.
\end{equation*}
\end{lemma}

\begin{proof}
We proceed by contradiction. Suppose that there exist $\theta_* \in (1/2,1)$, $0<\varepsilon_0<\mathrm{P}(\theta_*)$ and a strictly increasing sequence $\{k_i\}$ of positive integers such that for all $m \in \mathbb N$ and for all $i \in \mathbb N$,
\[
N_m(k_i) \leq 2^{\left(m+1\right)\theta_*}e^{\left(\mathrm{P}(\theta_*)-\varepsilon_0\right)k_i}.
\]
For any small $0<\delta<1/2-\theta_*$ and for all $i \in \mathbb N$, we see that
\begin{align*}
\sum_{(a_1,\ldots,a_{k_i}) \in \mathbb{N}^{k_i}} |I_{k_i}(a_1,\ldots,a_{k_i})|^{\theta_*+\delta} &\leq \sum^\infty_{m=1} N_m(k_i)\cdot 2^{-(\theta_*+\delta)(m-1)}\\
&\leq 2^{\theta_*+\delta} e^{(\mathrm{P}(\theta_*)-\varepsilon_0)k_i}
\sum^\infty_{m=1}2^{(m+1)\theta_*}\cdot 2^{-(\theta_*+\delta)m}\\
& = \frac{2^{2\theta_*+\delta}}{2^\delta -1}e^{(\mathrm{P}(\theta_*)-\varepsilon_0)k_i},
\end{align*}
which implies that
\[
\frac{1}{k_i}\log \sum_{(a_1,\ldots,a_{k_i}) \in \mathbb{N}^{k_i}} |I_{k_i}(a_1,\ldots,a_{k_i})|^{\theta_*+\delta} \leq \mathrm{P}(\theta_*)-\varepsilon_0 + \frac{\log(\frac{2^{2\theta_*+\delta}}{2^\delta -1})}{k_i}.
\]
By the definition of $\mathrm{P}(\cdot)$, we have $\mathrm{P}(\theta_*+\delta) \leq \mathrm{P}(\theta_*)-\varepsilon_0$. Since the Diophantine pressure function is continuous, we obtain $\mathrm{P}(\theta_*) \leq \mathrm{P}(\theta_*)-\varepsilon_0$. This is a contradiction.
\end{proof}

We are ready to give the proof for the lower bound of Theorem \ref{Alpha}. To this end, we first introduce some symbols and notations.

\subsubsection{Symbols and notations}
For simplicity, let $\theta_0:=\theta(\alpha,\beta)$. Then $1/2<\theta_0<1$, and $\mathrm{P}(\theta_0)=(2\alpha -\beta)\theta_0 - (\alpha -\beta)$. By Lemma \ref{NP}, for any $0<\varepsilon<\mathrm{P}(\theta_0)$, there are positive integers $m_0>1$ and $k_0>1$ such that
\begin{equation*}\label{k0n0}
N_{m_0}(k_0) > 2^{(m_0+1)\theta_0}e^{(\mathrm{P}(\theta_0)-\varepsilon)k_0}.
\end{equation*}
Denote
\[
U:= \frac{1}{k_0}\log N_{m_0}(k_0)\ \ \ \text{and}\ \ \ V:=  \frac{m_0+1}{2k_0}\log 2.
\]
Then $U>2V\theta_0+\mathrm{P}(\theta_0)-\varepsilon$, and so
\begin{equation}\label{UVP}
U+(\alpha -\beta) +\varepsilon > (2V+2\alpha-\beta)\theta_0.
\end{equation}

Note that $\{n_k\}$ is a strictly increasing sequence of positive integers. Let $n_0:=0$.
For any $k\in \mathbb N$, we write $n_k-n_{k-1}-1$ as
\begin{equation}\label{nk-nk-1}
n_k-n_{k-1}-1=\ell_kk_0+r_k,
\end{equation}
where $\ell_k \geq 0$ and $0\leq r_k<k_0$ are integers. Then
\begin{equation*}\label{ellkest}
\frac{n_k-n_{k-1}}{k_0}-1\leq \ell_k \leq \frac{n_k-n_{k-1}-1}{k_0}.
\end{equation*}
For any $k \in \mathbb N$, let $\widehat{\mathcal{W}}_k$ be the set of sequences of the form
\[
(\underbrace{1,\dots,1}_{r_k},j)
\]
with $s_k<j\leq s_k+t_k$, namely
\[
\widehat{\mathcal{W}}_k:=\left\{(1,\dots,1,j) \in \mathbb N^{r_k+1}: s_k < j \leq s_k+t_k\right\}.
\]
For $\sigma=(\sigma_1,\dots,\sigma_m)$ and $\tau=(\tau_1,\dots,\tau_n)$, we denote by $\sigma*\tau$ the concatenation of $\sigma$ and $\tau$, namely $\sigma*\tau=(\sigma_1,\dots,\sigma_m,\tau_1,\dots,\tau_n)$. For any $k \in \mathbb N$, let
\[
\mathcal{W}_{n_{k}-n_{k-1}}:=\left\{w_1*\cdots *w_{\ell_k}* w_{\ell_k+1}: w_i\in \mathcal{N}_{m_0}(k_0)\,(i=1,\dots,\ell_k),w_{\ell_k+1} \in \widehat{\mathcal{W}}_k \right\}.
\]
We remark that each sequence of $\mathcal{W}_{n_{k}-n_{k-1}}$ is of the form $w_1*\cdots *w_{\ell_k}* w_{\ell_k+1}$, where $w_1,\dots,w_{\ell_k}$ are from $\mathcal{N}_{m_0}(k_0)$ and $w_{\ell_k+1}$ belongs to $\widehat{\mathcal{W}}_k$. Hence the length of the sequence of $\mathcal{W}_{n_{k}-n_{k-1}}$ is $n_k-n_{k-1}$. Treating each $w_i$ as a block, we know that each sequence of $\mathcal{W}_{n_{k}-n_{k-1}}$ is divided into $\ell_k+1$ blocks, where the first $\ell_k$ blocks are of length $k_0$ (called \textbf{normal blocks}), and the last block is of length $r_k+2$ (called a \textbf{exceptional block}). Therefore, each sequence of $\mathcal{W}_{n_{k}-n_{k-1}}$ must have an exception block, whereas it may not contain a normal block. By definition, we deduce that the cardinality of $\mathcal{W}_{n_{k}-n_{k-1}}$ is
\begin{align}\label{Wnk-nk-1}
\sharp \mathcal{W}_{n_{k}-n_{k-1}} &= \left(N_{m_0}(k_0)\right)^{\ell_k} \left(\lfloor s_k+t_k\rfloor -\lfloor s_k\rfloor\right)\nonumber \\
&\geq \left(N_{m_0}(k_0)\right)^{\frac{n_k-n_{k-1}}{k_0}-1}\lfloor t_k\rfloor.
\end{align}
For any $k\in \mathbb N$, let
\[
\mathcal{W}_{n_{k}}:=\left\{W_{n_1}*W_{n_2-n_1}*\cdots *W_{n_k-n_{k-1}}: W_{n_{i}-n_{i-1}} \in \mathcal{W}_{n_{i}-n_{i-1}}\,(i=1,\dots,k)\right\}.
\]
We point out that each sequence of $\mathcal{W}_{n_{k}}$ is the concatenation of the sequences taken from $W_{n_1},W_{n_2-n_1},\dots,W_{n_k-n_{k-1}}$ respectively, and so it is of length $n_k$. Moreover, there are $\ell_1+\dots+\ell_{k-1}+\ell_k+k$ blocks in each sequence of $\mathcal{W}_{n_{k}}$. The illustration for the block structure of the sequence of $\mathcal{W}_{n_{k}}$ is as follows. We use the position of digits to chunk for the convenience of counting.
\begin{figure}[H]
  \centering
  \includegraphics[width=12cm, height=4.65cm]{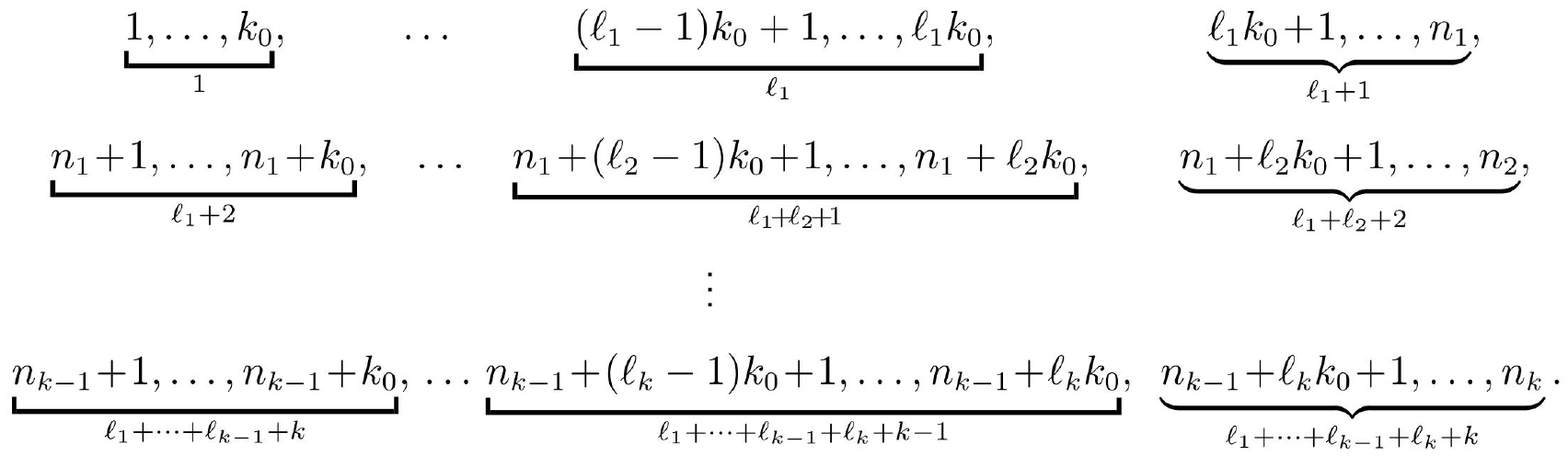}\\
 \caption{The block structure of the sequence of $\mathcal{W}_{n_{k}}$. The symbol `underbracket' represents a normal block, while the symbol `underbrace' represents a exceptional block.}
\end{figure}
\noindent Here the length of the bracket is $k_0$, and the length of the brace is $r_i+1$ for some $1\leq i\leq k$. The number below the brackets or the braces stands for the number of blocks. Moreover, the $(\ell_1+\cdots+\ell_i+i)$-th block is from $\widehat{\mathcal{W}}_i$ for all $1\leq i \leq k$, and other blocks belong to $\mathcal{N}_{m_0}(k_0)$. In particular, the digit at the position $n_i$ belongs to $(s_i, s_i+t_i]$ for all $1\leq i\leq k$, and the digits at other positions are bounded. By \eqref{Wnk-nk-1}, we derive that the cardinality of $\mathcal{W}_{n_{k}}$ is
\begin{align}\label{Wnk}
\sharp \mathcal{W}_{n_{k}} &= \sharp \mathcal{W}_{n_{1}} \times\sharp\mathcal{W}_{n_{2}-n_1}\times \cdots \times\sharp \mathcal{W}_{n_{k}-n_{k-1}}\nonumber\\
 &\geq \left(N_{m_0}(k_0)\right)^{\frac{n_k}{k_0}-k} \left(\lfloor t_1\rfloor \cdots \lfloor t_k\rfloor\right).
\end{align}
For any $k \in \mathbb N$ and $1\leq \ell \leq \ell_{k+1}$, let
\[
\mathcal{W}_{n_{k}+\ell k_0}:=\big\{W_{n_k}*w_{1}*\cdots*w_{\ell}: W_{n_k}\in \mathcal{W}_{n_{k}}, w_{i} \in \mathcal{N}_{m_0}(k_0)\,(i=1,\dots,\ell)\big\}.
\]
Then each sequence of $\mathcal{W}_{n_{k}+\ell k_0}$ is the concatenation of the sequence of $\mathcal{W}_{n_{k}}$ and $\ell$ sequences from $\mathcal{N}_{m_0}(k_0)$. By definition, $\mathcal{W}_{n_{k+1}}$ is the set of sequences of the form $W_{n_k + \ell_{k+1}k_0}*w_{\ell_{k+1}+1}$ for all $W_{n_k + \ell_{k+1}k_0} \in \mathcal{W}_{n_k + \ell_{k+1}k_0}$ and for all $w_{\ell_{k+1}+1} \in \widehat{\mathcal{W}}_{k+1}$. In view of \eqref{Wnk}, we get that the cardinality of $\mathcal{W}_{n_{k}+\ell k_0}$ is
\begin{align}\label{Wnkell}
\sharp \mathcal{W}_{n_{k}+\ell k_0}&= \sharp \mathcal{W}_{n_{k}}\times \left(N_{m_0}(k_0)\right)^\ell\nonumber \\
& \geq \left(N_{m_0}(k_0)\right)^{\frac{n_k+\ell k_0}{k_0}-k} \left(\lfloor t_1\rfloor \cdots \lfloor t_k\rfloor\right).
\end{align}

Now we are going to give the proof for the lower bound of $\dim_{\rm H}E(\{n_k\},\{s_k\},\{t_k\})$ by Lemma \ref{lower}.
The process is divided into the following three steps.

\subsubsection{Step I: Constructing the sequence $\{\mathrm{E}_n\}$ of subsets in $[0,1]$}
 Denote by $M_0$ the maximum of all digits of the sequences in $\mathcal{N}_{m_0}(k_0)$. For any $k \in \mathbb N$, let $L_0:=0$ and
\[
L_k:=\ell_1+\dots+\ell_k+k.
\]
Then $\{L_k\}$ is a strictly increasing sequence of positive integers. For any $n\in \mathbb N$, there exists a unique $k \in \mathbb N$ such that $L_{k-1}\leq  n<L_{k}$. The basic intervals of order $n$ and $\mathrm{E}_n$ are defined as follows.

(I-i) When $L_{k-1}\leq n< L_k-1$, that is, $0\leq n-L_{k-1} \leq \ell_k-1$, we know that each sequence of $\mathcal{W}_{n_{k-1}+(n-L_{k-1}) k_0}$ is divided into $n$ blocks, and the $(n+1)$th block is a normal block. For any $\sigma \in \mathcal{W}_{n_{k-1}+(n-L_{k-1})k_0}$, we define the basic interval of order $n$ as
 \[
  J_{n_{k-1}+(n-L_{k-1})k_0}(\sigma):= \bigcup^{M_0}_{j=1} cl (I_{n_{k-1}+(n-L_{k-1})k_0+1}(\sigma*j)),
  \]
  where $cl(\cdot)$ denotes the closure of a set, and
 \[
 \mathrm{E}_n:= \bigcup_{\sigma \in \mathcal{W}_{n_{k-1}+(n-L_{k-1})k_0}}  J_{n_{k-1}+(n-L_{k-1})k_0}(\sigma);
  \]

(I-ii) When $n=L_k-1$, that is, $n-L_{k-1} = \ell_k$, since the $(n+1)$th block is an exceptional block, the definition of the basic interval of order $n$ is different for $r_k>0$ and $r_k=0$. For any $\sigma \in \mathcal{W}_{n_{k-1}+\ell_kk_0}$,
\begin{itemize}
  \item when $r_k>0$, we define the basic interval of order $n$ as
  \[
  J_{n_{k-1}+\ell_kk_0}(\sigma):= cl (I_{n_{k-1}+\ell_kk_0+1}(\sigma*1));
  \]

 \item  when $r_k=0$, that is, $n_k=n_{k-1}+\ell_kk_0+1$, we define the basic interval of order $n$ as
  \[
  J_{n_{k-1}+\ell_kk_0}(\sigma):= \bigcup_{s_k<j\leq s_k+t_k} cl (I_{n_{k}}(\sigma*j)).
  \]
 \end{itemize}
 In this case, $\mathrm{E}_n$ is defined by
 \[
 \mathrm{E}_n:= \bigcup_{\sigma \in \mathcal{W}_{n_{k-1}+\ell_kk_0}}  J_{n_{k-1}+\ell_kk_0}(\sigma);
  \]

Hence $\{\mathrm{E}_n\}$ is a decreasing sequence of subsets of $\mathrm{E}_0:=[0,1]$. Let
\begin{equation}\label{Lemma25E}
\mathrm{E}= \bigcap^\infty_{n=0}\mathrm{E}_n.
\end{equation}
By the construction of $\mathrm{E}_n$, we obtain that $\mathrm{E}$ is a subset of $E(\{n_k\},\{s_k\},\{t_k\})$, and so $\dim_{\rm H}E(\{n_k\},\{s_k\},\{t_k\}) \geq \dim_{\rm H}\mathrm{E}$. Therefore, it remains to compute the lower bound of $\dim_{\rm H}\mathrm{E}$. To this end, by Lemma \ref{lower}, we need to estimate the values of $m_n$ and $\varepsilon_n$.

\subsubsection{Step II: Estimating the values of $m_n$ and $\varepsilon_n$}
 Let us first recall the definitions of $m_n$ and $\varepsilon_n$ in Lemma \ref{lower}: $m_n$ is the smallest number of basic intervals of order $n$ that are contained in the basic interval of order $n-1$, and $\varepsilon_n$ is the smallest number of the lengths of gaps between these basic intervals of order $n$. We propose to compute the lower bound of the liminf in Lemma \ref{lower}, so it is sufficient to estimate the lower bounds of $m_n$ and $\varepsilon_n$

(II-i) When $L_{k-1} \leq n-1<L_{k}-2$ for some $k$, according to the discussions in (I-i), we know that the basic interval of order $n-1$ is of the form $J_{n_{k-1}+\ell k_0}(\sigma)$ with $\ell:=n-1-L_{k-1}$ for some $\sigma \in \mathcal{W}_{n_{k-1}+\ell k_0}$, and contains the basic intervals of order $n$:
 \begin{equation}\label{baseintern}
J_{n_{k-1}+(\ell+1) k_0}(\sigma*\tau)=\bigcup^{M_0}_{j=1} cl (I_{n_{k-1}+(\ell+1) k_0+1}(\sigma*\tau*j))
  \end{equation}
for all $\tau \in \mathcal{N}_{m_0}(k_0)$. Hence $m_n=\sharp \mathcal{N}_{m_0}(k_0) = N_{m_0}(k_0)$. For the gaps between these basic intervals of order $n$, let $\tau^\prime=(b_1,\dots,b_{k_0-1},b_{k_0})$ and $\tau=(b_1,\dots,b_{k_0-1},b_{k_0}+1)$ be sequences of $\mathcal{N}_{m_0}(k_0)$. By the structure of the basic interval of order $n$ as in \eqref{baseintern}, we see that the cylinder $I_{n_{k-1}+(\ell+1) k_0+1}(\sigma*\tau*(M_0+1))$ is within the gap between $J_{n_{k-1}+(\ell+1) k_0}(\sigma*\tau)$ and $J_{n_{k-1}+(\ell+1) k_0}(\sigma*\tau^\prime)$. Then $\varepsilon_n$ is greater than the diameter of $I_{n_{k-1}+(\ell+1) k_0+1}(\omega*(M_0+1))$ for all $\omega\in \mathcal{W}_{n_{k-1}+(\ell+1) k_0}$, namely
   \begin{equation}\label{epsilonnfirst}
 \varepsilon_n >\min_{\omega \in \mathcal{W}_{n_{k-1}+(\ell+1) k_0}}\left\{\left|I_{n_{k-1}+(\ell+1) k_0+1}(\omega*(M_0+1))\right|\right\}.
  \end{equation}

 (II-ii) When $n-1=L_k-2$ for some $k$, we deduce from the conclusions in (I-i) and (I-ii) that the basic interval of order $n-1$ is of the form $J_{n_{k-1}+(\ell_k-1) k_0}(\sigma)$ for some $\sigma \in \mathcal{W}_{n_{k-1}+(\ell_k-1) k_0}$, and contains the basic intervals of order $n$:
 \begin{equation}\label{IIiirk>0}
J_{n_{k-1}+\ell_kk_0}(\sigma*\tau)= cl (I_{n_{k-1}+\ell_kk_0+1}(\sigma*\tau*1))\ \ \ \ \ \ \text{as $r_k>0$}
  \end{equation}
  or
 \begin{equation}\label{IIiirk=0}
J_{n_{k-1}+\ell_kk_0}(\sigma*\tau)=\bigcup_{s_k<j\leq s_k+t_k} cl (I_{n_{k}}(\sigma*\tau*j))\ \ \ \ \ \text{as $r_k=0$}
  \end{equation}
for all $\tau \in \mathcal{N}_{m_0}(k_0)$. By definition, we have $m_n=\sharp \mathcal{N}_{m_0}(k_0) = N_{m_0}(k_0)$. For $\varepsilon_n$, we assert that for all large $n$,
\begin{equation}\label{IIiiepsn}
 \varepsilon_n> \min_{\omega \in \mathcal{W}_{n_{k-1}+\ell_k k_0}}\left\{\left|I_{n_{k-1}+\ell_k k_0+1}(\omega*(M_0+1))\right|\right\}.
\end{equation}
To see this, when $r_k>0$, following the similar analyses in (II-i), it follows from \eqref{IIiirk>0} that the inequality \eqref{IIiiepsn} holds; when $r_k=0$, that is, $n_k=n_{k-1}+\ell_kk_0+1$, for any $\tau^\prime=(b_1,\dots,b_{k_0-1},b_{k_0}),\tau=(b_1,\dots,b_{k_0-1},b_{k_0}+1) \in \mathcal{N}_{m_0}(k_0)$, since $s_k > M_0+1$ for large $k$, we deduce from \eqref{IIiirk=0} that $I_{n_{k-1}+\ell_k k_0+1}(\sigma*\tau^\prime*(M_0+1))$ is within the gap between $J_{n_{k-1}+\ell_k k_0}(\sigma*\tau)$ and $J_{n_{k-1}+\ell_kk_0}(\sigma*\tau^\prime)$. Then $\varepsilon_n$ is greater than the diameter of $I_{n_{k-1}+\ell_kk_0+1}(\omega*(M_0+1))$ for all $\omega\in \mathcal{W}_{n_{k-1}+\ell_kk_0}$, and so \eqref{IIiiepsn} is true.

Combining (II-i) and (II-ii), when $L_{k-1} \leq n-1 \leq L_{k}-2$ for some $k$, that is, $L_{k-1}<n< L_{k}$, we have
\begin{itemize}
  \item $m_n=\sharp \mathcal{N}_{m_0}(k_0) = N_{m_0}(k_0)$;
  \item $\varepsilon_n$ is greater than the diameter of $I_{n_{k-1}+(n-L_{k-1})k_0+1}(\omega*(M_0+1))$ for all $\omega\in \mathcal{W}_{n_{k-1}+(n-L_{k-1})k_0}$, namely
\begin{equation*}\label{epsilonfirst}
  \varepsilon_n> \min_{\omega \in \mathcal{W}_{n_{k-1}+(n-L_{k-1})k_0}}\left\{\left|I_{n_{k-1}+(n-L_{k-1}) k_0+1}(\omega*(M_0+1))\right|\right\}.
  \end{equation*}
\end{itemize}

(II-iii) When $n-1=L_k-1$ for some $k$, that is, $n=L_k$, by the results in (I-i) and (I-ii), we know that the basic interval of order $n-1$ is of the form $J_{n_{k-1}+\ell_k k_0}(\sigma)$ for some $\sigma \in \mathcal{W}_{n_{k-1}+\ell_k k_0}$, and contains the basic intervals of order $n$:
 \begin{equation*}
J_{n_{k}}(\sigma*\tau)=\bigcup^{M_0}_{j=1} cl (I_{n_{k}+1}(\sigma*\tau*j))
  \end{equation*}
for all $\tau \in \widehat{\mathcal{W}}_k$. Being similar to the estimates in (II-i), we obtain
\begin{itemize}
  \item $m_n=\sharp \widehat{\mathcal{W}}_k = \lfloor s_k+t_k\rfloor - \lfloor s_k\rfloor \geq \lfloor t_k\rfloor$;

\item $\varepsilon_n$ is greater than the diameter of $I_{n_{k}+1}(\omega*(M_0+1))$ for all $\omega \in \mathcal{W}_{n_{k}}$, namely
 \begin{equation*}\label{epsilonnsecond}
 \varepsilon_n > \min_{\omega \in \mathcal{W}_{n_{k}}}\left\{\left|I_{n_{k}+1}(\omega*(M_0+1))\right|\right\}.
  \end{equation*}
\end{itemize}

\subsubsection{Step III: Calculating the liminf in Lemma \ref{lower}}
To calculate the liminf
\begin{equation}\label{LiminfLemma}
\liminf_{n \to \infty} \frac{\log (m_1\cdots m_{n-1})}{-\log (m_{n}\varepsilon_{n})},
\end{equation}
it is sufficient to estimate the lower bounds of the values of $m_1\cdots m_{n-1}$ and $m_{n}\varepsilon_{n}$.
For $m_1\cdots m_{n-1}$, it is the number of all basic intervals of order $n-1$. For $m_{n}\varepsilon_{n}$, we will use the estimates of $m_n$ and $\varepsilon_n$ in Step II.

For any $W_{n_k-n_{k-1}} \in \mathcal{W}_{n_k-n_{k-1}}$, we first estimate the diameter of the cylinder $I_{n_k-n_{k-1}}(W_{n_k-n_{k-1}})$. Note that $W_{n_k-n_{k-1}}=w_1*\cdots *w_{\ell_k}* \widehat{w}_{k}$, where $w_i\in \mathcal{N}_{m_0}(k_0)$ for all $1\leq i \leq \ell_k$ and $\widehat{w}_{k} \in \widehat{\mathcal{W}}_{k}$, by Lemma \ref{q}, we see that
\begin{align}\label{cylindenk1}
\left|I_{n_k-n_{k-1}}(W_{n_k-n_{k-1}})\right| &\geq \frac{1}{2}\cdot \left|I_{\ell_k k_0}(w_1*\cdots *w_{\ell_k})\right| \cdot \left|I_{r_k+1}(\widehat{w}_k)\right|\nonumber\\
& \geq \cdots\cdots\nonumber \\
&\geq \frac{1}{2^{\ell_k}} \left(\prod^{\ell_k}_{i=1} \left|I_{k_0}(w_i)\right|\right)\cdot \left|I_{r_k+1}(\widehat{w}_k)\right|.
\end{align}
By the definitions of $\mathcal{N}_{m_0}(k_0)$ and $\widehat{\mathcal{W}}_{k}$, and the inequality \eqref{cylinder}, we get that
\[
\left|I_{k_0}(w_i)\right| \geq \frac{1}{2^{m_0}}\ \ \ \text{and}\ \ \ \left|I_{r_k+1}(\widehat{w}_k)\right| > \frac{1}{2^{2(r_k+1)+1}(s_k+t_k)^2} > \frac{1}{8^{k_0}(s_k+t_k)^2}.
\]
Combining this with \eqref{cylindenk1}, we derive that
\begin{align}\label{cylindenk2}
\left|I_{n_k-n_{k-1}}(W_{n_k-n_{k-1}})\right| > \left(\frac{1}{2^{m_0+1}}\right)^{\ell_k} \cdot \frac{1}{8^{k_0}(s_k+t_k)^2}.
\end{align}
For any $\sigma \in \mathcal{W}_{n_k}$, we next estimate the diameter of the cylinder $I_{n_k}(\sigma)$. Note that $\sigma=W_{n_1}*W_{n_2-n_1}*\dots*W_{n_k-n_{k-1}}$, where $W_{n_i-n_{i-1}} \in \mathcal{W}_{n_i-n_{i-1}}$ for all $1\leq i \leq k$, so it follows from \eqref{cylindenk2} and Lemma \ref{q} that
\begin{align}\label{cylindenk}
\left|I_{n_k}(\sigma)\right| &\geq \frac{1}{2^{k-1}} \prod^{k}_{i=1} \left|I_{n_i-n_{i-1}}(W_{n_i-n_{i-1}})\right| \nonumber\\
&>\frac{2}{2^k\cdot 8^{k_0k}}\cdot \left(\frac{1}{2^{m_0+1}}\right)^{\ell_1+\cdots+\ell_k} \cdot \frac{1}{(s_1+t_1)^2\cdots(s_k+t_k)^2} \nonumber\\
&> \frac{2}{2^k\cdot8^{k_0k}}  \cdot \left(\frac{1}{2^{m_0+1}}\right)^{\frac{n_k}{k_0}} \cdot \frac{1}{(s_1+t_1)^2\cdots(s_k+t_k)^2},
\end{align}
where the last inequality follows from $\ell_i < \frac{n_i-n_{i-1}}{k_0}$ for all $1\leq i\leq k$. For any $1 \leq \ell \leq \ell_{k+1}$ and $\sigma \in \mathcal{W}_{n_k+\ell k_0}$, we estimate the diameter of the cylinder $I_{n_k+\ell k_0}(\sigma)$. Note that $\sigma=W_{n_k}*w_{1}*\cdots*w_{\ell}$, where $W_{n_k}\in \mathcal{W}_{n_{k}}$ and $w_{i} \in \mathcal{N}_{m_0}(k_0)$ for all $1\leq i\leq\ell$, we deduce from \eqref{cylindenk} and Lemma \ref{q} that
\begin{align}\label{nk+ell}
\left|I_{n_k+\ell k_0}(\sigma)\right| &\geq  \frac{1}{2^\ell} \cdot\left|I_{n_k}(W_{n_k})\right|\cdot\left(\prod^{\ell}_{i=1} \left|I_{k_0}(w_i)\right|\right)\nonumber\\
&>\left(\frac{1}{2^{m_0+1}}\right)^{\ell}\frac{2}{2^k\cdot8^{k_0k}}  \cdot \left(\frac{1}{2^{m_0+1}}\right)^{\frac{n_k}{k_0}+\ell} \cdot \frac{1}{(s_1+t_1)^2\cdots(s_k+t_k)^2}\nonumber\\
&= \frac{2}{2^k\cdot 8^{k_0k}} \cdot \left(\frac{1}{2^{m_0+1}}\right)^{\frac{n_k+\ell k_0}{k_0}} \cdot \frac{1}{(s_1+t_1)^2\cdots(s_k+t_k)^2}.
\end{align}

Now we are ready to estimate the lower bounds for the values of $m_1\cdots m_{n-1}$ and $m_{n}\varepsilon_{n}$, and the liminf \eqref{LiminfLemma}.

(III-i) When $n=L_k$ for some $k$, that is, $n-1=L_k-1$, we deduce from \eqref{Wnkell} that the number of all basic intervals of order $n-1$ is
\begin{align}\label{Liminfm1mn-1}
m_1\cdots m_{n-1} &= \sharp \mathcal{W}_{n_{k-1}+\ell_k k_0}\nonumber \\
&\geq\left(N_{m_0}(k_0)\right)^{\frac{n_{k-1}+\ell_k k_0}{k_0}-(k-1)} \left(\lfloor t_1\rfloor \cdots \lfloor t_{k-1}\rfloor\right) \nonumber\\
& \geq \left(N_{m_0}(k_0)\right)^{\frac{n_{k}}{k_0}-k}  \left(\lfloor t_1\rfloor \cdots \lfloor t_{k-1}\rfloor\right),
\end{align}
where the last inequality follows from $n_{k-1}+\ell_k k_0=n_k-(r_k+1) \geq n_k-k_0$. For $m_n\varepsilon_n$, by (II-iii), we know that $m_n \geq \lfloor t_{k}\rfloor$ and
\[
\varepsilon_n > \min_{\omega \in \mathcal{W}_{n_{k}}}\left\{|I_{n_{k}+1}(\omega*(M_0+1))|\right\}.
\]
For any $\omega \in \mathcal{W}_{n_{k}}$, by \eqref{cylindenk} and Lemma \ref{q}, we derive that
\begin{align*}
\left|I_{n_{k}+1}(\omega*(M_0+1))\right| &\geq \frac{1}{2}\cdot \left|I_{n_k}(\omega)\right|\cdot \left|I_1(M_0+1)\right|\\
&> \frac{1}{2^k8^{k_0k}(M_0+2)^2}  \left(\frac{1}{2^{m_0+1}}\right)^{\frac{n_k}{k_0}} \left(\prod^k_{i=1} \frac{1}{(s_i+t_i)^2}\right).
\end{align*}
Hence
\[
m_n\varepsilon_n >\frac{\lfloor t_{k}\rfloor}{2^k 8^{k_0k}(M_0+2)^2} \left(\frac{1}{2^{m_0+1}}\right)^{\frac{n_k}{k_0}} \left(\prod^k_{i=1} \frac{1}{(s_i+t_i)^2}\right).
\]
Combining these with \eqref{Liminfm1mn-1} and the hypothesis (H1) and $s_k,t_k \to \infty$ as $k \to \infty$, we deduce that the liminf in \eqref{LiminfLemma} with $n=L_k$ is not less than
\begin{equation}\label{equliminf1}
\liminf_{k \to \infty} \frac{\frac{n_k}{k_0}\log N_{m_0}(k_0) +\sum^{k-1}_{i=1}\log t_i}{\frac{n_k}{k_0}(m_0+1)\log 2 + 2\sum^{k}_{i=1}\log (s_i+t_i) -\log t_k}.
\end{equation}
It follows from the hypotheses (H2) and (H3) that
\[
\lim_{k \to \infty} \frac{\sum^{k-1}_{i=1}\log t_i}{n_k} =  \alpha -\beta\ \text{and}\ \lim_{k \to \infty} \frac{2\sum^{k}_{i=1}\log (s_i+t_i) -\log t_k}{n_k} =2\alpha -\beta.
\]
Hence we conclude from \eqref{UVP} that the liminf in \eqref{equliminf1} is equal to
\begin{align*}
\frac{\frac{1}{k_0}\log N_{m_0}(k_0) +(\alpha-\beta)}{\frac{1}{k_0}(m_0+1)\log 2 + (2\alpha -\beta)} &= \frac{U+(\alpha-\beta)}{2V+(2\alpha-\beta)}\\
 &> \theta_0- \frac{\varepsilon}{2V+(2\alpha-\beta)}\\
 &> \theta_0- \frac{\varepsilon}{\alpha}.
\end{align*}
Therefore, the liminf in \eqref{LiminfLemma} with $n=L_k$ is not less than $\theta_0 - \varepsilon/\alpha$.

(III-ii) When $L_{k-1}<n<L_k$ for some $k$, that is, $L_{k-1}\leq n-1<L_{k}-1$, let
\[
h_k:=n_{k-1}+ (n-1-L_{k-1})k_0.
\]
Then $n_{k-1}\leq h_k <n_k$ and $h_k/k \to \infty$ as $k \to \infty$. For $m_1\cdots m_{n-1}$, by \eqref{Wnkell}, we see that the number of all basic intervals of order $n-1$ is
\begin{align}\label{Liminfm1mnhk}
m_1\cdots m_{n-1} = \sharp \mathcal{W}_{h_k} \nonumber
&\geq \left(N_{m_0}(k_0)\right)^{\frac{h_k}{k_0}-(k-1)} \left(\lfloor t_1\rfloor \cdots \lfloor t_{k-1}\rfloor\right)\nonumber\\
&> \left(N_{m_0}(k_0)\right)^{\frac{h_k}{k_0}-k} \left(\lfloor t_1\rfloor \cdots \lfloor t_{k-1}\rfloor\right).
\end{align}
For $m_n\varepsilon_n$, we deduce from (II-i) and (II-ii) that $m_n= N_{m_0}(k_0)$ and
\[
\varepsilon_n >\min_{\omega \in \mathcal{W}_{h_{k}+k_0}}\left\{\left|I_{h_{k}+k_0+1}(\omega*(M_0+1))\right|\right\}.
\]
For any $\omega \in \mathcal{W}_{h_{k}+k_0}$, it follows from \eqref{nk+ell} and Lemma \ref{q} that
\begin{align*}
\left|I_{h_{k}+k_0+1}(\omega*(M_0+1))\right| &\geq \frac{1}{2}\cdot \left|I_{h_k+k_0}(\omega)\right|\cdot \left|I_1(M_0+1)\right| \\
&>  \frac{1}{2^{k} 8^{k_0k}(M_0+2)^2} \left(\frac{1}{2^{m_0+1}}\right)^{\frac{h_k+k_0}{k_0}} \left(\prod^{k-1}_{i=1}\frac{1}{(s_i+t_i)^2}\right).
\end{align*}
Hence
\[
m_n\varepsilon_n >\frac{N_{m_0}(k_0)}{2^k 8^{k_0k}(M_0+2)^2}  \cdot \left(\frac{1}{2^{m_0+1}}\right)^{\frac{h_k+k_0}{k_0}} \left(\prod^{k-1}_{i=1}\frac{1}{(s_i+t_i)^2}\right).
\]
Combining this with \eqref{Liminfm1mnhk} and the hypothesis (H1) and $s_k,t_k \to \infty$ as $k \to \infty$, in this case, we know that the liminf in \eqref{LiminfLemma} is not less than
\begin{equation}\label{equliminfhk1}
\liminf_{k \to \infty} \frac{\frac{h_k}{k_0}\log N_{m_0}(k_0) +\sum^{k-1}_{i=1}\log t_i}{\frac{h_k}{k_0}(m_0+1)\log 2 + 2\sum^{k-1}_{i=1}\log (s_i+t_i)}.
\end{equation}
By the hypothesis (H2), we have
\[
\lim_{k \to \infty}\frac{\sum^{k-1}_{i=1}\log s_i}{\sum^{k-1}_{i=1}\log t_i} = \lim_{k \to \infty}\frac{\sum^{k-1}_{i=1}\log s_i}{\sum^{k-1}_{i=1}\log (s_i+t_i)}=1,
\]
which implies that the liminf in \eqref{equliminfhk1} is equal to
\begin{equation}\label{equliminfhk2}
\liminf_{k \to \infty} \frac{\frac{h_k}{k_0}\log N_{m_0}(k_0) +\sum^{k-1}_{i=1}\log s_i}{\frac{h_k}{k_0}(m_0+1)\log 2 + 2\sum^{k-1}_{i=1}\log s_i} = \liminf_{k \to \infty}\frac{U +\frac{\sum^{k-1}_{i=1}\log s_i}{h_k}}{2V+2\frac{\sum^{k-1}_{i=1}\log s_i}{h_k}}.
\end{equation}
Since $U>2V\theta_0+\mathrm{P}(\theta_0)-\varepsilon>2V\theta_0>V$, we see that the map
\[
\gamma \mapsto \frac{U+\gamma}{2V+2\gamma}= \frac{1}{2}+ \frac{U-V}{2V+2\gamma}
\]
is decreasing, and so
\begin{equation}\label{UVnk-1}
\frac{U +\frac{\sum^{k-1}_{i=1}\log s_i}{h_k}}{2V+2\frac{\sum^{k-1}_{i=1}\log s_i}{h_k}} \geq \frac{U +\frac{\sum^{k-1}_{i=1}\log s_i}{n_{k-1}}}{2V+2\frac{\sum^{k-1}_{i=1}\log s_i}{n_{k-1}}}.
\end{equation}
It follows from \eqref{UVP} that
\begin{align*}
U+\alpha &>(2V+2\alpha-\beta)\theta_0 +\beta -\varepsilon \\
&= (2V+2\alpha)\theta_0 +(1-\theta_0)\beta -\varepsilon \\
&\geq (2V+2\alpha)\theta_0 -\varepsilon.
\end{align*}
Combining these with \eqref{equliminfhk2} and \eqref{UVnk-1}, in view of (H3), we deduce that the liminf in \eqref{equliminfhk1} is not less than
\[
\frac{U+\alpha}{2V+2\alpha} > \theta_0- \frac{\varepsilon}{2V+2\alpha} > \theta_0- \frac{\varepsilon}{2\alpha}>\theta_0- \frac{\varepsilon}{\alpha}.
\]
So, in this case, the liminf in \eqref{LiminfLemma} is not less than $\theta_0 - \varepsilon/\alpha$.

Therefore, we conclude from (III-i) and (III-ii) that
\begin{align*}
\dim_{\rm H} E(\{n_k\},\{s_k\},\{t_k\}) &\geq \dim_{\rm H} \mathrm{E} \\
&\geq \liminf_{n \to \infty} \frac{\log (m_1\cdots m_{n-1})}{-\log (m_{n}\varepsilon_{n})}\\
& >\theta_0 - \frac{\varepsilon}{\alpha}.
\end{align*}
Letting $\varepsilon \to 0$, we have
\[
\dim_{\rm H} E(\{n_k\},\{s_k\},\{t_k\}) \geq \theta_0:=\theta(\alpha,\beta).
\]
This completes the proof of the lower bound of Theorem \ref{Alpha}.

We end this section with the proof of the lower bound of Theorem \ref{Xbc}.

\begin{proof}[Proof of the lower bound of Theorem \ref{Xbc}]
We proceed the proof from the following three cases.

(i) When $b=c=0$, we see that $\dim_{\rm H} X(b,c) =1$, which is the unique solution of $\mathrm{P}(\theta)=0$.

(ii) When $b,c\in [0,\infty)$ with $b+c>0$, for any $\varepsilon >0$, there exists $\gamma>0$ such that $b+\varepsilon=\frac{c+\varepsilon}{e^\gamma-1}$, that is, $e^\gamma = 1+\frac{c+\varepsilon}{b+\varepsilon}$. For any $k \in \mathbb N$, let
\[
n_k:=\lceil e^{\gamma k}\rceil+1\ \ \text{and}\ \  s_k=t_k:=e^{(c+\varepsilon)n_{k}}.
\]
We assert that $E(\{n_k\},\{s_k\},\{t_k\})$ is a subset of $X(b, c)$. In fact, for any $m \geq 0$, there is $K_0 \in \mathbb N$ such that $s_k >e^m$ for all $k \geq K_0$. For any $x\in E(\{n_k\},\{s_k\},\{t_k\})$ and for all $k >e^{\gamma K_0}/(e^\gamma-1)+K_0$, since
\begin{align*}
(c+\varepsilon)(n_{K_0}+\cdots+n_{k-1}) &\geq  (c+\varepsilon)\cdot\frac{e^{\gamma k}- e^{\gamma K_0}}{e^{\gamma}-1}+(c+\varepsilon)(k-K_0)\\
&>(b+\varepsilon)e^{\gamma k} \\
&>(b+\varepsilon)(n_k-1),
\end{align*}
we deduce that $a_{n_k}(x) >s_k =e^{(c+\varepsilon)n_k}>e^{c(n_k-1)}$ and
\begin{align*}
\Pi^{(m)}_{n_k-1}(x) &\geq a_{n_{K_0}}(x)\cdots a_{n_{k-1}}(x)\\
&> e^{(c+\varepsilon)(n_{K_0}+\cdots+n_{k-1})}> e^{(b+\varepsilon)(n_k-1)}> e^{b(n_k-1)}.
\end{align*}
That is, $x \in X_m(b,c)$. Hence $x\in X(b,c)$. This means that $E(\{n_k\},\{s_k\},\{t_k\})$ is a subset of $X(b,c)$, and so $\dim_{\rm H}X(b,c) \geq \dim_{\rm H}E(\{n_k\},\{s_k\},\{t_k\})$. Note that
\[
\alpha = \lim_{k \to \infty} \frac{1}{n_k} \sum^k_{j=1} (c+\varepsilon)n_j = b+c+2\varepsilon\ \ \text{and}\ \ \beta=\lim_{k \to \infty} \frac{1}{n_k}\log s_k =c+\varepsilon,
\]
it follows from the first statement of Theorem \ref{Alpha} that the Hausdorff dimension of $E(\{n_k\},\{s_k\},\{t_k\})$ is $\theta(b+c+2\varepsilon,c+\varepsilon)$. Thus $\dim_{\rm H}X(b,c) \geq \theta(b+c+2\varepsilon,c+\varepsilon)$. Letting $\varepsilon\to 0$, we have
$\dim_{\rm H}X(b,c) \geq \theta(b+c,c) = \Theta(b,c)$.

(iii) When $b=\infty$ or $c=\infty$, we prove that $\dim_{\rm H}X(b, c) \geq 1/2$. Since $X(\infty,\infty)$ is a subset of $X(b,c)$, it suffices to show $\dim_{\rm H}X(\infty, \infty) \geq 1/2$. Let $u_n=v_n :=e^{n^2}$ for all $n \in \mathbb N$. Then $F(\{u_n\},\{v_n\})$ in \eqref{Euv} is a subset of $X(\infty, \infty)$. We deduce from Lemma \ref{LRlemma} that
\begin{align*}
\dim_{\rm H} X(\infty,\infty) &\geq \dim_{\rm H}F(\{u_n\},\{v_n\})\\
&= \liminf_{n \to \infty}\frac{\sum^n_{k=1}k^2}{2\sum^{n}_{k=1}k^2 + (n+1)^2} =\frac{1}{2}.
\end{align*}
Hence $\dim_{\rm H}X(b, c) \geq \dim_{\rm H} X(\infty,\infty) \geq 1/2$.
\end{proof}

\section{Proof of the second statement of Theorem \ref{Alpha}}
In this section, we assume that $\alpha=0$. We are going to prove that
\[
\dim_{\rm H}E(\{n_k\},\{s_k\}, \{t_k\})=\dim_{\rm H} E_{L}(\{n_k\},\{s_k\}, \{t_k\})=1.
\]
Since $E(\{n_k\},\{s_k\}, \{t_k\})$ is a subset of $E_{L}(\{n_k\},\{s_k\}, \{t_k\})$, it is sufficient to show that $\dim_{\rm H}E(\{n_k\},\{s_k\}, \{t_k\})=1$.

Let $M>0$ be an integer, and let
\begin{equation}\label{Mpressuref}
\mathrm{P}_M(\theta):= \lim_{n \to \infty} \frac{1}{n}\log \sum_{1\leq a_1,\dots,a_n \leq M} q^{-2\theta}_n(a_1,\dots,a_n),\ \forall \theta >0.
\end{equation}
By \eqref{mgf}, the limit in \eqref{Mpressuref} exists, and so $\mathrm{P}_M(\theta)$ is well-defined. We call $\mathrm{P}_M(\theta)$ the restricted Diophantine pressure function of continued fractions. It will play the same role as the Diophantine pressure function $\mathrm{P}(\theta)$ in Section \ref{PFA}. To this end, we first list some properties of $\mathrm{P}_M(\theta)$. See \cite{Sar,PW,MU96} for more results of (restricted) Diophantine pressure function.

\begin{lemma}[\cite{Sar,MU96}]\label{PMbasic}
We have
\begin{enumerate}
  \item[(i)] For fixed $M\in\mathbb N$, $\mathrm{P}_M(\theta)$ is a decreasing Lipschitz function of $\theta$.
  \item[(ii)] For fixed $\theta>1/2$, $\mathrm{P}_M(\theta)$ increases to $\mathrm{P}(\theta)$ as $M \to \infty$. Denote by $\theta_M$ the unique zero point of $\mathrm{P}_M(\theta)$. As a consequence, $\theta_M$ increases to $1$ as $M \to \infty$.
\end{enumerate}
\end{lemma}

\begin{proof}
The proof is a consequence of Proposition 3.3 of \cite{MU96} and Theorem 2 of \cite{Sar} in the setting of iterated function systems. For the sake of completeness, we give direct proofs by means of the basic properties of continued fractions.

(i) Let $M \in \mathbb N$ be fixed. By definition, $\mathrm{P}_M(\theta)$ is a decreasing function of $\theta$. Now we are ready to prove that $\mathrm{P}_M(\theta)$ is Lipschitz continuous. Given $n \in \mathbb N$ and $a_1,\dots,a_n \in\{1,\dots, M\}$, for any $\delta>0$, it follows from \eqref{qnleq} that
\[
(2M)^{-2\delta n}<q^{-2\delta}_n(a_1,\dots,a_n) \leq 1,
\]
which implies that
\[
1\leq \frac{\sum_{1\leq a_1,\dots,a_n \leq M} q^{-2\theta}_n(a_1,\dots,a_n)}{\sum_{1\leq a_1,\dots,a_n \leq M} q^{-2(\theta+\delta)}_n(a_1,\dots,a_n)} <(2M)^{2\delta n}.
\]
Hence
\begin{align*}
&\left|\frac{1}{n}\log \sum_{1\leq a_1,\dots,a_n \leq M} q^{-2\theta}_n(a_1,\dots,a_n)- \frac{1}{n}\log \sum_{1\leq a_1,\dots,a_n \leq M} q^{-2(\theta+\delta)}_n(a_1,\dots,a_n)\right|\\
&\leq 2\delta \log(2M),
\end{align*}
which gives that $|\mathrm{P}_M(\theta)- \mathrm{P}_M(\theta+\delta)| \leq 2\delta \log(2M)$. That is, $\mathrm{P}_M(\theta)$ is a Lipschitz function.

(ii) Let $\theta>1/2$ be fixed. It is clear that $\mathrm{P}_M(\theta)$ is an increasing function of $M$. Fix $M\in \mathbb N_{\geq 2}$.
For any small $\varepsilon>0$, by the definition of $\mathrm{P}_M(\theta)$ in \eqref{Mpressuref}, there exists $K >0$ such that for all $m\in \mathbb N$,
\begin{equation}\label{lemMpre}
\sum_{1\leq a_1,\dots,a_m \leq M} q^{-2\theta}_m(a_1,\dots,a_m) \leq Ke^{(\mathrm{P}_M(\theta)+\varepsilon)m}.
\end{equation}
For $n \in \mathbb N$, let
\[
\delta_n(\theta):= \sum_{j>n} \left(\frac{2}{j}\right)^{2\theta}.
\]
Then $\delta_n(\theta) \to 0$ as $n \to \infty$. For $m \in \mathbb N$ and $(a_1,\dots,a_m) \in \mathbb N^m$ with $a_{j_1},\dots,a_{j_k}>M$ for some $0\leq k \leq m$ and $a_j \leq M$ for $j\neq j_1,\dots,j_k$, by Lemma \ref{qnsep}, we see that
\begin{equation}\label{qmsepji}
q_m(a_1,\dots,a_m) > \left(\prod^k_{i=1} \frac{a_{j_i}}{2}\right) q_{m-k}(\omega),
\end{equation}
where $\omega \in \mathbb N^{m-k}_{\leq M}$ is obtained by deleting $\{a_{j_i}\}_{1\leq i\leq k}$ from $(a_1,\dots,a_m)$. Note that the number of combinations of $k\ (0\leq k \leq m)$ numbers randomly selected from $m$ numbers is $\binom{m}{k}$, so it follows from \eqref{qmsepji} that
\begin{equation*}
\sum_{(a_1,\dots,a_m) \in \mathbb N^m} q^{-2\theta}_m(a_1,\dots,a_m) \leq  \prod^m_{k=0} \binom{m}{k} \delta^k_M(\theta)\sum_{\omega \in \mathbb N^{m-k}_{\leq M}}q^{-2\theta}_{m-k}(\omega).
\end{equation*}
Combining this with \eqref{lemMpre}, we conclude that
\begin{align*}
\sum_{(a_1,\dots,a_m) \in \mathbb N^m} q^{-2\theta}_m(a_1,\dots,a_m) &\leq  K\prod^m_{k=0} \binom{m}{k} \delta^k_M(\theta) e^{(\mathrm{P}_M(\theta)+\varepsilon)(m-k)}\\
&= K\left(e^{\mathrm{P}_M(\theta)+\varepsilon}+ \delta_{M}(\theta)\right)^m,
\end{align*}
which gives that $\mathrm{P}(\theta) \leq \log (e^{\mathrm{P}_M(\theta)+\varepsilon}+ \delta_{M}(\theta)) \leq \mathrm{P}_M(\theta) +\varepsilon+ \delta_{M}(\theta)$. Letting $\varepsilon \to 0$, we obtain
\[
\mathrm{P}_M(\theta) \leq \mathrm{P}(\theta) \leq \mathrm{P}_M(\theta) + \delta_{M}(\theta).
\]
Hence $\mathrm{P}_M(\theta) \to \mathrm{P}(\theta)$ as $M\to \infty$.
\end{proof}

We will use the method in Section \ref{PFA} to prove Theorem \ref{Alpha} for the case $\alpha=0$.
For $k,m \in \mathbb{N}$, let
\[
\mathcal{M}_m(k):=\left\{(\sigma_1,\ldots,\sigma_k)\in \{1,2,\dots,M\}^k: \frac{1}{2^{m}} \leq |I_k(\sigma_1,\ldots,\sigma_k)|<\frac{1}{2^{m-1}}\right\}.
\]
Denote by $M_m(k) := \sharp \mathcal{M}_m(k)$ the cardinality of $\mathcal{M}_m(k)$. Being similar to the proof of Lemma \ref{NP}, we have the following estimation of $M_m(k)$.

\begin{lemma}\label{MP}
Let $M \in \mathbb N_{\geq 2}$ be fixed. For any $0<\theta <\theta_M$ and for sufficiently large $k$, there exists $m \in \mathbb{N}$ such that
\begin{equation*}
M_m(k) > 2^{(m+1)\theta}.
\end{equation*}
\end{lemma}

\begin{proof}
We prove the lemma by contradiction. Fix $M \in \mathbb N_{\geq 2}$. Suppose that there exist $0<\theta_* <\theta_M$ and a strictly increasing sequence $\{k_i\}$ of positive integers such that for all $m \in \mathbb N$ and for all $i \in \mathbb N$,
\begin{equation*}
M_m(k_i) \leq 2^{(m+1)\theta_*}.
\end{equation*}
For any small $0<\delta<\theta_M-\theta_*$, we see that
\[
\sum_{1\leq a_1,\dots,a_{k_i} \leq M} |I_{k_i}(a_1,\dots,a_{k_i})|^{\theta+\delta} \leq \sum^\infty_{m=1} \frac{M_m(k_i)}{2^{(m-1)(\theta_*+\delta)}}\leq \sum^\infty_{m=1} \frac{2^{2\theta_*+\delta}}{2^{\delta m}} =\frac{2^{2\theta_*+\delta}}{2^\delta-1},
\]
which yields that $\mathrm{P}_{M}(\theta_*+\delta)=0$. This contradicts with the fact that ``$\theta_M$ is the unique zero point of $\mathrm{P}_{M}(\cdot)$" since $\theta_*+\delta<\theta_M$.
\end{proof}

For any $M\geq 2$, we have $\theta_M=\dim_{\rm H}E_M>1/2$. For any fixed $1/2<\theta <\theta_M$, by Lemma \ref{MP}, there exist positive integers $k_0$ and $m_0$ such that
\begin{equation}\label{PMn0k0}
M_{m_0}(k_0) > 2^{(m_0+1)\theta}.
\end{equation}
Replacing $\mathcal{N}_{m_0}(k_0)$ by $\mathcal{M}_{m_0}(k_0)$ in Section 3.2.1, and following Step I in Section 3.2.2, we can obtain a subset of $E(\{n_k\},\{s_k\}, \{t_k\})$:
\[
\mathrm{E}= \bigcap^\infty_{n=0}\mathrm{E}_n.
\]
Moreover, the estimates of $m_n$ and $\varepsilon_n$ in Step II of Section 3.2.3 are valid as long as we replace $N_{m_0}(k_0)$ by $M_{m_0}(k_0)$.
For the lower bound of $\dim_{\rm H}\mathrm{E}$, based on the estimates of $m_1\cdots m_{n-1}$ and $m_n\varepsilon_n$ in Step III of Section 3.2.3,
it suffices to calculate the following two liminfs:
\begin{equation}\label{Zeroequliminf1}
\liminf_{k \to \infty} \frac{\frac{n_k}{k_0}\log M_{m_0}(k_0) +\sum^{k-1}_{i=1}\log t_i}{\frac{n_k}{k_0}(m_0+1)\log 2 + 2\sum^{k}_{i=1}\log (s_i+t_i) -\log t_k}
\end{equation}
and
\begin{equation}\label{Zeroequliminfhk1}
\liminf_{k \to \infty} \frac{\frac{h_k}{k_0}\log M_{m_0}(k_0) +\sum^{k-1}_{i=1}\log t_i}{\frac{h_k}{k_0}(m_0+1)\log 2 + 2\sum^{k-1}_{i=1}\log (s_i+t_i)},
\end{equation}
where $n_{k-1} \leq h_k <n_k$. See \eqref{equliminf1} and \eqref{equliminfhk1} in Step III of Section 3.2.3. By the hypotheses (H2) and (H3), we deduce that
\[
\lim_{k \to \infty} \frac{\sum^{k-1}_{i=1}\log t_i}{n_k} = \lim_{k \to \infty} \frac{2\sum^{k}_{i=1}\log (s_i+t_i) -\log t_k}{n_k} =0
\]
and
\[
\lim_{k \to \infty} \frac{\sum^{k-1}_{i=1}\log t_i}{h_k} = \lim_{k \to \infty} \frac{\sum^{k-1}_{i=1}\log (s_i+t_i)}{h_k} =0.
\]
This means that the liminfs in \eqref{Zeroequliminf1} and \eqref{Zeroequliminfhk1} are the same, and are equal to
\[
\frac{\log M_{m_0}(k_0)}{(m_0+1)\log 2}.
\]
Combining this with \eqref{PMn0k0}, we see that $\dim_{\rm H} E(\{n_k\},\{s_k\}, \{t_k\})\geq \dim_{\rm H}\mathrm{E} >\theta$.
Since $\theta <\theta_M$ is arbitrary, we have $\dim_{\rm H} E(\{n_k\},\{s_k\}, \{t_k\})\geq \theta_M$. Letting $M \to \infty$, in view of Lemma \ref{PMbasic}, we obtain $$\dim_{\rm H} E(\{n_k\},\{s_k\}, \{t_k\})=1.$$
This completes the proof for the second statement of Theorem \ref{Alpha}.

\section{Proof of the third statement of Theorem \ref{Alpha}}

In this section, we assume that $\alpha=\infty$. We first give the proof for the Hausdorff dimension of $E(\{n_k\},\{s_k\}, \{t_k\})$ and then deal with the Hausdorff dimension of $E_{L}(\{n_k\},\{s_k\}, \{t_k\})$.

\subsection{Proof for the dimension of $E(\{n_k\},\{s_k\}, \{t_k\})$}
We will prove that
\begin{equation}\label{infinityequation}
  \dim_{\rm H} E(\{n_k\},\{s_k\}, \{t_k\}) = \frac{1}{2+\xi},
  \end{equation}
  where $\xi \in [0,\infty]$ is defined as
  \[
  \xi:= \limsup_{k \to \infty} \frac{\log s_{k+1}}{\log s_1+\cdots+\log s_k}.
  \]

For the upper bound of \eqref{infinityequation}, since
\[
 E(\{n_k\},\{s_k\}, \{t_k\}) = \bigcup^\infty_{K=1} \bigcup^\infty_{L=1}E_{K,L}(\{n_k\},\{s_k\}, \{t_k\}),
\]
where $E_{K,L}(\{n_k\},\{s_k\}, \{t_k\})$ is defined as
\[
\big\{x\in \mathbb{I}: 1\leq a_j(x) \leq L, \forall j\neq n_k(k \in \mathbb N); s_k < a_{n_k}(x) \leq s_k+ t_k \ \forall k\geq K\big\},
\]
we have
\begin{equation}\label{KLE}
\dim_{\rm H} E(\{n_k\},\{s_k\}, \{t_k\}) = \sup_{K,L}\Big\{\dim_{\rm H} E_{K,L}(\{n_k\},\{s_k\}, \{t_k\}) \Big\}.
\end{equation}
So it remains to calculate the Hausdorff dimension of $E_{K,L}(\{n_k\},\{s_k\}, \{t_k\})$ for all $K,L \in \mathbb N$. For any fixed $L\in \mathbb N$, we only consider the case where $K=1$, the proofs for other cases are similar.

For any $k \in \mathbb N$, let $\mathcal{C}_{n_k-1}$ be the set of sequences $(\sigma_1,\sigma_2,\dots,\sigma_{n_k-1})\in \mathbb N^{n_k-1}$ such that $1\leq \sigma_j \leq L$ for all $1\leq j \leq n_k-1$ with $j \neq n_1,\dots,n_{k-1}$ and $s_i < \sigma_{n_i}\leq s_i+ t_i$ for all $1\leq i\leq k-1$. For any $(\sigma_1,\sigma_2,\dots,\sigma_{n_k-1})\in \mathcal{C}_{n_k-1}$, let
\[
J_{n_k-1}(\sigma_1,\sigma_2,\dots,\sigma_{n_k-1}) := \bigcup_{s_k<j\leq s_k+t_k} I_{n_{k}} (\sigma_1,\sigma_2,\dots,\sigma_{n_k-1},j).
\]
Then $J_{n_k-1}(\sigma_1,\sigma_2,\dots,\sigma_{n_k-1})$ is an interval. Moreover,
\[
E_{1,L}(\{n_k\},\{s_k\}, \{t_k\})= \bigcap^\infty_{k=1} \bigcup_{(\sigma_1,\sigma_2,\dots,\sigma_{n_k-1})\in \mathcal{C}_{n_k-1}}J_{n_k-1}(\sigma_1,\sigma_2,\dots,\sigma_{n_k-1}).
\]
This means that for all $k \in \mathbb N$, $E_{1,L}(\{n_k\},\{s_k\}, \{t_k\})$ is covered by the family of intervals $J_{n_k-1}(\sigma_1,\sigma_2,\dots,\sigma_{n_k-1})$ for all $(\sigma_1,\sigma_2,\dots,\sigma_{n_k-1})\in \mathcal{C}_{n_k-1}$. Note that
\[
\sharp \mathcal{C}_{n_k-1} \leq L^{n_k-k} (t_1+1)\cdots(t_{k-1}+1)
\]
and
\begin{align*}
\left|J_{n_k-1}(\sigma_1,\sigma_2,\dots,\sigma_{n_k-1})\right| &=\sum_{s_k<j\leq s_k+t_k} \left|I_{n_{k}} (\sigma_1,\sigma_2,\dots,\sigma_{n_k-1},j)\right|\\
&\leq 2\left|I_{n_{k}-1} (\sigma_1,\sigma_2,\dots,\sigma_{n_k-1})\right| \cdot \sum_{s_k<j\leq s_k+t_k} \frac{1}{j(j+1)}\\
& < \frac{2}{(s_1s_2 \cdots s_{k-1})^2s_k},
\end{align*}
we conclude from the hypotheses (H2) and (H3), and Lemma \ref{upp} that
\begin{align*}
\dim_{\rm H} E_{1,L}(\{n_k\},\{s_k\},\{t_k\}) &\leq \liminf_{k \to \infty} \frac{(n_k-k)\log L +\sum^{k-1}_{j=1}\log(t_j+1)}{2\sum^{k-1}_{j=1}\log s_j -\log 2+ \log s_{k}}\\
&\leq \liminf_{k \to \infty} \frac{\sum^{k-1}_{j=1}\log(t_j+1)}{2\sum^{k-1}_{j=1}\log s_j+\log s_{k}}\\
&=\liminf_{k \to \infty} \frac{\sum^{k-1}_{j=1}s_j}{2\sum^{k-1}_{j=1}\log s_j+\log s_{k}}= \frac{1}{2+\xi}.
\end{align*}
Similarly, $\dim_{\rm H} E_{K,L}(\{n_k\},\{s_k\},\{t_k\}) \leq 1/(2+\xi)$ for all $K,L \in \mathbb N$. Combining this with \eqref{KLE}, we derive that
\[
\dim_{\rm H} E(\{n_k\},\{s_k\},\{t_k\}) \leq \frac{1}{2+\xi}.
\]

For the lower bound of \eqref{infinityequation}, we follow the method used in Section 3.2.

\begin{lemma}\label{NmkInfinity}
For any $1/2<\theta<1$ and sufficiently large $k$, there exists $m \in \mathbb{N}$ such that
\begin{equation*}
N_m(k) > 2^{(m+1)\theta}.
\end{equation*}
\end{lemma}
\begin{proof}
The proof is similar to that of Lemmas \ref{NP} and \ref{MP}, so we omit it.
\end{proof}

Let $1/2<\theta_0<1$ be fixed. By Lemma \ref{NmkInfinity}, there exist positive integers $k_0$ and $m_0$ such that
\begin{equation*}\label{InfinityNmk}
N_{m_0}(k_0) > 2^{(m_0+1)\theta_0}.
\end{equation*}
Denote
\[
\widetilde{U}:= \frac{1}{k_0}\log N_{m_0}(k_0)\ \ \ \text{and}\ \ \ \widetilde{V}:=  \frac{m_0+1}{2k_0}\log 2.
\]
Then $\widetilde{U}>2\widetilde{V}\theta_0>\widetilde{V}$, and so the map
\begin{equation}\label{mapUV}
\gamma \mapsto \frac{\widetilde{U}+\gamma}{2\widetilde{V}+2\gamma}= \frac{1}{2}+ \frac{\widetilde{U}-\widetilde{V}}{2\widetilde{V}+2\gamma}
\end{equation}
is decreasing on $(0,\infty)$. Following the Steps I-III in Section 3.2, we obtain a subset of $E(\{n_k\},\{s_k\},\{t_k\})$ whose Hausdorff dimension is not less than the minimum of the liminfs
\begin{equation}\label{Infinityequliminf1}
\liminf_{k \to \infty} \frac{\frac{n_k}{k_0}\log N_{m_0}(k_0) +\sum^{k-1}_{i=1}\log t_i}{\frac{n_k}{k_0}(m_0+1)\log 2 + 2\sum^{k}_{i=1}\log (s_i+t_i) -\log t_k}
\end{equation}
and
\begin{equation}\label{Infinityequliminfhk1}
\liminf_{k \to \infty} \frac{\frac{h_k}{k_0}\log N_{m_0}(k_0) +\sum^{k-1}_{i=1}\log t_i}{\frac{h_k}{k_0}(m_0+1)\log 2 + 2\sum^{k-1}_{i=1}\log (s_i+t_i)},
\end{equation}
where $n_{k-1} \leq h_k <n_k$.

By the hypotheses (H2) and (H3), we deduce that
\[
\lim_{k \to \infty} \frac{\sum^{k-1}_{i=1}\log s_i}{\sum^{k-1}_{i=1}\log t_i} = \lim_{k \to \infty} \frac{ 2\sum^{k-1}_{i=1}\log s_i +\log s_k}{ 2\sum^{k}_{i=1}\log (s_i+t_i) -\log t_k} =1
\]
and
\[
 \lim_{k \to \infty}\frac{\sum^{k}_{i=1}\log t_i}{n_{k}}=\lim_{k \to \infty}\frac{\sum^{k}_{i=1}\log (s_i+t_i)}{n_{k}}=\lim_{k \to \infty}\frac{2\sum^{k}_{i=1}\log (s_i+t_i) -\log t_k}{n_k}=\infty.
\]
For \eqref{Infinityequliminf1}, it is equal to
\[
\lim_{k \to \infty}\frac{\sum^{k-1}_{i=1}\log t_i}{2\sum^{k}_{i=1}\log (s_i+t_i) -\log t_k}=\lim_{k \to \infty}\frac{\sum^{k-1}_{i=1}\log s_i}{2\sum^{k-1}_{i=1}\log s_i +\log s_k} = \frac{1}{2+\xi}.
\]
For \eqref{Infinityequliminfhk1}, since the map in \eqref{mapUV} is decreasing, it is not less than
\[
\liminf_{k \to \infty} \frac{\widetilde{U} +\frac{\sum^{k-1}_{i=1}\log t_i}{n_{k-1}}}{2\widetilde{V} + 2\cdot \frac{\sum^{k-1}_{i=1}\log (s_i+t_i)}{n_{k-1}}} =\frac{1}{2} \geq \frac{1}{2+\xi}.
\]
Therefore,
\[
\dim_{\rm H} E(\{n_k\},\{s_k\},\{t_k\}) \geq \frac{1}{2+\xi}.
\]

\subsection{Proof for the dimension of $E_{L}(\{n_k\},\{s_k\}, \{t_k\})$}
We will prove that
    \begin{equation}\label{gammaplus1}
  \dim_{\rm H} E_{L}(\{n_k\},\{s_k\}, \{t_k\})=\frac{1}{\gamma+1},
  \end{equation}
  where $\gamma \in [1,\infty]$ is defined by
\[
\gamma:=\limsup_{k\to \infty}\exp\left(\frac{\log\log s_k}{n_k}\right).
\]

For the upper bound of \eqref{gammaplus1}, there are three cases to discuss:
\begin{itemize}
  \item when $\gamma=1$, by the definitions of $\alpha$ in \eqref{important} and $X(b,c)$ in subsection 3.1, we deduce that $E_{L}(\{n_k\},\{s_k\}, \{t_k\})$ is a subset of $X(\infty,0)$. So it follows from Theorem \ref{Xbc} that
\[
\dim_{\rm H} E_{L}(\{n_k\},\{s_k\}, \{t_k\}) \leq\frac{1}{2} = \frac{1}{\gamma+1};
\]
  \item when $\gamma \in (1,\infty)$, according to the definition of $\gamma$, for any $\varepsilon<\gamma-1$, we see that $s_k >\exp((\gamma-\varepsilon)^{n_k})$ for infinitely many $k \in \mathbb N$. This implies that $E_{L}(\{n_k\},\{s_k\}, \{t_k\})$ is a subset of
$A(\psi)$ with $\psi(n)=\exp((\gamma-\varepsilon)^{n})$. In view of Theorem \ref{HDApsi}, we derive that
\[
\dim_{\rm H} E_{L}(\{n_k\},\{s_k\}, \{t_k\}) \leq\frac{1}{\gamma-\varepsilon +1}.
\]
Letting $\varepsilon \to 0$, we have $\dim_{\rm H} E_{L}(\{n_k\},\{s_k\}, \{t_k\}) \leq 1/(\gamma+1)$;
  \item  when $\gamma=\infty$, for any $G>1$,
$s_k >\exp(G^{n_k})$ for infinitely many $k \in \mathbb N$. Then $E_{L}(\{n_k\},\{s_k\}, \{t_k\})$ is a subset of
$A(\psi)$ with $\psi(n)=\exp(G^{n})$, and so
\[
\dim_{\rm H} E_{L}(\{n_k\},\{s_k\}, \{t_k\}) \leq\frac{1}{G+1}.
\]
by Theorem \ref{HDApsi}. Letting $G \to \infty$, we get that $\dim_{\rm H} E_{L}(\{n_k\},\{s_k\}, \{t_k\})=0$.
\end{itemize}
Therefore, in all these three cases, we have
\[
  \dim_{\rm H} E_{L}(\{n_k\},\{s_k\}, \{t_k\}) \leq\frac{1}{\gamma+1}.
\]

\medskip

For the lower bound of \eqref{gammaplus1}, without loss of generality, we can assume that $1\leq \gamma <\infty$. For any $\varepsilon>0$, we define the sequences $\{\widetilde{s}(n)\}$ and $\{\widetilde{t}(n)\}$ of positive numbers inductively as follows.
\begin{itemize}
\item Let $\widetilde{s}(1):=2$ and $\widetilde{t}(1):=2$ when $n_1 \neq 1$.
  \item When $n=n_k$ for some $k \in \mathbb N$, let
  \[
  \widetilde{s}(n):=s_k \ \ \ \text{and} \ \ \ \widetilde{t}(n):=t_k.
  \]
  \item When $n\neq n_k$ for all $k \in \mathbb N$, let
  \[
  \widetilde{s}(n):=\left(\prod_{j<n}\widetilde{s}(j) \right)^{\gamma+\varepsilon-1}\ \ \text{and} \ \ \ \widetilde{t}(n):=\left(\prod_{j<n}\widetilde{t}(j) \right)^{\gamma+\varepsilon-1}.
  \]
  In this case,
  \begin{equation}\label{gammaepsilon}
  \sum_{j\leq n}\log\widetilde{s}(j) = (\gamma+\varepsilon)\sum_{j< n}\log\widetilde{s}(j).
  \end{equation}
  Let $k(n):=\sharp\{k\geq 1: n_k \leq n\}$. By (H1), we deduce that $k(n)/n \to 0$ as $n \to \infty$. Combining this with \eqref{gammaepsilon}, we see that
   \begin{equation}\label{stilde}
  \sum_{j\leq n}\log\widetilde{s}(j) \geq (\gamma+\varepsilon)^{n-k(n)} \geq \left(\gamma+\frac{\varepsilon}{2}\right)^{n}
  \end{equation}
  for sufficiently large $n \in \mathbb N$.
\end{itemize}
Then $\widetilde{s}(n),\widetilde{t}(n) \to \infty$ as $n \to \infty$.
Let $F(\{\widetilde{s}(n)\}, \{\widetilde{t}(n)\})$ be the resulting set defined in \eqref{Euv}. Then it is a subset of $E_{L}(\{n_k\},\{s_k\}, \{t_k\})$.
For any $n\in \mathbb N$, let
\[
\widetilde{\mathcal{C}}_n:=\left\{(\sigma_1,\sigma_2,\dots,\sigma_n)\in \mathbb N^n: \widetilde{s}(m)<\sigma_m\leq \widetilde{s}(m)+\widetilde{t}(m), \forall m=1,2,\dots,n\right\}.
\]
For any $n \in \mathbb N$ and $(\sigma_1,\sigma_2,\dots,\sigma_n) \in \widetilde{\mathcal{C}}_n$, define the basic interval of order $n$ as
\[
\widetilde{J}_n(\sigma_1,\sigma_2,\dots,\sigma_n):= \bigcup^\infty_{j=2}cl\left(I_{n+1}(\sigma_1,\sigma_2,\dots,\sigma_n,j)\right),
\]
and let
\[
\widetilde{\mathrm{E}}_n:= \bigcup_{(\sigma_1,\sigma_2,\dots,\sigma_n) \in \widetilde{\mathcal{C}}_n}\widetilde{J}_n(\sigma_1,\sigma_2,\dots,\sigma_n).
\]
Then $\{\widetilde{\mathrm{E}}_n\}$ is a decreasing sequence of subsets of $\widetilde{\mathrm{E}}_0:=[0,1]$. Moreover,
\begin{equation*}
\widetilde{\mathrm{E}}:= \bigcap^\infty_{n=0}\widetilde{\mathrm{E}}_n
\end{equation*}
is the set $F(\{\widetilde{s}(n)\}, \{\widetilde{t}(n)\})$. Hence $\widetilde{\mathrm{E}}$ is a subset of $E_{L}(\{n_k\},\{s_k\}, \{t_k\})$, and so
\begin{equation*}
\dim_{\rm H} E_{L}(\{n_k\},\{s_k\}, \{t_k\}) \geq \dim_{\rm H} \widetilde{\mathrm{E}}.
\end{equation*}

We will use Lemma \ref{lower} to prove the lower bound of $\dim_{\rm H} \widetilde{\mathrm{E}}$. To this end, we need to estimate the values of $m_n$ and $\varepsilon_n$ given in Lemma \ref{lower}. By the definition of $\widetilde{\mathrm{E}}_n$, we see that each basic interval of order $n-1$  contains
\begin{equation}\label{stildemn}
m_n= \lfloor\widetilde{s}(m)+\widetilde{t}(m)\rfloor - \lfloor\widetilde{s}(m) \rfloor \geq \lfloor\widetilde{t}(m) \rfloor
\end{equation}
basic intervals of order $n$. Moreover, these basic intervals of order $n$ are separated by gaps of at least
\begin{equation}\label{stildeen}
\varepsilon_n = \frac{1}{4^{n+2}(\widetilde{s}(1)+\widetilde{t}(1))^2\cdots(\widetilde{s}(n)+\widetilde{t}(n))^2}.
\end{equation}

Now we are ready to calculate the lininf
\begin{equation}\label{stildeliminf}
\liminf_{n \to \infty} \frac{\log (m_1m_2\cdots m_{n-1})}{-\log (m_n\varepsilon_n)}.
\end{equation}
By the definitions of $\widetilde{s}(n)$ and $\widetilde{t}(n)$,
\begin{equation*}\label{sttilde}
\lim_{n \to \infty} \frac{\log\widetilde{s}(n)}{\log\widetilde{t}(n)} = \lim_{n \to \infty} \frac{\log(\widetilde{s}(n)+\widetilde{t}(n))}{\log\widetilde{t}(n)}=1.
\end{equation*}
Combining this with \eqref{stildemn} and \eqref{stildeen}, we derive that the liminf in \eqref{stildeliminf} is equal to
\begin{equation}\label{stildeliminf2}
\liminf_{n \to \infty} \frac{\sum^{n-1}_{j=1}\log\widetilde{s}(j)}{\log\widetilde{s}(n) +2\sum^{n-1}_{j=1}\log\widetilde{s}(j)}.
\end{equation}
When $n=n_k$ for some $k\in \mathbb N$, for any $\varepsilon>0$, according to the definition of $\gamma$,
\begin{equation*}
\log \widetilde{s}(n_k) \leq \left(\gamma+\frac{\varepsilon}{3}\right)^{n_k-1}
\end{equation*}
for sufficiently large $k \in \mathbb N$. Combining this with \eqref{stilde}, we deduce that
\begin{equation*}
\lim_{k\to \infty}\frac{\log \widetilde{s}(n_k)}{\sum_{j\leq n_k-1}\log\widetilde{s}(j)} =0.
\end{equation*}
This implies that the liminf in \eqref{stildeliminf2} along the subsequence $\{n_k\}$ is equal to $1/2$. When $n\neq n_k$ for all $k\in \mathbb N$, it follows from the definition of $\widetilde{s}(n)$ that
\[
\log\widetilde{s}(n) = (\gamma+\varepsilon-1)\sum_{j\leq n-1}\log\widetilde{s}(j).
\]
In this case, the liminf in \eqref{stildeliminf2} is equal to $1/(\gamma+\varepsilon+1)$. Therefore,
\begin{equation*}
\dim_{\rm H} E_{L}(\{n_k\},\{s_k\}, \{t_k\}) \geq \dim_{\rm H} \widetilde{\mathrm{E}} \geq \frac{1}{\gamma+\varepsilon+1}.
\end{equation*}
Letting $\varepsilon \to 0$, we obtain
\[
\dim_{\rm H} E_{L}(\{n_k\},\{s_k\}, \{t_k\}) \geq \frac{1}{\gamma+1}.
\]
Therefore, we complete the proof of \eqref{gammaplus1} on the dimension of $E_{L}(\{n_k\},\{s_k\}, \{t_k\})$.

\section{Proofs of applications}

\subsection{Limsup set of coefficients}
Let $\mathcal{N}\subseteq \mathbb N$ be an infinite set. For any $B>1$, let
 \[
 A(B,\mathcal{N}):=\left\{x\in \mathbb{I}:a_n(x) \geq B^n\;\text{for infinitely many\;$n \in \mathcal{N}$}\right\}.
 \]
Then $A(B)=  A(B,\mathbb N)$, and so Theorem \ref{ABN} is a consequence of the following result.

\begin{theorem}\label{ABhuaN}
Let $\mathcal{N}\subseteq \mathbb N$ be an infinite set. For any $B>1$, we have
\[
\dim_{\rm H} A(B,\mathcal{N}) = \theta(\log B).
\]
 \end{theorem}

\begin{proof}
Note that $A(B,\mathcal{N}) \subseteq A(B)$, so it is sufficient to compute the lower bound of $\dim_{\rm H} A(B,\mathcal{N})$ and the upper bound of $\dim_{\rm H} A(B)$.

For the upper bound of $\dim_{\rm H} A(B)$, it is clear that $A(B)$ is contained in the set $X(0,\log B)$. By Theorem \ref{Xbc}, we see that
\[
\dim_{\rm H} A(B) \leq \dim_{\rm H} X(0,\log B) = \Theta(0,\log B),
\]
where $\Theta(0,\log B)$ is the unique solution of $\mathrm{P}(\theta)=\theta\log B$. Hence $\Theta(0,\log B) = \theta(\log B)$, and so $\dim_{\rm H} A(B) \leq \theta(\log B)$.

For the lower bound of $\dim_{\rm H} A(B, \mathcal{N})$, choose a sequence $\{n_k\}$ from $\mathcal{N}$ such that $n_{k+1} >k(n_1+ \dots+n_{k})$ for all $k \in \mathbb N$. Let $s_k=t_k:=B^{n_k}$ for all $k \in \mathbb N$. Then
\[
\alpha=\lim_{k \to \infty} \frac{1}{n_k}\sum^k_{j=1}\log s_j = \log B \ \ \text{and}\ \ \beta= \lim_{k \to \infty}\frac{\log s_k}{n_k}=\log B.
\]
In this case, $E(\{n_k\},\{s_k\},\{t_k\})$ is a subset of $A(B, \mathcal{N})$. In view of Theorem \ref{Alpha}, we deduce that $\dim_{\rm H}A(B, \mathcal{N}) \geq \theta(\log B)$.

 Therefore,
 \[
 \dim_{\rm H} A(B,\mathcal{N}) = \theta(\log B).
 \]
\end{proof}

We are now in a position to prove Theorem \ref{HDApsi}. The proof is inspired by the method of Wang and Wu \cite[Theorem 4.2]{WW}. For any $b,c>1$, recall that
\[
A(b,c):=\left\{x\in \mathbb{I}: a_n(x) \geq b^{c^n}\ \text{for infinitely many $n \in \mathbb{N}$}\right\}
\]
and
\[
\widehat{A}(b,c):=\left\{x\in \mathbb{I}: a_n(x) \geq b^{c^n}\ \text{for sufficiently large $n \in \mathbb{N}$}\right\}.
\]
It was proved by {\L}uczak \cite{Luc}, and Feng et al. \cite{FWLT} that
\begin{equation}\label{LFWLTApplication}
\dim_{\rm H}\widehat{A}(b,c)= \dim_{\rm H}A(b,c)=\frac{1}{c+1}.
\end{equation}

\begin{proof}[Proof of Theorem \ref{HDApsi}]

(i) When $B_\psi=1$, for any $\varepsilon>0$, by the definition of $B_\psi$ in \eqref{DefBandb}, we see that $\psi(n) \leq (1+\varepsilon)^n$ for infinitely many $n \in \mathbb N$. Let
\[
\mathcal{N}:=\left\{n\in \mathbb{N}: \psi(n) \leq (1+\varepsilon)^n\right\}.
\]
Then $\mathcal{N} \subseteq \mathbb N$ is an infinite set. Moreover, $A(1+\varepsilon,\mathcal{N})$ is a subset of $A(\psi)$. In view of Theorem \ref{ABhuaN}, we deduce that
\[
\dim_{\rm H}A(\psi) \geq  \dim_{\rm H}A(1+\varepsilon,\mathcal{N}) = \theta(\log(1+\varepsilon)).
\]
Letting $\varepsilon \to 0$, we see that $\theta(\log(1+\varepsilon)) \to 1$, and so $\dim_{\rm H}A(\psi) =1$.

(ii) When $1< B_\psi<\infty$, for any $0<\varepsilon<B_\psi-1$, we derive that $\psi(n) \leq (B_\psi+\varepsilon)^n$ for infinitely many $n \in \mathbb N$, and $\psi(n) \geq (B_\psi-\varepsilon)^n$ for all large $n \in \mathbb N$. Let
\[
\mathcal{N}:=\left\{n\in \mathbb{N}: \psi(n) \leq (B_\psi+\varepsilon)^n\right\}.
\]
Then $A(B_\psi+\varepsilon, \mathcal{N})\subseteq  A(\psi) \subseteq A(B_\psi-\varepsilon, \mathcal{N})$. It follows from Theorem \ref{ABhuaN} that
\[
\theta(\log(B_\psi+\varepsilon)) \leq \dim_{\rm H}A(\psi) \leq \theta(\log(B_\psi-\varepsilon)).
\]
Letting $\varepsilon \to 0$, we obtain $\dim_{\rm H}A(\psi) = \theta(\log B_\psi)$.

(iii) When $B_\psi=\infty$, for any large $G>0$, we have $\psi(n) \geq G^n$ for all large $n \in \mathbb N$. Then $A(\psi)$ is a subset of $A(G)$. By Theorem \ref{ABhuaN}, we obtain $\dim_{\rm H} A(\psi) \leq \theta(\log G)$. Letting $G\to \infty$, we get that $\dim_{\rm H} A(\psi) \leq 1/2$.

(iii-1) When $b_\psi=1$, for any $\varepsilon >0$, we know that $\psi(n) \leq e^{(1+\varepsilon)^n}$ for infinitely many $n\in \mathbb N$. Note that $\widehat{A}(e,1+\varepsilon) \subseteq A(\psi)$, we deduce from \eqref{LFWLTApplication} that
\[
\dim_{\rm H} A(\psi) \geq  \frac{1}{2+\varepsilon}.
\]
Letting $\varepsilon \to 0$, we obtain $\dim_{\rm H} A(\psi) \geq 1/2$. Hence $\dim_{\rm H} A(\psi) = 1/2$.

(iii-2) When $1<b_\psi<\infty$, for any $0<\varepsilon <b_\psi-1$, we derive that $\psi(n) \leq e^{(b_\psi+\varepsilon)^n}$ for infinitely many $n \in \mathbb N$, and $\psi(n) \geq e^{(b_\psi-\varepsilon)^n}$ for all large $n\in \mathbb N$. Then
\[
\widehat{A}(e, b_\psi+\varepsilon) \subseteq A(\psi) \subseteq A(e, b_\psi-\varepsilon),
\]
which together with \eqref{LFWLTApplication} gives that
\[
\frac{1}{b_\psi+\varepsilon +1} \leq  \dim_{\rm H}A(\psi)\leq \frac{1}{b_\psi-\varepsilon +1}.
\]
Letting $\varepsilon \to 0$, we have $\dim_{\rm H} M(\psi)=1/(b_\psi+1)$.

(iii-3) When $b_\psi=\infty$, for any large $G>1$, we have $\psi(n) \geq e^{G^n}$ for all large $n \in \mathbb N$. Then $A(\psi)$ is a subset of $A(e,G)$, and so $\dim_{\rm H} A(\psi) \leq 1/(G+1)$. Letting $G\to \infty$, we see that $\dim_{\rm H} A(\psi)=0$.
\end{proof}

\subsection{Limit behaviours of the sum of coefficients}
Let $\varphi: \mathbb{R}^+ \to \mathbb R^+$ be an increasing function with $\varphi(n)/n \to \infty$ as $n \to \infty$.

\begin{proof}[Proof of Theorem \ref{AM}]
For any $n \in \mathbb N$, let $\phi(n):=\log\varphi(n)$. Then $\phi:\mathbb{N}\to\mathbb{R}^+$ is an increasing function and $\phi(n)/\sqrt{n} \to 0$ as $n \to \infty$. Define $t: \mathbb{N}\to\mathbb{R}^+$ by
\[
t(k):= \min\left\{\frac{\sqrt{n}}{\phi(n)}: n\geq k\right\},\ \ \forall k \in \mathbb N.
\]
Then
\begin{itemize}
  \item [(i)] $t(k)$ is increasing to infinity as $k$ tends to infinity;
  \item [(ii)] $\phi(k)\leq \frac{\sqrt{k}}{t(k)}$ for all $k \in \mathbb N$;
  \item [(iii)] $\frac{t(k)}{\sqrt{k}}$ is decreasing to zero as $k$ tends to infinity.
\end{itemize}
The first two facts come directly from the definition of $t(k)$. To see the last one, note that $t(k)=\min\{\sqrt{k}/\varphi(k), t(k+1)\}$ and $\phi$ is increasing, we deduce from (ii) that
\begin{align*}
\frac{\frac{t(k)}{\sqrt{k}}}{\frac{t(k+1)}{\sqrt{k+1}}} &= \frac{\sqrt{k+1}}{\sqrt{k}}\cdot\frac{t(k)}{t(k+1)}\\
&=\left\{
   \begin{array}{ll}
      \frac{\sqrt{k+1}}{\sqrt{k}} >1, & \hbox{if $t(k+1) \leq \frac{\sqrt{k}}{\phi(k)}$;} \\
     \frac{\sqrt{k+1}}{t(k+1)}\cdot \frac{1}{\phi(k)} \geq \frac{\phi(k+1)}{\phi(k)} \geq 1, & \hbox{if $t(k+1) > \frac{\sqrt{k}}{\phi(k)}$,}
   \end{array}
 \right.
\end{align*}
which means that $\frac{t(k)}{\sqrt{k}}$ is decreasing. Using (ii) again, we have $\frac{t(k)}{\sqrt{k}} \to 0$ as $k \to \infty$. For any $k \in \mathbb{N}$, let
\[
n_k:=\min\left\{n \in \mathbb{N}: \phi(n) \geq \frac{k}{t(k)}\right\}.
\]
Then $\{n_k\}$ is increasing. By (ii), for any $k>1$, we deduce that $\phi(n_k) \geq \frac{k}{t(k)} >\frac{\sqrt{k}}{t(k)} \geq \phi(k)$, and so $n_k \geq k$.
Combining this with (i), (ii) and the definition of $n_k$, we conclude that
\begin{equation}\label{nkgeqk}
\sqrt{n_k}\geq \phi(n_k)\cdot t(n_k) \geq \frac{k}{t(k)}\cdot  t(n_k) \geq k,\ \ \forall k>1.
\end{equation}
Hence $n_k \geq k^2$ for all $k>1$, and $n_k/k \to \infty$ as $k \to \infty$. Moreover, we assert that
\begin{equation}\label{sumphi}
\lim_{k \to \infty} \frac{1}{n_k} \sum^k_{j=1}\phi(n_j) =0.
\end{equation}
In fact, since $\phi$ is increasing, we have $\sum^k_{j=1}\phi(n_j) \leq k \phi(n_k) \leq \sqrt{n_k}\cdot\phi(n_k)$ for all $k>1$, and so $n^{-1}_k\sum^k_{j=1}\phi(n_j) \leq \phi(n_k)/\sqrt{n_k}\to 0$ as $k \to \infty$.

For any $k \in \mathbb N$, let $s_k:= \varphi(n_k)-\varphi(n_{k-1})+k$ and $t_k:= s_k/\log(k+1)$. Then $s_k,t_k \to \infty$ as $k \to \infty$, and $\log s_k/\log t_k \to 1$ as $k\to \infty$. It follows from \eqref{nkgeqk} and \eqref{sumphi} that
\begin{equation}\label{uknk}
\alpha:=\lim_{k \to \infty} \frac{1}{n_k} \sum^k_{j=1}\log s_j \leq \lim_{k \to \infty} \frac{1}{n_k} \sum^k_{j=1}\phi(n_j) +\lim_{k \to \infty} \frac{1}{k^2} \sum^k_{j=1}\log j=0.
\end{equation}
Moreover, we claim that $E(\{n_k\},\{s_k\}, \{t_k\})$ is a subset of $S(\varphi)$. In fact, for any $x\in E(\{n_k\},\{s_k\}, \{t_k\})$, there exists $M:=M(x)>0$ such that $1\leq a_j(x)\leq M$ for all $j\neq n_k\,(k\in\mathbb N)$ and $s_k<a_{n_k}(x)\leq s_k+t_k$ for all $k\in \mathbb N$. For any $k \in \mathbb N$, we see that
\begin{equation}\label{Snlowupp1}
\varphi(n_k)<\sum^k_{j=1}a_{n_j}(x)<\varphi(n_k)+ (k+1)^2+ \sum^k_{j=1}\frac{\varphi(n_j)-\varphi(n_{j-1})}{\log(j+1)}.
\end{equation}
For any large $n \in \mathbb N$, there exists a unique $k:=k(n) \in \mathbb{N}$ such that $n_k \leq n <n_{k+1}$. Hence
\begin{equation}\label{Snlowupp2}
\sum^k_{j=1}a_{n_j}(x) <S_n(x) <  nM +\sum^k_{j=1}a_{n_j}(x).
\end{equation}
By the definition of $n_k$, we have
\[
 \frac{k}{t(k)} \leq \phi(n_k) \leq \phi(n) < \frac{k+1}{t(k+1)},
\]
which together with (i) yields that
\[
0 \leq \phi(n) - \phi(n_k) <  \frac{k+1}{t(k+1)} - \frac{k}{t(k)} \leq \frac{1}{t(k)}.
\]
This means that $\phi(n) - \phi(n_k) \to 0$ as $n \to \infty$, that is, $\frac{\varphi(n)}{\varphi(n_k)} \to 1$ as $n \to \infty$.
Combining this with \eqref{nkgeqk}, \eqref{Snlowupp1} and \eqref{Snlowupp2}, we deduce that
\[
 \lim_{n \to \infty} \frac{S_n(x)}{\varphi(n)}=1,
\]
that is, $x\in S(\varphi)$. Therefore, $E(\{n_k\},\{s_k\}, \{t_k\}) \subseteq S(\varphi)$. It follows from \eqref{uknk} and Theorem \ref{Alpha} that
\[
\dim_{\rm H}S(\varphi) \geq \dim_{\rm H}=1.
\]
Hence $\dim_{\rm H}S(\varphi)=1$.
\end{proof}

We are going to prove Theorem \ref{one-halfM}.

\begin{proof}[Proof of Theorem \ref{one-halfM}]

(i) Taking $\varepsilon =1$, by \eqref{ed}, since $\varphi$ is increasing, there exists $\delta_1>0$ such that for all large $k \in \mathbb N$,
\[
\log\varphi((k+1)^2) - \log\varphi(k^2) \geq \log\varphi(k^2+k) - \log\varphi(k^2) \geq \delta_1,
\]
which implies that $\log\varphi(k^2) > 2\delta_1k/3$ for all large $k \in \mathbb N$. Hence there exists $N_1>0$ such that for all $n \geq N^2_1$, $\varphi(n) > e^{c\sqrt{n}}$ with $c:=\delta_1/2$.

For a fixed positive integer $M$, it follows from \eqref{ed} that there exist $\delta_2>0$ and $N_2>2M$ such that for all $n \geq N^2_2$,
\begin{equation}\label{Mdelta2}
\log\varphi\left(n+\frac{\sqrt{n}}{M}\right) - \log\varphi(n) \geq \delta_2.
\end{equation}
For all $k \geq N_2$ and for all $1\leq j \leq M$, since
\[
\lceil k^2 +2(j-1)\frac{k}{M}\rceil \geq k^2 +2(j-1)\frac{k}{M}
\]
and
\[
\lceil k^2 +2(j-1)\frac{k}{M}\rceil + \frac{\sqrt{\lceil k^2 +2(j-1)\frac{k}{M}\rceil}}{M} \leq k^2 +2j\frac{k}{M},
\]
we deduce from \eqref{Mdelta2} and the monotonicity of $\varphi$ that
\begin{equation*}
\log\varphi\left(k^2 +2j\frac{k}{M}\right) - \log\varphi\left(k^2 +2(j-1)\frac{k}{M}\right) \geq \delta_2.
\end{equation*}
Note that $e^s\geq 1+s$ for any $s>0$, so we have
\begin{equation}\label{psid2}
\varphi\left(k^2 +2j\frac{k}{M}\right) \geq (1+\delta_2)\cdot \varphi\left(k^2 +2(j-1)\frac{k}{M}\right).
\end{equation}

We claim that $S(\varphi)$ is a subset of $X(cM/5,0)$. To see this, for any $x\in S(\varphi)$, we derive that $\lim_{n \to \infty} S_n(x)/\varphi(n)=1$, and so there exists $N_3:=N_3(x)>0$ such that for all $n\geq N^2_3$,
\[
\frac{4+3\delta_2}{4(1+\delta_2)}<\frac{S_n(x)}{\varphi(n)}<1+\frac{\delta_2}{2}.
\]
Combining this with \eqref{psid2}, for all $k\geq \max\{N_2,N_3\}$ and for all $1\leq j\leq M$, we see that
\begin{align*}
&S_{\lceil k^2 +2j\frac{k}{M}\rceil}(x)- S_{\lfloor k^2 +2(j-1)\frac{k}{M}\rfloor}(x) \\
&> \frac{4+3\delta_2}{4(1+\delta_2)} \varphi\left(k^2 +2j\frac{k}{M}\right)- \left(1+\frac{\delta_2}{2}\right)\varphi\left(k^2 +2(j-1)\frac{k}{M}\right)\\
&\geq \left(\frac{4+3\delta_2}{4(1+\delta_2)} - \frac{2+\delta_2}{2(1+\delta_2)}\right)\varphi\left(k^2 +2j\frac{k}{M}\right)\\
&\geq \frac{\delta_2}{2(1+\delta_2)} \varphi(k^2).
\end{align*}
Choosing $N_4>0$ such that for all $k \geq N_4$,
\[
\frac{M}{3k}\cdot \frac{\delta_2}{2(1+\delta_2)}\cdot e^{ck} > e^{ck/2}.
\]
Let $N:=\max\{N_1,N_2,N_3,N_4\}$. Then for all $k\geq N$ and for all $1\leq j\leq M$, there exists a positive integer $k_j:=k_j(x) \in (2(j-1)k/M, 2jk/M]$ such that
\[
a_{k^2 +k_j}(x)  \geq \frac{M}{3k}\cdot \frac{\delta_2}{2(1+\delta_2)}\cdot\varphi(k^2) > \frac{M}{3k}\cdot \frac{\delta_2}{2(1+\delta_2)}\cdot e^{ck} > e^{ck/2}.
\]

Fix $m \geq 0$. For all $k \geq\widetilde{N}:=\max\{N, 2m/c\}$, we see that
\[
\prod_{k^2\leq i<(k+1)^2, a_i(x)> e^m} a_i(x) \geq \prod^M_{j=1}a_{k^2 +k_j}(x) >e^{cMk/2}.
\]
Therefore, for sufficiently large $k$, we conclude that
\[
\Pi^{(m)}_{(k+1)^2}(x) \geq \prod^k_{j=1} \prod_{j^2\leq i<(j+1)^2, a_i(x)> e^m} a_i(x) > e^{cM(\widetilde{N}+\cdots+k)/2} > e^{cM(k+1)^2/5}.
\]
That is, $x\in X_m(cM/5,0)$. Hence $x\in X(cM/5,0)$, and so $S(\varphi) \subseteq X(cM/5,0)$.
By Theorem \ref{Xbc}, we obtain $\dim_{\rm H} S(\varphi) \leq \Theta(cM/5,0)$. Since $M>0$ is arbitrary, we get that $\dim_{\rm H} S(\varphi) \leq 1/2$.

(ii) For any large $G>0$, in view of \eqref{limsupinfinity}, we deduce that $\varphi(n) \geq 2ne^{Gn}$ for infinitely many $n\in \mathbb N$.
We claim that $S(\varphi)$ is a subset of $X(0,G)$. In fact, for any $x\in S(\varphi)$, we derive that there exists $N:=N(x)>0$ such that $S_n(x) \geq \varphi(n)/2$ for all $n\geq N$. Hence $S_n(x) \geq ne^{Gn}$ for infinitely many $n\in \mathbb N$. This implies that $a_n(x) \geq e^{Gn}> e^{G(n-1)}$ for infinitely many $n\in \mathbb N$. That is, $x\in X(0, G)$. Therefore, $S(\varphi) \subseteq X(0,G)$. It follows from Theorem \ref{Xbc} that $\dim_{\rm H}S(\varphi) \leq \Theta(0, G)$. Letting $G\to \infty$, we have $\dim_{\rm H} S(\varphi) \leq 1/2$.
\end{proof}

Let us give the proofs of Corollaries \ref{onehalpCor1} and \ref{onehalpCor2}.

%\begin{lemma}
%Let $\varphi: \mathbb{R}^+ \to \mathbb R^+$ be an increasing function with $\varphi(n) -\varphi(n-1) \to \infty$ as $n \to \infty$.
%Assume that
%\[
%\limsup_{n \to \infty}\frac{\log \left(\varphi(n+1) -\varphi(n)\right)}{\log \left(\varphi(n) -\varphi(n-1)\right)}\leq 1.
%\]
%Then $\dim_{\rm H}S(\varphi) \geq 1/2$.
%\end{lemma}
%\begin{proof}
%Let $u_1=v_1:=\varphi(1)$ and
%\end{proof}

\begin{proof}[Proof of Corollary \ref{onehalpCor1}]
Let $\varphi(n) = \exp(c\sqrt{n} +r_1(n))$ with $0<c<\infty$, where $r_1:\mathbb{R}^+\to\mathbb{R}^+$ is increasing and $r_1(n)/\sqrt{n}\to 0$ as $n \to \infty$.

On the one hand, for any $\varepsilon>0$, there exists $N>0$ such that for all $n\geq N$,
\[
\sqrt{n+\varepsilon\sqrt{n}} - \sqrt{n}=\frac{\varepsilon\sqrt{n}}{\sqrt{n+\varepsilon\sqrt{n}}+\sqrt{n}} > \frac{\varepsilon}{3},
\]
and so
\begin{align*}
\log\varphi(n+\varepsilon\sqrt{n}) - \varphi(n)&= c\sqrt{n+\varepsilon\sqrt{n}} +r_1(n+\varepsilon\sqrt{n}) - c\sqrt{n} -r_1(n)\\
&\geq c\sqrt{n+\varepsilon\sqrt{n}} - c\sqrt{n}\\
& > \frac{\varepsilon c}{3}.
\end{align*}
That is, \eqref{ed} holds. By Theorem \ref{one-halfM}, we obtain $\dim_{\rm H}S(\varphi) \leq 1/2$.

On the other hand, let $u_1=v_1:=\varphi(1)$ and $u_n:=\varphi(n)-\varphi(n-1)$ and $v_n:=u_n/n$ for all $n \geq 2$. Then $u_n \to \infty$ and $v_n \to \infty$ as $n\to \infty$. Moreover, $F(\{u_n\}, \{v_n\})$ in \eqref{Euv} is a subset of $S(\varphi)$, and so $\dim_{\rm H}S(\varphi) \geq \dim_{\rm H}F(\{u_n\}, \{v_n\})$. Now it remains to compute the Hausdorff dimension of $F(\{u_n\}, \{v_n\})$. Since $r_1(n)/\sqrt{n}\to 0$ as $n \to \infty$, we have for sufficiently large $n \in \mathbb N$, $r_1(n)< \sqrt{n}$, and so
\[
\frac{e^{c\sqrt{n-1}}}{2\sqrt{n}} < u_n < e^{c\sqrt{n} +r_1(n)} <e^{(c+1)\sqrt{n}},
\]
which implies that
\[
\limsup_{n\to \infty} \frac{\log u_{n+1}}{\log u_1+\log u_2+\cdots+\log u_n} \leq \limsup_{n\to \infty} \frac{(c+1)\sqrt{n+1}}{c(1+\sqrt{2}+\cdots+\sqrt{n-1})}=0.
\]
By Lemma \ref{LRlemma}, we deduce that
\[
\dim_{\rm H}S(\varphi) \geq \dim_{\rm H}F(\{u_n\}, \{v_n\}) = \frac{1}{2+\limsup_{n\to \infty} \frac{\log u_{n+1}}{\sum^n_{k=1}\log u_k}} =\frac{1}{2}.
\]
Therefore, $\dim_{\rm H}S(\varphi) =1/2$.
\end{proof}

%\eqref{maxine} is equivalent to
%\begin{equation}\label{equivalentmaxine}
%\lim_{m \to \infty}\max\left\{\left|r_2(k)-r_2(\ell)\right|:m^2\leq k \neq \ell  \leq(m+1)^2\right\}=0.
%\end{equation}

\begin{proof}[Proof of Corollary \ref{onehalpCor2}]
Let $\varphi(n)=\exp(c\sqrt{n}+r_2(n))$ with $0<c<\infty$, where $r_2:\mathbb{R}^+ \to \mathbb{R}$ is a function such that $c\sqrt{n}+r_2(n)$ is increasing and satisfies the hypothesis \eqref{maxine}. We remark that the hypothesis \eqref{maxine} implies that $r_2(n)/\sqrt{n} \to 0$ as $n \to \infty$.

On the one hand, for any $\varepsilon>0$, by \eqref{maxine}, we have $|r_2(n+\varepsilon\sqrt{n})-r_2(n)| \to 0$ as $n\to \infty$, and so
\[
\lim_{n\to \infty} \left(\log\varphi(n+\varepsilon\sqrt{n}) -\log\varphi(n)\right) = \frac{\varepsilon c}{2}.
\]
Hence there exists $\delta>0$ such that the inequality \eqref{ed} holds. By Theorem \ref{one-halfM}, we obtain $\dim_{\rm H}S(\varphi) \leq 1/2$.

On the other hand, let $\widetilde{u}_1=\widetilde{v}_1:=\varphi(1)$ and $\widetilde{u}_n:=\varphi(n)-\varphi(n-1)+n$ and $\widetilde{v}_n:=\widetilde{u}_n/\log n$ for all $n \geq 2$. Then $\widetilde{u}_n \to \infty$ and $\widetilde{v}_n \to \infty$ as $n\to \infty$.
Moreover, $F(\{\widetilde{u}_n\}, \{\widetilde{v}_n\})$ in \eqref{Euv} is a subset of $S(\varphi)$, and so $\dim_{\rm H}S(\varphi) \geq \dim_{\rm H}F(\{\widetilde{u}_n\}, \{\widetilde{v}_n\})$. Note that $r_2(n)/\sqrt{n} \to 0$ as $n\to \infty$ and $n\leq\widetilde{u}_n<e^{c\sqrt{n}+r_2(n)}+n$ for all $n\in \mathbb N$, we deduce that
\[
\limsup_{n\to \infty} \frac{\log u_{n+1}}{\log u_1+\log u_2+\cdots+\log u_n}  =0.
\]
Combining this with Lemma \ref{LRlemma}, we conclude that
\[
\dim_{\rm H}S(\varphi) \geq \dim_{\rm H}F(\{\widetilde{u}_n\}, \{\widetilde{v}_n\}) = \frac{1}{2+\limsup_{n\to \infty} \frac{\log u_{n+1}}{\sum^n_{k=1}\log u_k}} =\frac{1}{2}.
\]
Therefore, $\dim_{\rm H}S(\varphi) =1/2$.
\end{proof}

We end this subsection with the proof of Theorem \ref{CM}.

\begin{proof}[Proof of Theorem \ref{CM}]
(i) Let $\varphi(n)=\exp(c\cdot\frac{(\lfloor dn^{1-r}\rfloor)^{r/(1-r)}}{d^{r/(1-r)}})$ with $c,d\in (0,\infty)$ and $r\in [1/2,1)$.
For any $k \in \mathbb N$, let
\[
n_k:= \left\lceil\left(\frac{k}{d}\right)^{1/(1-r)}\right\rceil \ \ \ \text{and} \ \ \ d_k:=\left(\frac{k}{d}\right)^{r/(1-r)}.
\]
If $n_k \leq n<n_{k+1}$ for some $k\in \mathbb N$, then $\lfloor dn^{1-r}\rfloor=k$, and so $\varphi(n)=e^{cd_k}$. For any $k \in \mathbb N$, let
$s_1=t_1:=e^{cd_1}$,
\[
 s_k:=e^{cd_k}- e^{cd_{k-1}}\ \ \ \text{and}\ \ \ t_k:=s_k/k.
\]
Note that $r/(1-r) \geq 1$, we derive that
\[
\alpha:=\lim_{k \to \infty} \frac{1}{n_k} \sum^k_{j=1} \log s_j=cd(1-r)\ \ \ \text{and} \ \ \ \beta:= \lim_{k \to \infty} \frac{\log s_k}{n_k}=0.
\]

For the lower bound of $\dim_{\rm H}S(\varphi)$, we assert that $E(\{n_k\},\{s_k\}, \{t_k\})$ is a subset of $S(\varphi)$. To see this, for any $x \in E(\{n_k\},\{s_k\}, \{t_k\})$, there exists $M:=M(x)>0$ such that $1\leq a_j(x) \leq M$ for all $j\neq n_k\,(k \in \mathbb N)$ and for all $k\in \mathbb N$,
\[
e^{cd_k}-e^{cd_{k-1}} < a_{n_k}(x) \leq \left(1+\frac{1}{k}\right)\left(e^{cd_k} - e^{cd_{k-1}}\right).
\]
For any $n \in \mathbb N$, there exists $k \in \mathbb N$ such that $n_k \leq n <n_{k+1}$, and so $\varphi(n)= e^{c d_k}$. Moreover,
\[
e^{cd_k} <S_n(x)< nM + e^{cd_k} + \sum^k_{j=1} \frac{1}{j}\left(e^{cd_j} - e^{cd_{j-1}}\right).
\]
This implies that $\lim_{n\to \infty}S_n(x)/\varphi(n)=1$, that is, $x\in S(\varphi)$. By Theorem \ref{Alpha}, we get that
\[
\dim_{\rm H}S(\varphi) \geq \dim_{\rm H} E(\{n_k\},\{s_k\}, \{t_k\}) = \eta_d(c),
\]
where $\eta_d(c)$ is the unique solution of $\mathrm{P}(\theta)= cd(1-r)(2\theta -1)$.

For the upper bound we claim that $S(\varphi)$ is a subset of $E_{L}(\{n_k\},\{\widetilde{s}_k\}, \{\widetilde{t}_k\})$ with $\widetilde{s}_k:=(1-2e^{-c/d})e^{cd_k}$ and $\widetilde{t}_k:=(2e^{-c/d}+1)e^{cd_k}$.
In fact, for any $x\in S(\varphi)$, we deduce that $\lim_{n\to \infty}S_n(x)/\varphi(n)=1$, which yields that
\begin{equation}\label{Snk}
\lim_{n \to \infty} \frac{S_{n_{k+1}-1}(x)}{e^{cd_k}}=1 \ \ \ \text{and}\ \ \ \lim_{n \to \infty} \frac{S_{n_{k+1}}(x)}{e^{cd_{k+1}}}=1.
\end{equation}
When $r=1/2$, note that
  \[
  d_{k+1}-d_k = \frac{k+1}{d} - \frac{k}{d} = \frac{1}{d},
  \]
  combining this with \eqref{Snk}, we deduce that
  \[
\lim_{n \to \infty} \frac{a_{n_{k+1}}(x)}{e^{cd_{k+1}}}=1-e^{-c/d};
\]
When $r\in (1/2,1)$, note that
\[
  d_{k+1}-d_k  = \left(\frac{k+1}{d}\right)^{r/(1-r)}- \left(\frac{k}{d}\right)^{r/(1-r)} \geq \left(\frac{1}{d}\right)^{r/(1-r)} \frac{r}{1-r} k^{\frac{2r-1}{1-r}},
  \]
combining this with \eqref{Snk}, we conclude that
\[
\lim_{n \to \infty} \frac{a_{n_{k+1}}(x)}{e^{cd_{k+1}}}=1.
\]
Therefore, for all large $k \in \mathbb N$,
\[
1-\frac{2} {e^{c/d}}- \leq \frac{a_{n_{k+1}}(x)}{e^{cd_{k+1}}} < 2,
\]
that is, $x\in E_{L}(\{n_k\},\{\widetilde{s}_k\}, \{\widetilde{t}_k\})$. By Lemma \ref{Eularge}, we see that
\[
\dim_{\rm H}S(\varphi) \leq \dim_{\rm H} E_{L}(\{n_k\},\{s_k\}, \{t_k\}) \leq\eta_d(c).
\]

(ii) Let $\psi(n)= \exp(c \cdot \exp(\gamma \lfloor \frac{\log n}{\gamma}\rfloor))$ with $c, \gamma\in (0,\infty)$ be fixed. For any $k\in \mathbb N$, let
\[
n_k:= \left\lceil e^{\gamma k}\right\rceil\ \ \ \text{and}\ \ \ b_k:= e^{\gamma k}.
\]
If $n_k \leq n<n_{k+1}$ for some $k\in \mathbb N$, then $\lfloor \frac{\log n}{\gamma}\rfloor=k$, and so $\varphi(n)=e^{cb_k}$. For any $k\in \mathbb N$, let $s_1=t_1:=e^{cb_1}$,
\[
s_k:=e^{cb_k} - e^{cb_{k-1}}\ \ \ \text{and} \ \ \ t_k:=s_k/k.
\]
Then
\[
\alpha:=\lim_{k \to \infty} \frac{1}{n_k} \sum^k_{j=1} \log s_j=\frac{e^\gamma}{e^\gamma-1}\ \ \ \text{and} \ \ \ \beta:= \lim_{k \to \infty} \frac{\log s_k}{n_k}=1.
\]

For the lower bound of $\dim_{\rm H}S(\varphi)$, since $E(\{n_k\},\{s_k\}, \{t_k\})$ with is a subset of $S(\varphi)$, we deduce from Theorem \ref{Alpha} that
\[
\dim_{\rm H}S(\varphi) \geq \dim_{\rm H} E(\{n_k\},\{s_k\}, \{t_k\}) = \xi_\gamma(c),
\]
where $\xi_\gamma(c)$ is the unique solution of
\[
\mathrm{P}(\theta)= c\left(\frac{e^\gamma+1}{e^\gamma-1}\theta- \frac{1}{e^\gamma-1}\right).
\]

For the upper bound, we claim that $S(\varphi)$ is a subset of $E_{L}(\{n_k\},\{\widetilde{s}_k\}, \{\widetilde{t}_k\})$ with $\widetilde{s}_k:=e^{cb_k}/2$ and $\widetilde{t}_k:=e^{cb_k}$. To see this, for any $x\in S(\varphi)$, we have
\begin{equation*}
\lim_{n \to \infty} \frac{S_{n_{k+1}-1}(x)}{e^{cb_k}}=1 \ \ \ \text{and}\ \ \ \lim_{n \to \infty} \frac{S_{n_{k+1}}(x)}{e^{cb_{k+1}}}=1,
\end{equation*}
and so
\[
\lim_{n \to \infty} \frac{a_{n_{k+1}}(x)}{e^{cb_{k+1}}}=1.
\]
This implies that $e^{cb_{k+1}}/2 < a_{n_{k+1}}(x) \leq 3e^{cb_{k+1}}/2$ for all large $k\in \mathbb N$. That is, $x\in E_{L}(\{n_k\},\{\widetilde{s}_k\}, \{\widetilde{t}_k\})$. It follows from Lemma \ref{Eularge} that
\[
\dim_{\rm H}S(\varphi)\leq \dim_{\rm H} E_{L}(\{n_k\},\{s_k\}, \{t_k\}) \leq \xi_\gamma(c).
\]
\end{proof}

\subsection{Limsup and liminf sets of the maximum of coefficients}

Let $\mathcal{N}\subseteq \mathbb N$ be an infinite set. For any $B>1$, let
 \[
M(B,\mathcal{N}):=\left\{x\in \mathbb{I}:M_n(x) \geq B^n\;\text{for infinitely many\;$n \in \mathcal{N}$}\right\}
 \]
and $M(B):=M(B,\mathbb N)$. Then $A(B,\mathcal{N}) \subseteq M(B,\mathcal{N}) \subseteq M(B) =A(B)$. In the light of Theorem \ref{ABhuaN}, we deduce that
\begin{equation}\label{MBhuaN}
\dim_{\rm H}M(B,\mathcal{N})=\theta(\log B).
\end{equation}

\begin{lemma}\label{Mbc}
For any $b,c>1$, let
\[
M(b,c):=\left\{x\in \mathbb{I}: M_n(x) \geq b^{c^n}\ \text{for infinitely many $n \in \mathbb{N}$}\right\}
\]
and
\[
\widehat{M}(b,c):=\left\{x\in \mathbb{I}: M_n(x) \geq b^{c^n}\ \text{for sufficiently large $n \in \mathbb{N}$}\right\}.
\]
Then
\[
\dim_{\rm H} M(b,c) = \dim_{\rm H} \widehat{M}(b,c) =\frac{1}{c+1}.
\]
\end{lemma}
\begin{proof}
Since $\widehat{A}(b,c) \subseteq \widehat{M}(b,c) \subseteq M(b,c) = A(b,c)$, the proof is a consequence of the results of {\L}uczak \cite{Luc} and Feng et al. \cite{FWLT}.
\end{proof}

By \eqref{MBhuaN} and Lemma \ref{Mbc}, we can prove Theorem \ref{Mpsi} by following the proof of Theorem \ref{HDApsi} step by step.
The details are left to the reader.

Now we are going to compute the Hausdorff dimension of the liminf set $\widehat{M}(\psi)$. For any $C>1$, recall that
\[
\widehat{M}(C):=\left\{x\in \mathbb{I}: M_n(x) \geq C^{n}\ \text{for sufficiently large $n\in \mathbb N$}\right\}
\]
and $\widehat{\theta}(\log C)$ is the unique solution of $\mathrm{P}(\theta)=(\sqrt{\theta}+\sqrt{2\theta-1})^2\log C$.
Before proving Theorem \ref{BhatMN}, we first give a lemma.

\begin{lemma}\label{comnum}
Let $\kappa, \rho \in \mathbb{R}^+$ and $\{x_k\}$ be a sequence of positive real numbers. If
\begin{equation}\label{comnumin}
x_{k} \geq \kappa x_{k-1} +\rho x_{k+1}
\end{equation}
for all large $k$, then $4\kappa\rho \leq 1$.
\end{lemma}

\begin{proof}
We prove it by contradiction. Suppose that $4\kappa\rho > 1$. Then $2\sqrt{\kappa/\rho} >1/\rho$. Choose a small $\varepsilon>0$ such that
\[
2\sqrt{\frac{\kappa}{\rho}}> \frac{1}{\rho}+\varepsilon.
\]
Let $y_k:=x_{k+1}/x_k$ for all $k\in \mathbb N$. Then
\[
y_{k-1} + \frac{\kappa}{\rho y_{k-1}} \geq 2\sqrt{\frac{\kappa}{\rho}}> \frac{1}{\rho}+\varepsilon.
\]
Combining this with \eqref{comnumin}, we deduce that for all large $k$,
\[
y_k \leq \frac{1}{\rho} - \frac{\kappa}{\rho y_{k-1}} <y_{k-1} -\varepsilon,
\]
which gives that $y_k$ is eventually negative. This is a contradiction.
\end{proof}

Now we are able to give the proof of Theorem \ref{BhatMN}.

\begin{proof}[Proof of Theorem \ref{BhatMN}]
For any $C>1$, let $\widehat{\theta}:=\widehat{\theta}(\log C)$. For the lower bound of $\dim_{\rm H}\widehat{M}(C)$, let $\gamma\in (0,\infty)$ be defined as
\begin{equation}\label{egammaFc}
e^{\gamma}:=1+\sqrt{\frac{2\widehat{\theta}-1}{\widehat{\theta}}}.
\end{equation}
For any $k \in \mathbb N$, let
\[
n_k:= \left\lceil e^{\gamma k}\right\rceil \ \ \ \text{and}\ \ \ u_k:= e^{\gamma (k+1)}.
\]
Then
\[
\lim_{k \to \infty} \frac{1}{n_k} \sum^k_{j=1} u_j= \frac{e^{2\gamma}}{e^\gamma-1}\ \ \ \text{and} \ \ \ \lim_{k \to \infty} \frac{u_k}{n_k}=e^\gamma.
\]
Since $E(\{n_k\},\{s_k\}, \{t_k\})$ with $s_k=t_k=C^{u_k}$ is a subset of $\widehat{M}(C)$, we deduce from the first statement of Theorem \ref{Alpha} that
\[
\dim_{\rm H} \widehat{F}(c) \geq \dim_{\rm H}  E(\{n_k\},\{s_k\}, \{t_k\}) = \widetilde{\theta},
\]
where $\widetilde{\theta}$ is the unique solution of
\begin{equation}\label{lowerEc}
\mathrm{P}(\theta)= e^\gamma\left(\frac{e^\gamma+1}{e^\gamma-1}\theta- \frac{1}{e^\gamma-1}\right)\log C.
\end{equation}
We assert that $\widehat{\theta}$ is the solution of the equation \eqref{lowerEc}. In fact, we conclude from \eqref{egammaFc} that
\begin{align*}
&e^\gamma\left(\frac{e^\gamma+1}{e^\gamma-1}\widehat{\theta}- \frac{1}{e^\gamma-1}\right)\log C \\
&= \sqrt{\frac{\widehat{\theta}}{2\widehat{\theta}-1}}\left(1+\sqrt{\frac{2\widehat{\theta}-1}{\widehat{\theta}}}\right)\left( \left(2+\sqrt{\frac{2\widehat{\theta}-1}{\widehat{\theta}}}\right)\widehat{\theta}-1\right)\log C\\
&=\left(\sqrt{\widehat{\theta}}+\sqrt{2\widehat{\theta}-1}\right)^2\log C =\mathrm{P}(\widehat{\theta}).
\end{align*}
Then $\widetilde{\theta}= \widehat{\theta}$, and so $\dim_{\rm H} \widehat{M}(C) \geq \widehat{\theta}$.

For the upper bound od $\dim_{\rm H} \widehat{M}(C)$, let $\theta^*:=\dim_{\rm H} \widehat{M}(C)$. Then $\theta^* \geq \widehat{\theta} >1/2$. We will prove that $\theta^* \leq \widehat{\theta}$. For any $x\in \widehat{M}(C)$ and $k\geq 1$, defined $n_1:\equiv1$ and
\[
n_{k+1}(x):= \min\left\{n >n_k(x): a_n(x) > a_{n_k(x)}(x)\right\},
\]
and $s_k(x):= \sum^k_{j=1}n_j(x)$. For typographical convenience, we often write $n_k$ and $s_k$ instead of $n_k(x)$ and $s_k(x)$ if no confusion arises. Then for all $k \in \mathbb N$,
\begin{equation}\label{nkan}
 a_{n_k}(x) = M_{n_{k+1}-1}(x) \geq e^{c(n_{k+1}-1)}.
\end{equation}
Fix a large positive integer $N$. For $\lambda, \mu \in \mathbb{Q}^+$, let
\begin{align*}
Y_{N,\lambda,\mu}(\theta^*,C):= \Big\{x\in \widehat{M}(C):\ &\text{for infinitely many $k$}, \frac{s_k(x)-k}{n_k(x)} >\lambda,\frac{n_{k+1}(x)-1}{n_k(x)} >\mu \\ n_k(x)\mathrm{P}(\theta^*)&\leq \left(1-\frac{1}{N}\right)\big((2\theta^*-1)s_k(x) +\theta^* n_{k+1}(x) \big)\log C\Big\}.
\end{align*}
Then $Y_{N,\lambda,\mu}(\theta^*,C)$ is a subset of $X(\lambda\log C,\mu\log C)$. To see this, for any $x\in Y_{N,\lambda,\mu}(\theta^*,C)$, it follows from \eqref{nkan} that for any large $k$,
\[
a_{n_k}(x) \geq C^{n_{k+1}-1} > C^{\mu n_k} >C^{\mu (n_k-1)}.
\]
Moreover, for any $m \geq 0$ and for any large $k$,
\[
\Pi^{(m)}_{n_k-1}(x) \geq C^{s_k-k} >C^{\lambda n_k}> C^{\lambda (n_k-1)}.
\]
That is, $x\in X(\lambda\log C,\mu\log C)$. So
\[
\dim_{\rm H}Y_{N,\lambda,\mu}(\theta^*,C) \leq \dim_{\rm H}X(\lambda\log C,\mu\log C) = \Theta(\lambda\log C,\mu\log C),
\]
where $\Theta(\lambda\log C,\mu\log C)$ is the unique solution of the Diophantine pressure equation
\[
\mathrm{P}(\theta)=\lambda(2\theta-1)\log C+ \mu\theta\log C.
\]

Let us estimate the Hausdorff dimension of the set
\begin{align*}
Y_{N}(\theta^*,C):= \Big\{x\in \widehat{M}(C):\, \text{for infinitely many $k$,}\\
n_k(x)\mathrm{P}(\theta^*)\leq\left(1-\frac{1}{N}\right)\big((2\theta^*-1)s_k(x)+& \theta^* n_{k+1}(x) \big)c\log C
\Big\}.
\end{align*}
We claim that
\begin{equation}\label{XXYNc}
Y_{N}(\theta^*,C) \subseteq X(0,\infty) \cup X(\infty,0) \cup \left(\bigcup_{(\lambda,\mu) \in \mathcal{Q}} Y_{N,\lambda,\mu}(\theta^*,C)\right),
\end{equation}
where $\mathcal{Q}$ is an at most countable set defined by
\[
\mathcal{Q}:= \left\{(\lambda,\mu) \in \mathbb{Q}^+\times\mathbb{Q}^+: \mathrm{P}(\theta^*)< \left(1-\frac{1}{2N}\right) \big((2\theta^*-1)\lambda+\theta^*\mu\big)\log C\right\}.
\]
In fact, for any $x\in Y_{N}(\theta^*,C)$, we have
\begin{enumerate}
  \item[(i)] if $\{\frac{n_{k+1}}{n_k}\}$ is unbounded, then for any large $G>0$, $n_{k+1}-1 \geq G(n_k-1)$ for infinitely many $k$. By \eqref{nkan}, we see that $a_{n_k}(x)\geq C^{n_{k+1}-1} \geq C^{G(n_{k}-1)}$ for infinitely many $k$. This means that $x\in X(0,G\log C)$ for any large $G>0$. So $x\in X(0,\infty)$;

 \item[(ii)] if $\{\frac{s_{k}}{n_k}\}$ is unbounded, then for any large $H>0$, $s_{k}-k \geq H(n_k-1)$ for infinitely many $k$. For any $m \geq 0$, it follows form \eqref{nkan} that $\Pi^{(m)}_{n_k-1}(x) \geq C^{s_k-k} \geq C^{H (n_k-1)}$ for infinitely many $k$. So $x\in X(\infty,0)$;

 \item[(iii)] if $\{\frac{n_{k+1}}{n_k}\}$ and $\{\frac{s_{k}}{n_k}\}$ are bounded, then there exist $\{k_i\}$ and $\widetilde{\lambda}, \widetilde{\mu} \in [1,\infty)$ such that
  \[
  \lim_{i \to \infty} \frac{s_{k_i}}{n_{k_i}} = \widetilde{\lambda}\ \ \ \text{and}\ \ \ \lim_{i \to \infty} \frac{n_{k_i+1}}{n_{k_i}}=\widetilde{\mu}.
  \]
  Hence there exist $\lambda, \mu \in \mathbb{Q}^+$ such that
  \[
  \lambda < \widetilde{\lambda} < \left(1+\frac{1}{2N}\right)\lambda\ \ \text{and}\ \  \mu < \widetilde{\mu} < \left(1+\frac{1}{2N}\right)\mu,
  \]
  which implies that
\begin{align}\label{thetahatetheta}
\mathrm{P}(\theta^*)&\leq \left(1-\frac{1}{N}\right)\left((2\theta^*-1)\widetilde{\lambda}+\theta^*\widetilde{\mu}\right)\log C \notag\\
&< \left(1-\frac{1}{2N}\right)\big((2\theta^*-1)\lambda+\theta^*\mu\big)\log C.
\end{align}
So $x\in \bigcup_{(\lambda,\mu) \in \mathcal{Q}} Y_{N,\lambda,\mu}(\theta^*,C)$.
\end{enumerate}
Since $\dim_{\rm H}Y_{N,\lambda,\mu}(\theta^*,C)=\Theta(\lambda\log C,\mu\log C) >\frac{1}{2}$ for all $(\lambda,\mu) \in \mathcal{Q}$, $\dim_{\rm H} X(0,\infty)= \dim_{\rm H} X(\infty,0)= \frac{1}{2}$, it follows from \eqref{XXYNc} that
 \[
 \dim_{\rm H} Y_{N}(\theta^*,C) \leq \sup_{(\lambda,\mu) \in \mathcal{Q}}\big\{\Theta(\lambda\log C,\mu\log C)\big\}.
 \]

In the following, we are going to show that $\sup_{(\lambda,\mu) \in \mathcal{Q}}\{\Theta(\lambda\log C,\mu\log C)\} <\theta^*$. For any $(\lambda,\mu) \in \mathcal{Q}$, let $\Theta := \Theta(\lambda\log C,\mu\log C)$. Then
\begin{equation}\label{thetatilde}
\mathrm{P}(\Theta)=\lambda(2\Theta-1)\log C+\mu\Theta\log C.
\end{equation}
Note that
\begin{enumerate}
  \item[(i)] if
  \[
  \theta^* \geq \left(1+\frac{1}{4N}\right) \Theta -\frac{1}{8N},
  \]
  then
  \[
 \Theta\leq \frac{8N\theta^*+1}{8N+2};
  \]

 \item[(ii)] if
  \[
  \theta^* < \left(1+\frac{1}{4N}\right)\Theta-\frac{1}{8N},
  \]
  then
  \[
  2\theta^* -1 < \left(1+\frac{1}{4N}\right) \left(2\Theta-1\right).
  \]
  Since $(\lambda,\mu)$ satisfies the inequality \eqref{thetahatetheta}, we see that
  \begin{align*}
\mathrm{P}(\theta^*)&< \left(1-\frac{1}{2N}\right) \big((2\theta^*-1)\lambda+\theta^*\mu\big)\log C\\
&<\left(1-\frac{1}{2N}\right)\left(1+\frac{1}{4N}\right) \left((2\Theta-1)\lambda+\Theta\mu\right)\log C\\
&< \left(1-\frac{1}{4N}\right)\mathrm{P}(\Theta),
\end{align*}
where the last inequality follows from \eqref{thetatilde}. Since the Diophantine pressure function $\mathrm{P}(\cdot)$ is strictly decreasing and analytic on $(1/2,\infty)$, we have
\[
\Theta < \mathrm{P}^{-1}\left(\frac{4N}{4N-1} \mathrm{P}(\theta^*)\right).
\]
\end{enumerate}
Therefore, we conclude that
\[
\Theta(\lambda\log C,\mu\log C) =:\Theta <\max\left\{\frac{8N\theta^*+1}{8N+2},\ \mathrm{P}^{-1}\left(\frac{4N}{4N-1} \mathrm{P}(\theta^*)\right)\right\}.
\]
Taking the supremum for all $(\lambda,\mu) \in \mathcal{Q}$, we get that
\[
\sup_{(\lambda,\mu) \in \mathcal{Q}}\left\{\Theta(\lambda\log C,\mu\log C)\right\} \leq  \max\left\{\frac{8N\theta^*+1}{8N+2},\ \mathrm{P}^{-1}\left(\frac{4N}{4N-1}\mathrm{P}(\theta^*)\right)\right\}.
\]
Since $\theta^* >1/2$ and the Diophantine pressure function $\mathrm{P}(\cdot)$ is strictly decreasing, the right-hand side of the above inequality is strictly less than $\theta^*$. Therefore,
\[
 \dim_{\rm H} Y_{N}(\theta^*,C) \leq \sup_{(\lambda,\mu) \in \mathcal{Q}} \left\{\Theta(\lambda\log C,\mu\log C)\right\}  <\theta^*:=\dim_{\rm H} \widehat{M}(C).
 \]
This implies that the complement of $Y_{N}(\theta,c)$ has positive Hausdorff dimension, and so there exists $x\in \widehat{M}(C)$ such that for all large $k$,
\begin{equation}\label{compP}
n_k\mathrm{P}(\theta^*)> \left(1-\frac{1}{N}\right)\big((2\theta^*-1)s_k +\theta^* n_{k+1} \big)\log C.
\end{equation}
Since $s_k >n_k$ and $n_{k+1}>n_k$, we have
\begin{equation}\label{compP2}
\mathrm{P}(\theta)> \left(1-\frac{1}{N}\right) \big(3\theta^*-1 \big)\log C.
\end{equation}
Note that $n_k=s_k-s_{k-1}$ and $n_{k+1} =s_{k+1} -s_k$, so the inequality \eqref{compP} becomes
\[
s_k > \kappa s_{k-1} +\rho s_{k+1},
\]
where $\kappa, \rho>0$ are given by
\[
\kappa:=\frac{\mathrm{P}(\theta^*)}{\mathrm{P}(\theta^*)+\left(1-\frac{1}{N}\right)(1-\theta^*)\log C}\  \text{and} \ \rho:=\frac{\left(1-\frac{1}{N}\right)\theta^*\log C}{\mathrm{P}(\theta^*)+\left(1-\frac{1}{N}\right)(1-\theta^*)\log C}.
\]
We remark that the denominators of $\kappa$ and $\rho$ are positive because of the inequality \eqref{compP2}. Applying Lemma \ref{comnum}, we deduce that
\[
4\kappa\rho = \frac{4\left(1-\frac{1}{N}\right)\theta^*\mathrm{P}(\theta^*)\log C}{\left(\mathrm{P}(\theta^*)+\left(1-\frac{1}{N}\right)(1-\theta^*)\log C\right)^2} \leq 1.
\]
Since this inequality holds for all large $N$, letting $N\to \infty$ gives that
\[
\left(\mathrm{P}(\theta^*)+(1-\theta^*)\log C\right)^2 \geq 4\theta^*\mathrm{P}(\theta^*)\log C,
\]
which is equivalent to $\left(\mathrm{P}(\theta^*)-(3\theta^*-1)\log C\right)^2 \geq 4(\log C)^2\theta^*(2\theta^*-1)$, namely
\[
|\mathrm{P}(\theta^*) - (3\theta^*-1)\log C| \geq 2\log C\sqrt{\theta^*(2\theta^*-1)}.
\]
It follows from \eqref{compP2} that $\mathrm{P}(\theta^*) \geq (3\theta^*-1)\log C$. Hence
\[
\mathrm{P}(\theta^*) \geq (3\theta^*-1)\log C+ 2\sqrt{\theta^*(2\theta^*-1)}\log C=(\sqrt{\theta^*} +\sqrt{2\theta^*-1})^2\log C.
\]
Since $\widehat{\theta}$ is the unique solution of $\mathrm{P}(\theta)=(\sqrt{\theta} +\sqrt{2\theta-1})^2\log C$, we have $\theta^* \leq \widehat{\theta}$. Therefore, $\dim_{\rm H} \widehat{M}(C) \leq \widehat{\theta}$ as we want.
\end{proof}

We are able to give the proof of Theorem \ref{AMhat}.

\begin{proof}[Proof of Theorem \ref{AMhat}]
Given a function $\psi:\mathbb{R}^{+}\to\mathbb{R}^{+}$, recall that $C_\psi$ and $c_{\psi}$ are constants defined in \eqref{DefC}.
(i) When $C_\psi=1$, it follows from the definition that for any $\varepsilon>0$, we see that $\psi(n) \leq (1+\varepsilon)^n$ for all large $n \in \mathbb N$. This implies that $\widehat{M}( 1+\varepsilon) \subseteq \widehat{M}(\psi)$, and so $\dim_{\rm H} \widehat{M}(\psi) \geq \widehat{\theta}(\log(1+\varepsilon))$. Letting $\varepsilon \to 0$, we have $\dim_{\rm H} \widehat{M}(\psi) =1$.

(ii) When $1<C_\psi<\infty$, for any $\varepsilon>0$, we know that $\psi(n) \leq (C_\psi+\varepsilon)^n$ for all large $n\in \mathbb N$. Then $\widehat{M}(\psi) \supseteq \widehat{M}(C_\psi+\varepsilon)$, and so
\[
\dim_{\rm H}\widehat{M}(\psi) \geq \widehat{\theta}(\log(C_\psi+\varepsilon)).
\]
Letting $\varepsilon \to 0$ yields that $\dim_{\rm H}\widehat{M}(\psi) \geq \widehat{\theta}(\log C_\psi)$.

Assume that the limsup in the definition of $C_{\psi}$ is a limit. For any $0<\varepsilon<C_\psi-1$, we derive that $\psi(n) \geq (C_\psi-\varepsilon)^n$ for all large $n\in \mathbb N$. Hence $\widehat{M}(\psi) \subseteq \widehat{M}(C_\psi-\varepsilon)$, and so
\[
\dim_{\rm H}\widehat{M}(\psi) \leq  \widehat{\theta}(\log(C_\psi-\varepsilon)).
\]
Letting $\varepsilon \to 0$, we have $\dim_{\rm H}\widehat{M}(\psi) \leq \widehat{\theta}(\log C_\psi)$. Therefore, the equation holds.

(iii) When $C_\psi=\infty$, for any $G>0$, we deduce that $\psi(n) \geq G^n$ for infinitely many $n \in \mathbb N$. Hence $\widehat{M}(\psi) \subseteq M(G)$, and so $\dim_{\rm H}\widehat{M}(\psi) \leq \dim_{\rm H}M(G) = \theta(\log G)$. Letting $G\to \infty$, we have  $\dim_{\rm H}\widehat{M}(\psi) \leq 1/2$.

(iii-1) When $c_\psi =1$, for any $\varepsilon>0$, we know that $\psi(n) \leq e^{(1+\varepsilon)^n}$ for all large $n \in \mathbb N$.
Then $\widehat{M}(e, 1+\varepsilon) \subseteq \widehat{M}(\psi)$, and so $\dim_{\rm H}\widehat{M}(\psi) \geq \dim_{\rm H}\widehat{M}(e, 1+\varepsilon)$. It follows from Lemma \ref{Mbc} that $\dim_{\rm H}\widehat{M}(\psi) \geq 1/(2+\varepsilon)$. Letting $\varepsilon \to 0$, we have $\dim_{\rm H}\widehat{M}(\psi) \geq 1/2$. Therefore, $\dim_{\rm H}\widehat{M}(\psi) = 1/2$.

(iii-2) When $1<c_\psi<\infty$, for any $0<\varepsilon <c_\psi-1$, we conclude that $\psi(n) \leq e^{(c_\psi+\varepsilon)^n}$ for all large $n\in \mathbb N$, and $\psi(n) \geq e^{(c_\psi-\varepsilon)^n}$ for infinitely many $n\in \mathbb N$. Hence
\[
\widehat{M}(e, c_\psi+\varepsilon) \subseteq \widehat{M}(\psi) \subseteq M(e, c_\psi-\varepsilon),
\]
and so
\[
\frac{1}{c_\psi+\varepsilon+1} \leq \dim_{\rm H}\widehat{M}(\psi) \leq \frac{1}{c_\psi-\varepsilon+1}.
\]
Letting $\varepsilon \to 0$, we obtain $\dim_{\rm H} \widehat{M}(\psi) = 1/(c_\psi+1)$.

(iii-3) When $c_\psi=\infty$, for any $G>1$, $\psi(n) \geq e^{G^n}$ for infinitely many $n \in \mathbb N$. Then $\widehat{M}(\psi) \subseteq M(e, G)$, and so $\dim_{\rm H} \widehat{M}(\psi) \leq 1/(G+1)$. Letting $G\to \infty$, we derive that $\dim_{\rm H} \widehat{M}(\psi) =0$.
\end{proof}

\end{document}